\documentclass[12pt]{article}
\usepackage[T1]{fontenc}
\usepackage[english]{babel}
\usepackage{geometry}
\usepackage{microtype} 

\usepackage{lmodern}
\usepackage{amsmath, amsthm, amssymb}
\usepackage{dsfont}
\usepackage{bbm}
\usepackage{bm}
\usepackage{accents}
\usepackage{faktor}
\usepackage{empheq}
\usepackage{libertine}
\usepackage{algorithm}
\usepackage{algpseudocode}
\usepackage{mathrsfs}
\usepackage{booktabs}

\renewcommand{\geq}{\geqslant}
\renewcommand{\leq}{\leqslant}

\usepackage[
  backend=biber,
  style=numeric,
  sorting=ynt,
  maxbibnames=99
]{biblatex}
\usepackage{theoremref}

\usepackage{appendix}

\usepackage{tikz}
\usepackage{pgfplots}
\pgfplotsset{compat=newest} 

\newlength\figureheight
\newlength\figurewidth
\usetikzlibrary{patterns, plotmarks}

\numberwithin{figure}{section}

\usepackage{subcaption}
\usepackage{float}
\usepackage{enumitem}
\usepackage{csquotes}
\usepackage[normalem]{ulem}

\theoremstyle{plain}
\newtheorem{theorem}{Theorem}[section]
\newtheorem{definition}[theorem]{Definition}
\newtheorem{lemma}[theorem]{Lemma}
\newtheorem{proposition}[theorem]{Proposition}

\newtheorem{example}[theorem]{Example}

\newtheorem{remark}[theorem]{Remark}
\newtheorem{scheme}[theorem]{Algorithm}

\newtheorem{assumption}[theorem]{Assumption}

\numberwithin{equation}{section}

\allowdisplaybreaks

\newcounter{corr}
\definecolor{violet}{rgb}{0.580,0.,0.827}
\newcommand{\corr}[3]{%
  \typeout{Warning: a correction remains in page \thepage}%
  \stepcounter{corr}%
  {\color{blue}\ifmmode\text{\,\sout{\ensuremath{#1}}\,}\else\sout{#1}\fi}%
  {\color{red}#2}%
  {\color{violet} #3}%
}

\usepackage[hidelinks]{hyperref}
\title{Convergence and error analysis for numerical approximations of stochastic Leray--Lions equations}
\author{%
  Kim-Ngan Le\thanks{School of Mathematics, Monash University, Australia.
    Email: ngan.le@monash.edu}
  \and
  Huateng Zhu\footnotemark[1]\thanks{School of Mathematics, Monash University, Australia.
    Email: huateng.zhu@monash.edu}
}
\date{}  
\begin{document}
\maketitle
\begin{abstract}
    We analyse a divergence form stochastic differential equation featuring a Leray-Lions flux function that satisfies standard Carath\'eodory, growth, monotonicity, and coercivity conditions.
    We consider two types of noise: multiplicative Lipschitz continuous noise and a transport noise.
    Our approach utilises the gradient discretisation method, a generic framework encompassing various numerical schemes such as conforming and nonconforming finite element methods, hybrid mimetic mixed methods, and hybrid high order methods.
    For both types of noise, 
    we propose a stable generic scheme,
    prove that scheme solutions  converge to a strong solution in the probabilistic sense.
    For the multiplicative Lipschitz continuous noise,
    we establish the rate of convergence.
\end{abstract} 

\medskip
\noindent\textbf{Acknowledgements:} This work was partially supported by the Australian
Government through the Australian Research Council's Discovery Projects funding scheme
(grant number DP220100937).

\section{Introduction}
We are interested in the numerical approximations of the nonlinear stochastic partial differential equation in divergence form, specifically, the following model:
\begin{subequations}\label{model:general}
\begin{alignat}{3}
d \overline{u} = \, & \text{div} (\bm{a} (x, \nabla  \overline{u}))\,dt + d \mathbb{W}_{\overline{u}} && \quad  \text{in \((0,T) \times \mathcal{O}\)} \label{model:pde with stochastic force}\\
\overline{u} = \, & 0 &&\quad  \text{on \( (0,T) \times \partial \mathcal{O}\)}\\
\overline{u}(0, \cdot) = \, & u_0 && \quad \text{in \(\mathcal{O}\)},
\end{alignat}
\end{subequations}
where \(\mathcal{O} \subset \mathbb{R}^d\) \((d\geq 1)\) is a polygonal bounded connected domain with boundary \(\partial \mathcal{O}\),
\(\mathbb{W}_u\) is a noise of Gaussian type depending on \(\overline{u}\),
and \(\bm{a}: \mathcal{O} \times \mathbb{R}^d \rightarrow \mathbb{R}^d\) is a flux function of Leray-Lions type that satisfies the following assumption.
\begin{assumption}\label{def:assumption opertator a}
    The operator \(\bm{a}: \mathcal{O} \times \mathbb{R}^d \rightarrow \mathbb{R}^d\) is a Carath\'eodory function;
    that is,
    for almost all \(x \in \mathcal{O}\), \(\bm{a}(x, \cdot)\) is continuous, 
    and for all \(\bm{\mu} \in \mathbb{R}^d\), \(\bm{a}(\cdot, \bm{\mu})\) is measurable. 
    Moreover, \(\bm{a}\) satisfies the following properties: for a fixed \(p \in (1, \infty)\), 
    for almost every \(x \in \mathcal{O}\) and for all \((\bm{\mu}, \bm{\nu}) \in \mathbb{R}^d \times \mathbb{R}^d\),
    \begin{itemize}[align=left, leftmargin=*]
        \item (coercivity) There exist \(\underline{a} \in (0, \infty)\) and \(\underline{\bm{a}} \in L^1(\mathcal{O})\) such that \[\bm{a}(x, \bm{\mu}) \cdot \bm{\mu} \geq \underline{a}|\bm{\mu}|^p - \underline{\bm{a}}(x).\]
        \item (monotonicity) \[(\bm{a}(x, \bm{\mu}) - \bm{a}(x, \bm{\nu}))\cdot(\bm{\mu} - \bm{\nu}) \geq 0.\] 
        \item (boundedness) There exist \(\bm{\overline{a}} \in L^{p'}(\mathcal{O})^d\) and \(\lambda\in (0, \infty)\) such that \[|\bm{a}(x, \bm{\mu})| \leq |\bm{\overline{a}}(x)| + \lambda |\bm{\mu}|^{p-1}.\]
    \end{itemize}
\end{assumption}

The flux function \(\bm{a}\) of Leray-Lions type, covers a range of operators, that arise in the practical applications.
For instance, 
the \(p\)-laplacian \(\bm{a}(x, \bm{\xi}) = |\bm{\xi}|^{p-2} \bm{\xi}\) appears in the modelling of glacier motion \cite{glowinski2003approximation}, of incompressible turbulent flows in porous media \cite{diaz1994nonlinear}, and in airfoil design \cite{glowinski2013numerical};
the linear anisotropic heterogeneous operators \(\bm{a}(x, \bm{\xi}) = \Lambda(x) \bm{\xi}\) involve in models used in underground engineering;
non-local operators arise in image processing, in the spirit of \cite{catte1992image, drblikova2008convergence} and references therein.
In the deterministic setting, numerical analysis on both stationary and evolutionary Leray-Lions problems has been investigated thoroughly; 
see, for example, \cite{andreianov2007discrete, di2017w, pietro2021improved, droniou2013gradient, Droniou2018TheGD} and references therein.

To predict the turbulence and uncertainties that arise in practical scenarios involving large scales and rapid fluctuations in velocity and pressure, a stochastic perturbation is incorporated into the model.
This approach is essential in fluid mechanics, statistical physics, mathematical finance, population dynamics, and chemical and biological reactions.
Consequently, theorectical and numerical analysis of stochastic models has gained significant attention in recent decades.
See, for example, \cite{breit2018stochastically,kuksin2012mathematics,oksendal2003stochastic, evans2012introduction, da2014stochastic, lord2014introduction} and references therein.
In this paper, we investigate two distinct types of stochastic forcing.
First, we incorporate a multiplicative Lipschitz continuous noise to represent intrinsic randomness,
where the magnitude of the fluctuations depends on the current state of the system.
Second, we impose a transport noise to describe random advection, 
which models the dispersive behavior of a random velocity field.
In particular,
\begin{itemize}[leftmargin=*]
    \item \(d\mathbb{W}_u =  B(u) dW_t\), where \(B\) is a Lipschitz continuous function and \((W_t)_t\) is a Hilbert space-valued Wiener process,
    \item \(d\mathbb{W}_u = (\nabla u \cdot \bm{\sigma}) \circ d\beta_t\), where \(\bm{\sigma}\) is a divergence-free vector field and \((\beta_t)_t\) is a \(\mathbb{R}\)-valued Wiener process.
\end{itemize}

Finite element methods are widely utilised in discretisation of stochastic differential equations; 
see, for example, 
stochastic Stokes equations \cite{feng2021optimally, feng2021analysis},
stochastic Navier-Stokes equations \cite{carelli2012rates, brzezniak2013finite, breit2021convergenceNS, qiu2024analysisNS, breit2024error},
a stochastic mean curvature flow \cite{feng2014finite},
stochastic Landau-Lifshitz-Gilbert equations \cite{banas2013computational, bavnas2014convergent, goldys2016finiteStocLLG, goldys2020weakLLG},
stochastic Allen-Cahn equations \cite{feng2017finite, , majee2018optimal},
a stochastic Cahn-Hilliard equation \cite{feng2020fully},
stochastic \(p\)-Laplace-type system \cite{breit2021space},
stochastic \(p\)-Stokes system \cite{le2024class},
stochastic total variation flow \cite{bavnas2023numerical, bavnas2021convergent},
a class of fourth order stochastic evolution equations \cite{goldys2025mixed}.
However,  in practical applications, finite element methods have limitations related to the conformity constraint on the mesh.
For instance, in numerical simulations for geological flows and fluids, all geometric information is carried by the mesh.
The mesh may have elements that are unnecessarily small in certain regions or unnecessarily large in others, from computational viewpoints.
Even though the latter issue could be addressed by local mesh refinement, coarsening a given mesh in finite element codes is not possible in general.
Another limitation in practice is that the need to identify a local function space on which the degrees of freedom are unisolvent, which sometimes become a serious hurdle.
For instance, in the case of the Crouzeix-Raviart element, this constraint makes it difficult to identify higher-order versions of space that can be used in practice;
in the case of the mixed finite elements for elasticity, this constraint makes it formidably challenging to incorporate the symmetry constraint into the space. 
These motivate to analyse models using other numerical methods.
However, the existing literature on numerical analysis of the stochastic model using nonconforming methods is limited.
In this paper, we use the gradient discritisation method (GDM) to analysis the stochastic model, which is a generic framework that covers conforming finite element methods, Crouzeix-Raviart finite element method,  mixed finite element methods, and polytopal methods (which do not have the mesh or function constraints) like the virtual element method, discontinuous Galerkin method, hybrid mimetic mixed methods \cite{droniou2010unified}, and hybrid high order methods \cite{di2025finite, di2020hybrid}.
See monograph \cite{Droniou2018TheGD} for detailed discussions and methods that GDM covers.
By analysing a model using GDM, the obtained results hold for methods in the framework.
Numerous models have been analysed using GDM; 
for example, 
the \(p\)-laplacian equations \cite{Droniou2018TheGD}, anisotropic equations \cite{droniou2016gradient, droniou2018improved}, Leray-Lions equations \cite{droniou2013gradient}, elasticity equations \cite{droniou2015gradient}, Stokes equations \cite{droniou2016gradientStokes}, porous media equations \cite{droniou2020gradientPorous},
a stocahtic Leray-Lions problem with multiplicative continuous noise \cite{droniou2022design}, stochastic Navier-Stokes equations with transport noise \cite{droniou2024reaching}, stochastic Stefan probelm  with multiplitive continuous noise \cite{droniou2024numericalStefan}.
 
A conventional method of showing convergence of numerical approximations, 
is via the compactness: extracting (weakly) converging subsequence from bounded sequences and showing the limits satisfy the model in an appropriate sense.
For this method, the existence of the continuous solution is not assumed and is a by-product of the analysis.
In the stochastic setting, 
there are vast amount of studies on convergence to the martingale solution (see, e.g.  \cite{goldys2016finiteStocLLG,  goldys2020weak,  droniou2022design, ondrejat2023numerical,bavnas2023numerical,droniou2024numericalStefan,  goldys2025mixed}).
One core step for these proofs is, extracting \(\mathbb{P}\)-almost surely strong convergence subsequences in appropriate spaces and topologies, using a version of Skorohod embedding theorem.
Although the \(\mathbb{P}\)-almost surely strong convergences can be extremely useful when dealing with non-linearities, these limits might be in the different probabilistic space and driven by the different stochastic forcing with the numerical solution.
This sense of limit solution is therefore a weak solution in the probabilistic sense.
And this might make some numerical simulations lack of theoretical backings.
Note that these limits might be able to be brought back to the original probabilistic space using the Gyöngy-Krylov lemma (see, e.g. \cite[Section 2.10]{breit2018stochastically}), but the additional assumption of path-wise uniqueness of is needed and that the \(\mathbb{P}\)-almost sure convergence reduces to convergence in probability  (see, e.g. \cite{gyongy2022existence, bavnas2023numerical}).
There is, however, limited number of studies on convergence to the strong solution in the probabilistic sense (e.g. variational solutions), due to the difficulties arise when identifying non-linearities in the limit (see, e.g. \cite{gyongy2005discretization, emmrich2017nonlinear}). 
 
Another conventional method of showing convergence of numerical approximations is through error estimates: bounding the distance between the approximate solution and the exact solution in terms of discretisation parameters (e.g. mesh size and/or time step) or regularisation parameters.
Unlike the compactness method, 
the goal is to establish the strong convergences in Banach space and an explicit rate of convergence,
under assumptions for regularity of the exact solution.
For instance,
in the context of conforming finite elements, such results exist for the stochastic \(p\)-Laplacian system \cite{breit2021space}, the stochastic Allen–Cahn equation \cite{carelli2012rates}, and the stochastic Navier-Stokes equations \cite{breit2021convergenceNS, breit2024error};
within mixed finite element methods, 
error estimates have been conducted for the stochastic Stokes equation \cite{feng2021optimally} 
and the stochastic Navier-Stokes equation \cite{qiu2024analysisNS};
and within virtual element methods, recent developments include works such as \cite{liu2026strong}.
Despite these advancements, there remains a notable lack of results regarding error estimates for stochastic differential equations involving nonlinear diffusivity or using nonconforming methods. 

The aim of this paper is to show that our approximations converge to solutions of the model \eqref{model:general} in a suitable sense using the gradient discretisation method.
The primary challenges of this paper arise from non-linearities in terms involving with the divergence and noise.
Within the framework of the gradient discretisation method,
we propose generic numerical schemes for the stochastic Leray-Lions equation: one incorporating multiplicative Lipschitz continuous noise and the other featuring transport noise.
For both noise types, we establish the existence and stability of the approximate solutions in their respective schemes,
prove the convergence to a variational solution (a strong solution in the probabilistic sense).

In summary, this paper analyses the stochastic Leray-Lions differential equation using the gradient discretisation method.
We consider two types of noise, multiplicative Lipschiz continuous noise and transport noise.
For both noises, we propose a generic stable scheme, show that scheme solutions converge to a variational solution.
For the multiplicative Lipschitz continuous noise, we establish the rate of convergence.
Notations and notion of gradient discretisation are given in Section \ref{sec:notations}.
The main results for the multiplicative Lipschitz continuous noise is given in Section \ref{sec:LC} with proofs given in Section \ref{sec:proof-LC cv compactness} and Section \ref{sec:proof-LC cv error}.
The main result for the transport noise is given in Section \ref{sec:transport} with proof given in Section \ref{sec:proof-transport cv compactness}.

\section{Gradient discretisations method}\label{sec:notations}
Given \(T > 0\), 
let \(((0,T), \mathcal{B}(0,T), dt)\) be a measure space, where \(\mathcal{B}(0,T)\) is complete Borel \({\sigma}\)-algebra and \(dt\) is Lebesgue measure on \((0,T)\);
\((\Omega, \mathcal{F}, \mathbb{P})\) be a complete probability space.
Fix a stochastic basis  \((\Omega, \mathcal{F}, \mathbb{F} := (\mathcal{F}_t)_{t\in [0,T]}, \mathbb{P})\).
For \(p \in (1, \infty)\), denote \(p' := p/(p-1)\).
For a Banach space \(X\), 
\(L^p(\Omega, \mathbb{P};X)\) 
and \(L^p(\Omega_T, d t \otimes \mathbb{P}; X)\) (where \(\Omega_T := (0,T) \times \Omega\)) 
denote the standard Bochner spaces (with respect to \(\mathcal{F}\)), 
whose norms are given by 
\[
\|f\|_{L^p(\Omega;X)} := \mathbb{E} \big[\|f\|_X^p\big]^{1/p}
\quad \text{and} \quad 
\|f\|_{L^p(\Omega_T; X)} := \mathbb{E}\Big[ \int_0^T \|f\|_X^p dt \Big]^{1/p}.
\]
Let \(\mathcal{O} \subset \mathbb{R}^d\) \((d\geq 1)\) be a polygonal open bounded connected domain with boundary \(\partial \mathcal{O}\).
The space \(W_0^{1,p}(\mathcal{O})\) denotes the Sobolev space consisting of functions in \(L^p(\mathcal{O})\), whose first weak derivative belongs to \(L^p(\mathcal{O})\), and which vanish on the boundary;
see \cite{brezis2011functional} for detailed definitions.

In the rest of this section, we introduce the notations and approximation properties of gradient discretisations for stochastic evolutionary models with homogeneous Dirichlet boundary conditions.

\begin{definition}\label{def:stoc GD}
  A space-time gradient discretisation \(\mathcal{D}_T\) with homogeneous Dirichlet boundary conditions is defined by
  \(\mathcal{D}_T = (X_{\mathcal{D}}, \Pi_{\mathcal{D}}, \nabla _{\mathcal{D}},\mathcal{I}_{\mathcal{D}},(t^{(n)})_{n=0}^N)\),
  where:
  \begin{enumerate}[label=(\roman*),align=left, leftmargin=*]
       \item the discrete space \(X_{\mathcal{D}}\) is a finite dimensional vector space on \(\mathbb{R}\);
    \item the operator \(\Pi_{\mathcal{D}}: X_{\mathcal{D}} \rightarrow L^{\max(2,p)}(\mathcal{O})\) is linear;
    \item the operator \(\nabla _{\mathcal{D}}: X_{\mathcal{D}} \rightarrow L^{p}(\mathcal{O})^d\) is linear and must be chosen such that \(\| \cdot\|_{\mathcal{D}} := \| \nabla_{\mathcal{D}} \cdot \|_{L^p(\mathcal{O})}\) is a norm on \(X_{\mathcal{D}}\);
    \item \(\mathcal{I}_{\mathcal{D}}: L^2(\mathcal{O}) \cap W_{0}^{1,p}\) is an interpolation operator;
    \item the time partition \(t^{(0)} = 0 < t^{(1)} < \cdots < t^{(N)} = T\) is uniform in the sense that \(k_{\mathcal{D}} = t^{(n+1)} - t^{(n)}\) is constant time step.
  \end{enumerate}
  To a family of \(X_{\mathcal{D}}\)-valued random variable \((v^{(n)})_{n = 0}^N\),
  we associate the piece-wise constant
  \(X_{\mathcal{D}}\)-valued processes \((v(t))_{t\in [0,T]}\) and \((\underline{v}(t))_{t\in [0,T]}\),
  \(L^p(\mathcal{O})\)-valued processes \((\Pi_{\mathcal{D}}v(t))_{t\in [0,T]}\) and  \((\Pi_{\mathcal{D}}\underline{v}(t))_{t\in [0,T]}\),
  and \(L^p(\mathcal{O})^d\)-valued processes \((\nabla _{\mathcal{D}} v(t))_{t\in [0,T]}\) and \((\nabla _{\mathcal{D}} \underline{v}(t))_{t\in [0,T]}\)
  defined by:
  for all \(n = 0, \ldots, N\),
  \[
    v(t^{(n)}) :=  v^{(n)} =:  \underline{v}(t^{(n)}),
  \]
  and
  for all \(t \in (t^{(n)}, t^{(n+1)})\),
  \[
    v(t) := v^{(n+1)}\quad \text{and} \quad  \underline{v}(t) := v^{(n)},
  \]
  and for almost every \(\bm{x} \in \mathcal{O}\),
  \begin{align*}
    \Pi_{\mathcal{D}} v(t, \bm{x}) = \Pi_{\mathcal{D}} [v(t)](\bm{x}),
    \quad
    \Pi_{\mathcal{D}} \underline{v}(t, \bm{x}) = \Pi_{\mathcal{D}} [\underline{v}(t)](\bm{x}),
  \end{align*}
  and
  \begin{align*}
    \nabla _{\mathcal{D}} v(t, \bm{x}) = \nabla _{\mathcal{D}} [v(t)](\bm{x}),
    \quad
    \nabla _{\mathcal{D}} \underline{v}(t, \bm{x}) = \nabla _{\mathcal{D}} [\underline{v}(t)](\bm{x}).
  \end{align*}
\end{definition}

The following Stokes' formula is an important property that satisfied by the continuous gradient:
\begin{equation*}
  \int_{\mathcal{O}} \nabla u \cdot \bm{\varphi} + u \, \text{div} \bm{\varphi} \, dx = 0, \quad \text{\(\forall u \in W^{1,p} (\mathcal{O})\), \(\bm{\varphi} \in W_{\text{div}}^{p'}(\mathcal{O})\), \(p' = \frac{p}{p-1}\).}
\end{equation*}
The limit-conformity states that, asymptotically, \(\Pi_{\mathcal{D}}\) and \(\nabla_{\mathcal{D}}\) satisfy this above Stokes' formula; see Definition \ref{def:GD-limit-conformity} below.
It is clear that the limit-conformity is inherent in all conforming finite element methods.
For non-conforming methods, however, the limit-conformity is needed to ensure
that the discrete function reconstruction and the discrete gradient reconstruction are chosen in a coherent way.
\begin{definition}[Limit-conformity]\label{def:GD-limit-conformity}
  Let \(\mathcal{D}_T\) is a gradient discretisation in the sense of Definition \ref{def:stoc GD},
  and let \(W_{\operatorname{div}}^{q}(\mathcal{O}):= \{ \bm{\varphi} \in (L^q(\mathcal{O}))^d: \operatorname{div} \, \bm{\varphi}  \in L^q(\mathcal{O})\}\)  with \(q = p/(p-1)\)
  and \(W_{\mathcal{D}}: W_{\operatorname{div}}^{q}(\mathcal{O}) \rightarrow [0, \infty)\) be defined by
  \begin{equation*}
    W_{\mathcal{D}}(\bm{\varphi}) := \max_{w_{\mathcal{D}} \in X_{\mathcal{D}} \setminus \{0\}} \frac{|\widetilde{W}_{\mathcal{D}}(\bm{\varphi}, w_{\mathcal{D}})|}{\|w_{\mathcal{D}}\|_{\mathcal{D}}} ,
  \end{equation*}
  where
  \begin{equation*}
    \widetilde{W}_{\mathcal{D}}(\bm{\varphi}, w_{\mathcal{D}}) = \int_{\mathcal{O}} \nabla _{\mathcal{D}}w_{\mathcal{D}}(\bm{x}) \cdot \bm{\varphi}(\bm{x})  + \Pi_{\mathcal{D}}w_{\mathcal{D}} (\bm{x}) \, \text{div} \,\bm{\varphi} (\bm{x}) \, d\bm{x}.
  \end{equation*}
  A sequence \(((\mathcal{D}_T)_m)_{m \in \mathbb{N}}\) of gradient discretisations in the sense of Definition \ref{def:stoc GD} is limit-conforming
  if for all \(\bm{\varphi} \in W_{\operatorname{div}}^{q}(\mathcal{O})\),
  \(W_{\mathcal{D}_m}(\bm{\varphi}) \rightarrow 0\) as \(m \rightarrow \infty\).
\end{definition} 

Consistency is ensured by a proper choice of the reconstructions of function and its gradient, the interpolation for the initial condition, and the time partition.
\begin{definition}[Consistency for GD]\label{def:stoc-GD consistency}
  Let \(\mathcal{D}_T\) is a gradient discretisation in the sense of Definition \ref{def:stoc GD}.
  Let \(S_{\mathcal{D}}: L^{2}(\mathcal{O}) \cap W_{0}^{1, p}(\mathcal{O}) \rightarrow [0, \infty) \) be defined by
  \begin{equation*}
    \forall \phi \in L^{2}(\mathcal{O}) \cap W_{0}^{1,p}(\mathcal{O}),
    \quad
    S_{\mathcal{D}}(\phi) := \min_{v_{\mathcal{D}} \in X_{\mathcal{D}}} \big( \| \Pi_{\mathcal{D}}v_{\mathcal{D}} - \phi\|_{L^{\max(2,p)}(\mathcal{O})} + \|\nabla _{\mathcal{D}} v_{\mathcal{D}} - \nabla  \phi\|_{L^p(\mathcal{O})} \big).
  \end{equation*}
  A sequence \(((\mathcal{D}_T)_m)_{m \in \mathbb{N}}\) of gradient discretisation in the sense of Definition \ref{def:stoc GD} is  {consistent}
  if, as \(m\rightarrow \infty\),
  \begin{itemize}[align=left, leftmargin=*]
    \item For all \(\phi \in L^{2}(\mathcal{O}) \cap W_{0}^{1,p}(\mathcal{O})\), \(S_{\mathcal{D}_m}(\phi) \rightarrow 0\);
    \item For all \(u \in L^{2}(\mathcal{O}) \cap W_{0}^{1,p}(\mathcal{O})\), \(\| u - \Pi_{\mathcal{D}_m} \mathcal{I}_{\mathcal{D}_m} u\|_{L^2(\mathcal{O})}
          + \| \nabla u - \nabla_{\mathcal{D}_m} \mathcal{I}_{\mathcal{D}_m} u\|_{L^p(\mathcal{O})}\rightarrow 0 \);
    \item \(k_{\mathcal{D}_m} \rightarrow 0\).
  \end{itemize}
\end{definition}

The notion of space size is introduced in Definition \ref{def:GD space size} to quantitatively measure the approximation properties of GD.
\begin{definition}[{Space size}]\label{def:GD space size}
  Let \(E_s\) be a Banach space that is continuously embedded in \(W_0^{1,p}(\Omega)\).
  Let \(\bm{E}_w\) be a Banach space continuous embedded in \(W_{\text{div}}^{p'}(\mathcal{O})\).
  Define
  \begin{equation*}
    s_{\mathcal{D}}(E_s) := \sup_{\varphi \in  E_s \setminus \{0\}}  \frac{S_{\mathcal{D}} (\varphi)}{||\varphi||_{E_s}}
    \quad \text{and} \quad
    w_{\mathcal{D}}(\bm{E}_w) := \sup_{\bm{\psi} \in \bm{E}_w \setminus \{0\}} \frac{W_{\mathcal{D}}( \bm{\psi})} {||\bm{\psi}||_{\bm{E}_w}}.
  \end{equation*}
  The \textit{space size} of \(\mathcal{D}\) with respect to \(E_s\) and \(\bm{E}_w\) is defined by
  \begin{align*}
    h_{\mathcal{D}} (E_s; \bm{E}_w) := \max\left(s_{\mathcal{D}}(E_s), w_{\mathcal{D}}(\bm{E}_w)\right).
  \end{align*}
\end{definition}

Coercivity is at the core of techniques which allows one to carry out convergence proofs without assuming nonphysical regularities on the data or the solution.
\begin{definition}[Coercivity]\label{def:GD-coercivity}
  Let \(((\mathcal{D}_T)_m)_{m \in \mathbb{N}}\) of gradient discretisation in the sense of Definition \ref{def:stoc GD}, define \(C_{\mathcal{D}}\) as the norm of the linear mapping \(\Pi_{\mathcal{D}}\):
  \begin{equation}\label{eq-def:coercivity}
    C_{\mathcal{D}} := \max_{v_{\mathcal{D}} \in X_{\mathcal{D}}\setminus \{0\}} \frac{\|\Pi_{\mathcal{D}} v_{\mathcal{D}}\|_{L^p(\mathcal{O})}}{\|v_{\mathcal{D}} \|_{\mathcal{D}}}.
  \end{equation}
  A sequence \(((\mathcal{D}_T)_m)_{m \in \mathbb{N}}\) of gradient discretisations in the sense of Definition \ref{def:stoc GD}  is coercive if there exists \(C_p \in \mathbb{R}_+\) such that \(C_{\mathcal{D}_m} \leq C_p\) for all \(m \in \mathbb{N}\).
\end{definition}

The following lemmas, 
Lemma \ref{lem: stoc regularity of the limit} and  Lemma \ref{lem:GD-stoc interpolation strong cv},
are counterparts of \cite[Lemma 2.15, Lemma 3.21, Lemma 4.8, Lemma 4.10]{droniou2018gradient}
\begin{lemma}[Regularity of the limit]\label{lem: stoc regularity of the limit}
  Let \((\Omega, \mathcal{F}, \mathbb{F}:= (\mathcal{F}_t)_t, \mathbb{P})\) be a stochastic basis.
  Let \(p \in (1, \infty)\) and
  let \(((\mathcal{D}_T)_m)_{m \in \mathbb{N}}\) be a sequence of limit-conforming, consistent, and coercive space-time GDs for homogeneous Dirichlet boundary conditions.
  Let \(q \in (1, \infty)\) and
  take, for any \(m \in \mathbb{N}\), piece-wise constant in time \(X_{\mathcal{D}_m}\)-valued stochastic process \(u_{\mathcal{D}_m} \) such that 
  \((\| u_{\mathcal{D}_m}\|_{{L}^q(\Omega_T; X_{\mathcal{D}_m})})_m\) is bounded and each \(u_{\mathcal{D}_m}^{(n)}\)is \(\mathcal{F}_{t^{(n)}}\)-adapted.

  Then there exists a progressively measurable \(u \in L^q(\Omega_T; W_0^{1, p}(\mathcal{O}))\) such that,
  up to a subsequence,
  \begin{align*}
     & \Pi_{\mathcal{D}_m} u_{\mathcal{D}_m} \rightarrow u \text{ weakly in \(L^q(\Omega_T;L^p(\mathcal{O}))\)},             \\
     & \nabla_{\mathcal{D}_m} u_{\mathcal{D}_m} \rightarrow \nabla u \text{ weakly in \(L^q(\Omega_T;L^p(\mathcal{O})^d)\)}.
  \end{align*}
\end{lemma}
\begin{proof}
  Since
  \[\|\underline{u}_{\mathcal{D}_m}\|_{L^q(\Omega_T; X_{\mathcal{D}_m})}^q
    \leq k_{\mathcal{D}_m} \| u_{\mathcal{D}_m}^{(0)}\|_{L^q(\Omega; X_{\mathcal{D}_m})}^q
    + \| u_{\mathcal{D}_m}\|_{L^q(\Omega_T; X_{\mathcal{D}_m})}^q\]
  and the coercivity implies that
  \((\Pi_{\mathcal{D}_m} u_{\mathcal{D}_m})_m\) and \((\Pi_{\mathcal{D}_m} \underline{u}_{\mathcal{D}_m})_m\)
  are bounded in \(L^q(\Omega_T; L^p(\mathcal{O}))\),
  and
  \((\nabla_{\mathcal{D}_m} u_{\mathcal{D}_m})_m\) and \((\nabla_{\mathcal{D}_m} \underline{u}_{\mathcal{D}_m})_m\)
  are bounded in \(L^q(\Omega_T; L^p(\mathcal{O})^d)\).
  These spaces are reflexive since \(p, q \in (1, \infty)\), so by \cite[Theorem 3.18]{brezis2011functional}
  there exist \(\overline{u}, \tilde{u} \in L^q(\Omega_T; L^p(\mathcal{O}))\)
  and \(\bm{v}, \bm{w} \in L^q(\Omega_T; L^p(\mathcal{O})^d)\) such that,
  along a subsequence as \(m \rightarrow \infty\), 
  \(\Pi_{\mathcal{D}_m} u_{\mathcal{D}_m}\) (respt. \(\Pi_{\mathcal{D}_m} \underline{u}_{\mathcal{D}_m}\)) converges weakly to \(\overline{u}\) (respt. \(\tilde{u}\)) in \( L^q(\Omega_T; L^p(\mathcal{O}))\) 
  and 
  \(\nabla_{\mathcal{D}_m} u_{\mathcal{D}_m}\) (respt. \(\nabla_{\mathcal{D}_m} \underline{u}_{\mathcal{D}_m}\)) converges weakly to \(\bm{v}\) (respt. \(\bm{w}\)) in \( L^q(\Omega_T; L^p(\mathcal{O})^d)\).

  Now we show that  \(\bm{v} = \nabla \overline{u}\), \(dt \otimes \mathbb{P}\)-almost surely.
  First, for almost all \((t,\omega) \in \Omega_T\),
  we extend functions \(\Pi_{\mathcal{D}_m} u_{\mathcal{D}_m}\), \(\overline{u}\), \(\nabla_{\mathcal{D}_m} u_{\mathcal{D}_m}\), and  \(\bm{v}\)  by 0 outside of \(\mathcal{O}\).
  These convergences still hold weakly in \(L^q(\Omega_T; L^p(\mathbb{R}^d))\) and \(L^q(\Omega_T;L^p(\mathbb{R}^d)^d)\).
  Taking \(\mathbb{R}\)-valued stochastic process \(\psi \in L^{q'}(\Omega_T)\)
  and deterministic functions \(\bm{\varphi} \in C_c^{\infty}(\mathbb{R}^d)^d\),
  by the Hölder inequality,
  the following can be obtained:
  \begin{align*}
    \mathbb{E}\Big[ \Big|\int_0^T \int_{\mathbb{R}^d} \nabla_{\mathcal{D}_m} u_{\mathcal{D}_m}  \cdot ( \psi  \bm{\varphi}) + \Pi_{\mathcal{D}_m} u_{\mathcal{D}_m}  \, \psi \operatorname{div} \bm{\varphi} \Big| \Big]
    \leq
    W_{\mathcal{D}_m}(\bm{\varphi}) \| u_{\mathcal{D}_m} \|_{L^p(\Omega_T; X_{\mathcal{D}_m})} \|\psi\|_{L^{q'}(\Omega_T)}.
  \end{align*}
  Thanks to the limit-conformity of the sequence of approximations,
  passing the limit leads to
  \begin{equation*}
    \mathbb{E}\Big[\int_0^T
      \psi
      \Big(
      \int_{\mathbb{R}^d}  \bm{v} \cdot  \bm{\varphi}  + \overline{u}  \, \operatorname{div} \, \bm{\varphi}   \, dx
      \Big)
      \, dt  \Big] = 0, \quad \forall \psi \in L^{q'}(\Omega_T).
  \end{equation*}
  This implies that
  \begin{equation*}
    \int_{\mathbb{R}^d}  \bm{v}  \cdot \bm{\varphi}  + \overline{u}  \, \text{div} \,  \bm{\varphi}  \, dx = 0
    \quad
    \text{\(dt \otimes\mathbb{P}\)-almost surely.}
  \end{equation*}
  Therefore, \(\bm{v}\) is a distributional derivative of \(\overline{u}\), that is, \(\bm{v} = \nabla \overline{u}\). 

  Recall that, on \([t^{(n)}, t^{(n+1)})\) for \(n = 0, \ldots, N_m - 1\),
  \(\underline{u}_{\mathcal{D}_m}\) equals \(u_{\mathcal{D}_m}^{(n)}\) which is \(\mathcal{F}_{t^{(n)}}\)-measurable.
  Hence \(\underline{u}_{\mathcal{D}_m}\) is adapted with right-continuous paths, and thus progressively measurable.
  Since \(\Pi_{\mathcal{D}_m}\) and \(\nabla_{\mathcal{D}_m}\) are linear maps,
  \(\Pi_{\mathcal{D}_m} \underline{u}_{\mathcal{D}_m}\) and \(\nabla_{\mathcal{D}_m} \underline{u}_{\mathcal{D}_m}\)
  are progressively measurable.
  Since the progressively measurable elements of \(L^q(\Omega_T; L^p(\mathcal{O}))\) form a closed linear subspace,
  this subspace is weakly closed by \cite[Theorem 3.7]{brezis2011functional} (cf. the proof of \cite[Theorem 4.2.4]{liu2015stochastic}).
  Hence the weak limit \(\tilde{u}\) is progressively measurable; similarly for \({\bm{w}}\).
  So now,  it suffices to identify \(\overline{u}\) with \(\tilde{u}\), and \(\nabla \overline{u}\) with \(\bm{w}\).

  Take \(f \in L^{q'}(\Omega_T; L^{p'}(\mathcal{O}))\), consider its piecewise constant in time approximation:
  for each
  \(n = 0, 1, \ldots, N_m - 1\),
  \begin{equation*}
    F_m^{(n)} := \frac{1}{k_{\mathcal{D}_m}}\int_{t^{(n)}}^{t^{(n+1)}} f(s) \, ds,
    \text{ and }
    F_m(t) =
    \begin{cases}
      F_m^{(0)} & \text{if \(t = 0\)},                                                        \\
      F_m^{(n)} & \text{if \(t\in (t^{(n-1)}, t^{(n)}]\), for all \( 1 \leq n\leq N_m - 1\),} \\
      0         & \text{otherwise}.
    \end{cases}
  \end{equation*}
  Then \(F_m \rightarrow f\) strongly in \(L^{q'}(\Omega_T; L^{p'}(\mathcal{O}))\),
  as \(m \rightarrow \infty\) (see Lemma \ref{lem:strong cv piecewise approximation} with \(k_m = k_{\mathcal{D}_m}\)).
  For any  \(f \in L^{q'}(\Omega_T; L^{p'}(\mathcal{O}) )\),
  we can write
  \begin{equation}\label{eq-lem:GD adaptivity of limits 2}
    \int_{0}^{T}  \langle F_m(s),  \Pi_{\mathcal{D}_m} {u}_{\mathcal{D}_m}(s) \rangle \, ds
    =  \int_{k_{\mathcal{D}}}^T  \langle f(s),  \Pi_{\mathcal{D}_m}\underline{ {u}}_{\mathcal{D}_m}(s)  \rangle \, ds
  \end{equation}
  and
  \begin{equation}\label{eq-lem:GD adaptivity of limits 3}
    \begin{aligned}
      \mathbb{E} \int_0^T \langle f(s), \overline{u}(s) - \tilde{u}(s)\rangle \, ds
      =
      \, & \mathbb{E} \int_0^T  \langle f(s), \overline{u}(s) - \Pi_{\mathcal{D}_m} {u}_{\mathcal{D}_m}(s) \rangle \, ds
      \\
         & +
      \mathbb{E} \int_0^T  \langle f(s), \Pi_{\mathcal{D}_m} {u}_{\mathcal{D}_m}(s) - \Pi_{\mathcal{D}_m}\underline{ {u}}_{\mathcal{D}_m}(s) \rangle \, ds
      \\
         & +
      \mathbb{E} \int_0^T  \langle f(s), \Pi_{\mathcal{D}_m}\underline{ u}_{\mathcal{D}_m}(s) - \tilde{u}(s) \rangle \, ds
      =: I_1 + I_2 + I_3.
    \end{aligned}
  \end{equation}
  The terms \(I_1\) and \(I_3\) converge to \(0\) as \(m \rightarrow \infty\).
  Consider the term \(I_2\), by \eqref{eq-lem:GD adaptivity of limits 2},
  \begin{align*}
    I_2 = \mathbb{E} \int_0^T
    \langle f(s) - F_m(s), \Pi_{\mathcal{D}_m} {u}_{\mathcal{D}_m}(s) \rangle \, ds
    - \mathbb{E} \int_{0}^{k_{\mathcal{D}_m}}  \langle f(s),  \Pi_{\mathcal{D}_m}\underline{ {u}}_{\mathcal{D}_m}(s) \rangle \, ds =: I_{2,1} + I_{2,2},
  \end{align*}
  where the term \(I_{2,1}\) converges to \(0\) as \(m \rightarrow \infty\) and by the Hölder inequality,
  \begin{align*}
    I_{2,2}
    \leq \, &
    \| f\|_{L^{q'}(\Omega \times (0, k_{\mathcal{D}_m}); L^{p'}(\mathcal{O}))} \| \Pi_{\mathcal{D}_m}\underline{ {u}}_{\mathcal{D}_m}\|_{L^{q}(\Omega \times (0, k_{\mathcal{D}_m}); L^{p}(\mathcal{O}))} 
    \rightarrow
    0
    \quad \text{as \(k_{\mathcal{D}_m} \rightarrow  0\).}
  \end{align*}

  Now, letting \(m \rightarrow \infty\) in \eqref{eq-lem:GD adaptivity of limits 3},
  leads to
  \begin{align*}
    \mathbb{E} \int_0^T \langle f(s), \overline{u}(s) - \tilde{u}(s) \rangle \, ds = 0,
    \quad
    \forall
    f \in L^{q'}(\Omega_T; L^{p'}(\mathcal{O})).
  \end{align*}
  Hence, \(\overline{u} =  \tilde{u}\), \(dt \otimes \mathbb{P}\)-almost surely.
  The proof for \(\nabla \overline{u}\) is obtained using the same argument.
  Since the piecewise constant in time approximation holds for \(\bm{f}  \in L^{p'}(\Omega_T; L^{p'}(\mathcal{O})^d)\) as well,
  we simply replace
  \(f \in L^{q'}(\Omega_T; L^{p'}(\mathcal{O})) \) with \(\bm{f} \in L^{q'}(\Omega_T; L^{p'}(\mathcal{O})^d)\),
  \(\overline{u}\) with \(\nabla \overline{u}\),
  \(\tilde{u}\) with \(\bm{w}\),
  \(\Pi_{\mathcal{D}_m} {u}_{\mathcal{D}_m}\) with \(\nabla_{\mathcal{D}_m}{u}_{\mathcal{D}_m}\),
  \(\Pi_{\mathcal{D}_m}\underline{ {u}}_{\mathcal{D}_m}\) with \(\nabla_{\mathcal{D}_m}\underline{ {u}}_{\mathcal{D}_m}\).
\end{proof} 

\begin{lemma}[Interpolation of stochastic process]
  \label{lem:GD-stoc interpolation strong cv}
  Let \((\Omega, \mathcal{F}, \mathbb{F}:= (\mathcal{F}_t)_t, \mathbb{P})\) be a stochastic basis.
  Let \(p \in (1, \infty)\), \(T>0\), and \(V := L^2(\mathcal{O}) \cap W_0^{1,p}(\mathcal{O})\).
  Let \(((\mathcal{D}_T)_m)_{m \in \mathbb{N}}\) be a GD-consistent (in the sense of Definition \ref{def:stoc-GD consistency}) sequence of space-time GDs (in the sense of Definition \ref{def:stoc GD}) for homogeneous Dirichlet boundary conditions.
  Let \(V\)-valued process \(\overline{v} \in L^{\max(2,p)}(\Omega_T; V)\),
  there exists a sequence of \(X_{\mathcal{D}_m}\)-valued processes \((v_{\mathcal{D}_m})_{m\in \mathbb{N}}\) such that \(v_{\mathcal{D}_m} = (v_{\mathcal{D}_m}^{(n)})_{ n = 0}^{N_m}\) for all \(m \in \mathbb{N}\),
  and, as \(m \rightarrow \infty\),
  \begin{subequations}\label{eq-lem:GD interpolation strong cv}
    \begin{align}
     & \Pi_{\mathcal{D}_m} v_{\mathcal{D}_m} \rightarrow \overline{v} \text{ strongly in \(L^{\max(2,p)}(\Omega_T; L^{\max(2,p)}(\mathcal{O}))\)}, \label{eq-lem:GD interpolation strong cv 1} \\
     & \nabla_{\mathcal{D}_m} v_{\mathcal{D}_m} \rightarrow \nabla \overline{v} \text{ strongly in \(L^p(\Omega_T; L^p(\mathcal{O})^d)\)}. \label{eq-lem:GD interpolation strong cv 2}
  \end{align}
  \end{subequations}
\end{lemma} 
\begin{proof} 
  Since \(L^{\max(2,p)}(\Omega_T) \otimes V\) is dense in \(L^{\max(2,p)}(\Omega_T; V)\) (see, e.g. \cite[Proposition~1.2]{pisier2016martingales}), 
  similarly to the proof of \cite[Lemma 4.10]{droniou2018gradient},
  this proof can be reduced to the case  \(\overline{v} = \psi \phi \in L^{\max(2,p)}(\Omega_T) \otimes V\)
  with \(\psi \in L^{\max(2,p)}(\Omega_T)\)
  and \(\phi \in V\).
  For \(n = 0, 1, \ldots, N_m - 1\),
  define
  \begin{equation*}
    \psi_m(t) := \psi_m^{(n)} := \frac{1}{t^{(n+1)} - t^{(n)}} \int_{t^{(n)}}^{t^{(n+1)}}\psi( s)\, ds
    \quad
    \forall t \in (t^{(n)}, t^{(n+1)}],
  \end{equation*}
  and \(\phi_{\mathcal{D}_m} := P_{\mathcal{D}_m} \phi \in X_{\mathcal{D}_m}\) which realises the minimum of \(S_{\mathcal{D}_m}(\phi)\);
  set \(v_{\mathcal{D}_m}(t) := v_{\mathcal{D}_m}^{(n+1)} := \psi_m^{(n)} P_{\mathcal{D}_m} \phi\) for all \(t \in (t^{(n)}, t^{(n+1)}]\) and \(v_{\mathcal{D}_m}^{(0)} = 0.\)
  Consider
  \begin{align*}
         & \|\Pi_{\mathcal{D}_m} v_{\mathcal{D}_m} - \overline{v}\|_{L^{\max(2,p)}(\Omega_T; L^{\max(2,p)}(\mathcal{O}))}
    = \|\psi_m \Pi_{\mathcal{D}_m} \phi_{\mathcal{D}_m} - \psi \phi\|_{L^{\max(2,p)}(\Omega_T; L^{\max(2,p)}(\mathcal{O}))}                                                                                           \\
    \leq & \|\psi_m (\Pi_{\mathcal{D}_m} \phi_{\mathcal{D}_m} - \phi)\|_{L^{\max(2,p)}(\Omega_T; L^{\max(2,p)}(\mathcal{O}))} + \|(\psi_m - \psi)\phi\|_{L^{\max(2,p)}(\Omega_T; L^{\max(2,p)}(\mathcal{O}))}         \\
    =    & \|\psi_m\|_{L^{\max(2,p)}(\Omega_T)} \|\Pi_{\mathcal{D}_m} \phi_{\mathcal{D}_m} - \phi\|_{L^{\max(2,p)}(\mathcal{O})} + \|\psi_m - \psi\|_{L^{\max(2,p)}(\Omega_T)} \|\phi\|_{L^{\max(2,p)}(\mathcal{O})} \\
    \leq    & \|\psi_m\|_{L^{\max(2,p)}(\Omega_T)} S_{\mathcal{D}_m}(\phi) + \|\psi_m - \psi\|_{L^{\max(2,p)}(\Omega_T)} \|\phi\|_{L^{\max(2,p)}(\mathcal{O})}.
  \end{align*}
  and
  \begin{align*}
         & \|\nabla_{\mathcal{D}_m} v_{\mathcal{D}_m} - \nabla \overline{v}\|_{L^p(\Omega_T; L^p(\mathcal{O})^d)}
    =  \|\psi_m \nabla_{\mathcal{D}_m} \phi_{\mathcal{D}_m} - \psi \nabla \phi\|_{L^p(\Omega_T; L^p(\mathcal{O})^d)}                                                                           \\
    \leq & \|\psi_m (\nabla_{\mathcal{D}_m} \phi_{\mathcal{D}_m} - \nabla \phi)\|_{L^p(\Omega_T; L^p(\mathcal{O})^d)} + \|(\psi_m - \psi)\nabla \phi\|_{L^p(\Omega_T; L^p(\mathcal{O})^d)}     \\
    =    & \|\psi_m\|_{L^p(\Omega_T)} \|\nabla_{\mathcal{D}_m} \phi_{\mathcal{D}_m} - \nabla \phi\|_{L^p(\mathcal{O})} + \|\psi_m - \psi\|_{L^p(\Omega_T)} \|\nabla \phi\|_{L^p(\mathcal{O})} \\
    \leq    & \|\psi_m\|_{L^p(\Omega_T)} S_{\mathcal{D}_m}(\phi) + \|\psi_m - \psi\|_{L^p(\Omega_T)} \|\nabla \phi\|_{L^p(\mathcal{O})}.
  \end{align*}
  Then the convergence \eqref{eq-lem:GD interpolation strong cv} follows from, as \(m \rightarrow \infty\), 
  the strong convergence \(\psi_m \rightarrow \psi\) in \(L^{\max(2,p)}(\Omega_T)\) (see, e.g. by Lemma \ref{lem:strong cv piecewise approximation}),
  and the GD-consistency \(S_{\mathcal{D}_m}(\phi) \rightarrow 0\).
\end{proof}
\section{SPDE with multiplicative  Lipschitz continuous noise }\label{sec:LC}
In this section, we consider 
\begin{equation}
\label{model:sde LM}
    \begin{cases}
        d u = \operatorname{div} (\bm{a}(x, \nabla u)) dt  + B(u) dW_t & \text{in \((0,T) \times \mathcal{O}\)},          \\
        u = 0                                                               & \text{on \((0,T) \times \partial \mathcal{O}\)}, \\
        u(0, \cdot) = u_0                                                   & \text{in \(\mathcal{O}\)}.
    \end{cases}
\end{equation}
The stochastic perturbation in \eqref{model:sde LM} is of It\^o multiplicative type, 
where the Gaussian noise is multiplied by an operator depending on the current state solution, 
so that the magnitude of the perturbation is governed by the current state.
Perturbations of this form are a standard way of accounting for uncertainty in the forcing and in the data of an evolution model; 
we refer to \cite{da2014stochastic, liu2015stochastic} and the references therein.
In particular, in \eqref{model:sde LM}, 
\(B\) is a continuous function of its argument,
\(W\) is a noise of Gaussian type,
which are defined as the following.

\begin{assumption}\label{def:assumption for SPDE with continuous ito noise}
    Let \(U\) be a separable Hilbert space defined over the domain \(\mathcal{O}\).
    Let \(\mathcal{Q}\) be a bounded linear operator on \(U\) that is non-negative, symmetric, and of finite trace.
    Let \(L_2(U, L^2(\mathcal{O}))\) be a separable Hilbert space of Hilbert-Schmidt operators from \(U\) to \(L^2(\mathcal{O})\).
    Denote the subspaces  \(U_0 := \mathcal{Q}^{1/2}(U)\) and \(L_2^0 := L_2(U_0, L^2(\mathcal{O}))\).
    Assume the followings:
    \begin{enumerate}[align=left, leftmargin=*]
        \item  The \(U\)-valued process \(W = (W(t): t\in [0,T])\) is \(\mathcal{Q}\)-Wiener process that is adapted to \(\mathbb{F}\) and, for any \(0 \leq s \leq t \leq T\), \(W(t) - W(s)\) is independent of \(\mathcal{F}_s\).
        \item The operator \(B:  L^2(\mathcal{O})  \rightarrow L_2^0\) satisfies the following properties:
              \begin{itemize}[align=left, leftmargin=*]
                  \item (Growth) \(\| B(u)\|_{L_2^0}^2 \leq 1 + \| u\|_{L^2(\mathcal{O})}^2\) for all \(u \in L^2(\mathcal{O})\);
                  \item (Lipschitz continuity) \(\| B(u) - B(v)\|_{L_2^0}^2 \leq C_{\operatorname{Lip}} \| u - v\|_{L^2(\mathcal{O})}^2\) for all \(u, v \in L^2(\mathcal{O})\).
              \end{itemize}
    \end{enumerate}
\end{assumption}
We fix the definition of variational solution to \eqref{model:sde LM} in the following sense, see, for example, \cite{da2014stochastic, liu2015stochastic}.
\begin{definition}\label{def:solution SPDE with continuous ito noise}
A continuous \(\mathbb{F}\)-adapted \(L^2(\mathcal{O})\)-valued process \(u\) is a variational solution to \eqref{model:sde LM} 
if \(u \in L^2(\Omega_T; L^2(\mathcal{O})) \cap L^p(\Omega_T; W_0^{1,p}(\mathcal{O}))\)
with \(u_0 \in  L^2(\Omega; L^2(\mathcal{O})) \cap L^p(\Omega; W_0^{1,p}(\mathcal{O}))\),
where  \(p \in  (1, \infty)\),
and satisfies
\begin{equation*}
\mathbb{E} \Big[ 
\sup_{t \in (0,T)} \| u(t)\|_{L^2(\mathcal{O})}^2
\Big] 
< \infty,
\end{equation*}
and for every \(t \in [0,T]\) and \(\varphi \in L^2(\mathcal{O}) \cap W_0^{1, p}(\mathcal{O})\), 
\begin{equation}\label{eq-def:solution SPDE with continuous ito noise}
\begin{aligned}
\langle u(t), \varphi\rangle 
+ 
    \int_0^t \langle \bm{a}(x,\nabla  u), \nabla  \varphi \rangle \, ds 
= 
    \langle u_0, \varphi \rangle 
    + 
    \int_0^t \langle \varphi,  B(u) \, dW_s \rangle \quad \text{\(\mathbb{P}\)-almost surely.}
\end{aligned}
\end{equation}
\end{definition}

We consider the filtration \(\mathbb{F}_{\mathcal{D}} := (\mathcal{F}_{t^{(n)}})_{n = 1}^{N}\) where
\begin{equation*}
    \mathcal{F}_{t^{(n)} }:= \sigma (W(t^{(n')}): n' = 0, \ldots,  n )
\end{equation*}
and propose the following fully discretised scheme:
the flux term is implicit in time and the noise term is explicit in time which is natural and corresponds to the setting of stochastic Itô integral.
\begin{scheme}\label{def:GS SPDE with continuous ito noise}
    Take \(X_{\mathcal{D}}\)-valued and \(\mathcal{F}_0\)-measurable \(u_{\mathcal{D}}^{(0)} = \mathcal{I}_{\mathcal{D}} u_0: \Omega \rightarrow X_{\mathcal{D}}\).
    Find  \(u_{\mathcal{D}} := (u_{\mathcal{D}}^{(n)})_{n= 1}^N: \Omega \rightarrow X_{\mathcal{D}}^N\) such that
    \(u_{\mathcal{D}}\) is adapted to the filtration \(\mathbb{F}_{\mathcal{D}}\) and,
    for any \(\phi_{\mathcal{D}} \in X_{\mathcal{D}}\), and for \(n = 0, \ldots, N-1,\)
    \begin{equation}\label{model:GS Lip continuous noise}
        \langle \delta_{\mathcal{D}}^{(n+1)}u_{\mathcal{D}}, \Pi_{\mathcal{D}} \phi_{\mathcal{D}} \rangle  +   k_{\mathcal{D}} \langle \bm{a}(x,  \nabla _{\mathcal{D}}u_{\mathcal{D}}^{(n+1)}), \nabla _{\mathcal{D}} \phi_{\mathcal{D}} \rangle
        =  \langle
        B(\Pi_{\mathcal{D}} u_{\mathcal{D}}^{(n)}) \Delta^{(n+1)}W, \Pi_{\mathcal{D}} \phi_{\mathcal{D}}
        \rangle
    \end{equation}
    \(\mathbb{P}\)-almost surely,
    where
    \(\delta_{\mathcal{D}}^{(n+1)}u_{\mathcal{D}} := \Pi_{\mathcal{D}}(u_{\mathcal{D}}^{(n+1)} - u_{\mathcal{D}}^{(n)})\)
    and
    \(\Delta^{(n+1)} W := W{(t^{(n+1)})} - W{(t^{(n)})}\).
\end{scheme}

\subsection*{Convergence via compactness}
The theorem for convergence of the scheme to the variational formulation is stated as follows:
\begin{theorem}\label{thm:cv 1}
    Suppose that the flux \(\bm{a}\) satisfies the Assumption \ref{def:assumption opertator a} and the stochastic forcing \(B\) satisfies the Assumption \ref{def:assumption for SPDE with continuous ito noise}.
    Let \(((\mathcal{D}_T)_m)_{m \in \mathbb{N}}\) be a sequence of space-time gradient discretisations in the sense of Definition \ref{def:stoc GD} that are consistent and limit-conforming.
    For every \(m \in\mathbb{N}\), there exists a random variable \(u_{\mathcal{D}_m}\) that is a solution to the gradient scheme \ref{def:GS SPDE with continuous ito noise} with \(\mathcal{D}= \mathcal{D}_m\)
    and
    along a subsequence as \(m \rightarrow \infty\),
    \begin{align*}
         & \Pi_{\mathcal{D}_m} u_{\mathcal{D}_m}\rightarrow \overline{u}\text{ strongly in \(L^2( \Omega;L^2(\mathcal{O}))\) uniformly on \([0,T]\),} \\ 
         & \nabla _{\mathcal{D}_m} u_{\mathcal{D}_m} \rightarrow \nabla  \overline{u}  \text{ weakly in \(L^p( \Omega_T;  L^p(\mathcal{O})^d)\)},
    \end{align*}
    where \(\overline{u}\) is a solution to \eqref{model:sde LM} in the sense of Definition \ref{def:solution SPDE with continuous ito noise}. 
\end{theorem}
\begin{proof}
    See Section \ref{sec:proof-LC cv compactness}.
\end{proof}
\begin{theorem}\label{thm:cv 1.1}
    In addition to the assumptions of Theorem \ref{thm:cv 1},
    suppose that the flux function \(\bm{a}\) satisfies a strict monotonicity condition; that is,
    \[
    (\bm{a}(x, \bm{\mu}) - \bm{a}(x, \bm{\nu}))\cdot(\bm{\mu} - \bm{\nu}) > 0, 
    \quad 
    \text{for a.e. \(x\in \mathcal{O}\) and for all \(\bm{\mu} \neq \bm{\nu} \in \mathbb{R}^d,\)}
    \]
    then
    \begin{align*}
        \nabla _{\mathcal{D}_m} u_{\mathcal{D}_m} \rightarrow \nabla  \overline{u}  \text{ strongly in \(L^p( \Omega_T;  L^p(\mathcal{O})^d)\)},
    \end{align*}
\end{theorem}
\begin{proof}
    See Section \ref{sec:proof-LC cv compactness}.
\end{proof}
\begin{remark}
    The previous work \cite{droniou2022design} addresses a flux of the form \(\bm{a}(x, \overline{u}, \nabla \overline{u})\)
    and a continuous stochastic forcing, and establishes, using the theorems of Prokhorov and Skorokhod,
    the almost sure convergence of the gradient scheme towards a weak martingale solution;
    that convergence takes place on a new probability space, for processes having the same laws as the discrete solutions. 
    In addition to the current assumptions, 
    assume pathwise uniqueness holds, 
    the convergence of reconstructions of function and gradient on the original probability space, in probability in \(L^p(\Omega_T; L^p(\mathcal{O}))\), 
    can be deduced from \cite{droniou2022design}, by the Gyöngy-Krylov lemma (see, e.g. \cite[Section 2.10]{breit2018stochastically}).
    This argument, however, cannot lead to the strong convergence of function reconstruction in \(L^2(\Omega; L^2(\mathcal{O}))\) pointwise (or uniform) in time,
    since the time-continuity available in the martingale framework is with respect to the weak topology of \(L^2(\mathcal{O})\) only.
    Therefore, Theorems \ref{thm:cv 1} and \ref{thm:cv 1.1} do not follow from \cite{droniou2022design}.
\end{remark}
 
\subsection*{Convergence via error estimates}

In error analysis, convergence rates depend on interpolation properties, in space and in time.
In space we use an interpolator \(\mathcal{P}_{\mathcal{D}}: W_0^{1,p}(\mathcal{O}) \cap W_s \rightarrow X_{\mathcal{D}}\),
acting on the exact solution \(\overline{u}\) at each time step. In time, the scheme carries one value
per step, so the natural error functional compares \(\nabla_{\mathcal{D}} u_{\mathcal{D}}^{(n+1)}\)
with the value of \(\nabla \overline{u}\) at the corresponding node,
\begin{equation}\label{eq:metric}
    \mathbb{E} \bigg[ \sum_{n=0}^{N-1} k_{\mathcal{D}}
    \big\| \nabla_{\mathcal{D}} \mathcal{P}_{\mathcal{D}} \overline{u}(t^{(n+1)}) - \nabla \overline{u}(t^{(n+1)}) \big\|_{L^p(\mathcal{O})}^p \bigg].
\end{equation}
Assuming \(\nabla \overline{u} \in L^p(\Omega; W^{\gamma,p}(0,T;L^p(\mathcal{O})))\) with
\(\gamma \in (0,1/2)\), the quantity \eqref{eq:metric} is well defined when \(\gamma p > 1\), the
embedding \(W^{\gamma,p}(0,T;L^p(\mathcal{O})) \hookrightarrow C([0,T];L^p(\mathcal{O}))\) then
supplying the nodal values. When \(\gamma p \le 1\) it is not, and following \cite{breit2021space},
we replace the deterministic partition by a random time grid, so that \(\nabla \overline{u}\) is defined on those nodes.
Rather than carrying the two situations through the analysis separately, we isolate the single
property of the time grid that the argument uses. All statements below are then proved once, for an
arbitrary grid possessing that property.

\begin{definition}\label{def:admissible grid}
    Let \((\overline{\Omega}, \overline{\mathcal{F}}, \overline{\mathbb{P}})\) be a probability space
    with expectation \(\overline{\mathbb{E}}\), independent of \((\Omega,\mathcal{F},\mathbb{P})\) and
    of \(W\). A family \(\mathfrak{T} = (\mathfrak{t}^{(n)})_{n=0}^{N}\) of \([0,T]\)-valued random
    variables on \((\overline{\Omega}, \overline{\mathcal{F}}, \overline{\mathbb{P}})\) is an
    \emph{admissible time grid of mesh \(k_{\mathcal{D}} = T/N\) with regularity constants
        \(0 < c_- \le c_+ < \infty\)} if, \(\overline{\mathbb{P}}\)-almost surely,
    \(0 = \mathfrak{t}^{(0)} < \mathfrak{t}^{(1)} < \cdots < \mathfrak{t}^{(N)} \le T\) with
    \(T - \mathfrak{t}^{(N)} \le c_+ k_{\mathcal{D}}\), and the increments
    \(\kappa^{(n+1)} := \mathfrak{t}^{(n+1)} - \mathfrak{t}^{(n)}\) satisfy
    \begin{equation}\label{eq:kappa bounds}
        c_- k_{\mathcal{D}} \;\le\; \kappa^{(n+1)} \;\le\; c_+ k_{\mathcal{D}},
        \qquad n = 0, \ldots, N-1 .
    \end{equation}
    Denote
    \(\mathbb{E}_{\mathfrak{T}} := \overline{\mathbb{E}} \otimes \mathbb{E}\)
    and
    \(\mathcal{D}_{\mathfrak{T}} := (X_{\mathcal{D}}, \Pi_{\mathcal{D}}, \nabla_{\mathcal{D}}, \mathcal{I}_{\mathcal{D}}, \mathfrak{T})\).
\end{definition}

\begin{example}\label{ex:two grids}
    We consider the following two grids:
    \begin{enumerate}[label={(\alph*)}, leftmargin=*, align=left]
        \item \emph{(Deterministic grid \(\mathfrak{T}_{\mathrm{det}}\).)}
              \(\overline{\Omega}\) a singleton and \(\mathfrak{t}^{(n)} = t^{(n)}\).
              Then \(\kappa^{(n+1)} \equiv k_{\mathcal{D}}\), \(c_- = c_+ = 1\), and
              \(\mathbb{E}_{\mathfrak{T}} = \mathbb{E}\).
        \item \emph{(Random grid \(\mathfrak{T}_{\mathrm{ran}}\), \cite{breit2021space}.)}
              \(\mathfrak{t}^{(0)} = 0\) and \(\mathfrak{t}^{(1)}, \ldots, \mathfrak{t}^{(N)}\) are independent with
              uniformly distributed
              \(\mathfrak{t}^{(n)} \sim \mathcal{U} ([k_{\mathcal{D}}(n - 1/4),\, k_{\mathcal{D}}(n+1/4)])\)
              for \(n \le N-1\) and
              \(\mathfrak{t}^{(N)} \sim \mathcal{U}([T -  k_{\mathcal{D}}/4,\, T])\).
              Then \(c_- = 1/2\), \(c_+ = 3/2\).
    \end{enumerate}
\end{example} 

Fix \(p \in (1,\infty)\) and \(\gamma \in (0,1/2)\).
For a Banach space \(E\),
and 
\(I \subset (0,T)\), write
\[
    [g]_{W^{\gamma,p}(I;E)}^p
    := \int_{I} \int_{I} \frac{\|g(\xi) - g(\eta)\|_{E}^p}{|\xi - \eta|^{1 + \gamma p}} \, d\eta \, d\xi ,
    \qquad
    \|g\|_{W^{\gamma,p}(I;E)}^p := \|g\|_{L^p(I;E)}^p + [g]_{W^{\gamma,p}(I;E)}^p ,
\]
and set \(\mathbb{W}^{\gamma,p}(E) := L^p(\Omega; W^{\gamma,p}(0,T;E))\) with
\[
    \|g\|_{\mathbb{W}^{\gamma,p}(E)}^p := \mathbb{E}\big[\|g\|_{W^{\gamma,p}(0,T;E)}^p\big] .
\]
\begin{definition}\label{def:admissible grid interpolation stab and cons}
    Let \(\mathfrak{T}\) be an admissible time grid in the sense of Definition \ref{def:admissible grid}.
    Let \(\mathcal{J}_{\mathfrak{T}}\) be a piecewise constant interpolator in time.
    Then,
    \begin{itemize}[leftmargin=*, align=left]
        \item \(({\mathfrak{T}}, \mathcal{J}_{\mathfrak{T}})\) is  stable if there is \(C_{\mathrm{s}}\), independent of \(k_{\mathcal{D}}\), with
              \begin{equation}\label{eq:G1}
                  \mathbb{E}_{\mathfrak{T}}\bigg[ \sum_{n=0}^{N-1} \kappa^{(n+1)}
                  \| \mathcal{J}_{\mathfrak{T}} g  \|_{E}^p \bigg]
                  \;\le\; C_{\mathrm{s}} \, \| g \|_{\mathbb{W}^{\gamma,p}(E)}^p
                  \qquad \forall\, g \in \mathbb{W}^{\gamma,p}(E);
              \end{equation}
        \item \(({\mathfrak{T}}, \mathcal{J}_{\mathfrak{T}})\) is  consistent if there is \(C_{\mathrm{c}}\), independent of \(k_{\mathcal{D}}\), with
              \begin{equation}\label{eq:G2}
                  \mathbb{E}_{\mathfrak{T}}\bigg[ \sum_{n=0}^{N-1} \int_{\mathfrak{t}^{(n)}}^{\mathfrak{t}^{(n+1)}}
                  \| \mathcal{J}_{\mathfrak{T}} g  - g(s) \|_{E}^p \, ds \bigg]
                  \;\le\; C_{\mathrm{c}} \, k_{\mathcal{D}}^{\gamma p} \, \| g \|_{\mathbb{W}^{\gamma,p}(E)}^p
                  \qquad \forall\, g \in \mathbb{W}^{\gamma,p}(E).
              \end{equation}
    \end{itemize}
\end{definition}
\begin{proposition}\label{prop:grids satisfy G}
    Let \(\mathfrak{T}\) be an admissible time grid in the sense of Definition \ref{def:admissible grid}.
    Consider right-end point interpolation
    \[\forall n = 0, \ldots, N-1,
        \quad
        \forall t \in (t^{(n)}, t^{(n+1)}),
        \quad
        \mathcal{J}_{\mathfrak{T}}g (t)= g(\mathfrak{t}^{(n+1)}).\]
    If \(\gamma p > 1\), then \(({\mathfrak{T}} = \mathfrak{T}_{\operatorname{det}}, \mathcal{J}_{\mathfrak{T}})\) is stable and consistent.
    If \(\gamma p > 0\), then \(({\mathfrak{T}} = \mathfrak{T}_{\operatorname{ran}}, \mathcal{J}_{\mathfrak{T}})\) is stable and consistent.
\end{proposition}
\begin{proof}
    See Section \ref{sec:proof of prop:grids satisfy G}.
\end{proof}

For the flux term, in addition to Assumption \ref{def:assumption opertator a}, we require
Lipschitz continuity and a stronger form of monotonicity.

\begin{assumption}\label{def:assumption opertator a 2}
    In addition to Assumption \ref{def:assumption opertator a}, the operator \(\bm{a}\) satisfies, for
    \(p \in (1,\infty)\):
    \begin{enumerate}[label=(\roman*), ref=\roman*, align=left, leftmargin=*]
        \item \(\bm{a}(\cdot,\bm{0}) \in L^{p'}(\mathcal{O})^d\) and there exists
              \(\alpha_{\bm{a}} \in (0,\infty)\) such that
              \[|\bm{a}(\cdot,\bm{\mu}) - \bm{a}(\cdot,\bm{0})| \le \alpha_{\bm{a}} |\bm{\mu}|^{p-1}\]
              for almost every \(x \in \mathcal{O}\) and all \(\bm{\mu} \in \mathbb{R}^d\).
        \item There exists \(\beta_{\bm{a}} \in (0,\infty)\) such that
              \[|\bm{a}(x,\bm{\mu}) - \bm{a}(x,\bm{\nu})| \le \beta_{\bm{a}} |\bm{\mu}-\bm{\nu}|(|\bm{\mu}|^{p-2} + |\bm{\nu}|^{p-2})\]
              for almost every \(x \in \mathcal{O}\) and all \((\bm{\mu},\bm{\nu})\).
        \item There exists \(\gamma_{\bm{a}} \in (0,\infty)\) such that
              \[[\bm{a}(x,\bm{\mu}) - \bm{a}(x,\bm{\nu})]\cdot[\bm{\mu}-\bm{\nu}] \ge \gamma_{\bm{a}} |\bm{\mu}-\bm{\nu}|^2 (|\bm{\mu}|+|\bm{\nu}|)^{p-2}\]
              for almost every \(x \in \mathcal{O}\) and all \((\bm{\mu},\bm{\nu})\).
    \end{enumerate}
\end{assumption}

\begin{theorem}\label{thm:LC noise error}
    Let \(p \in (1,\infty)\) and \(\gamma \in (0,1/2)\). Suppose the flux \(\bm{a}\) satisfies
    Assumption \ref{def:assumption opertator a 2} and the stochastic forcing \(B\) satisfies
    Assumption \ref{def:assumption for SPDE with continuous ito noise}.
    Let \(\mathcal{D}_T\) be a space-time gradient discretisation in the sense of Definition
    \ref{def:stoc GD} that is limit-conforming and consistent, let \(E_s\), \(\bm{E}_w\) and the space size
    \(h_{\mathcal{D}} := h_{\mathcal{D}}(E_s;\bm{E}_w)\) be given by Definition
    \ref{def:GD space size}, and let \(h_{\mathcal{M}}\) be the mesh size of the polytopal mesh of
    \(\mathcal{O}\). 
    Let \(\mathfrak{T}\) be an admissible time grid in the sense of Definitions \ref{def:admissible grid}
    and \(({\mathfrak{T}}, \mathcal{J}_{\mathfrak{T}})\) be consistent in the sense of Definition \ref{def:admissible grid interpolation stab and cons}.
    Let \(u_{\mathcal{D}}\) solve the gradient scheme \ref{def:GS SPDE with continuous ito noise} on
    \(\mathfrak{T}\) (with \(k_{\mathcal{D}}\) replaced by \(\kappa^{(n+1)}\)).
    Suppose that the following assumptions hold:
    \begin{enumerate}[label={(H\arabic*)}, leftmargin=*, align=left]
        \item The domain \(\mathcal{O} \subset \mathbb{R}^d\) \((d \ge 1)\) is open, bounded,
              connected and polytopal.
        \item The exact solution satisfies
              \begin{equation*}
                  \overline{u} \in L^2\big(\Omega; C^{\gamma}([0,T];L^2(\mathcal{O}))\big)
                  \cap L^2\big(\Omega_T; L^2(\mathcal{O}) \cap E_s\big)
                  \cap L^p\big(\Omega_T; W_0^{1,p}(\mathcal{O}) \cap E_s\big),
              \end{equation*}
              together with
              \begin{equation*}
                  \nabla \overline{u} \in L^p\big(\Omega; W^{\gamma,p}(0,T;L^p(\mathcal{O}))\big),
                  \qquad
                  \bm{a}(x,\nabla \overline{u}) \in L^{p'}(\Omega_T; \bm{E}_w).
              \end{equation*} 
        \item There exist \(\alpha > 0\) and \(C_{\mathcal{P}} > 0\), independent of \(\mathcal{D}\),
              and a spatial interpolator
              \(\mathcal{P}_{\mathcal{D}} : W_0^{1,p}(\mathcal{O}) \cap W_s \rightarrow X_{\mathcal{D}}\)
              such that, for all \(\phi \in W_0^{1,p}(\mathcal{O}) \cap W_s\),
              \begin{equation}\label{eq:H3}
                  \| \Pi_{\mathcal{D}} \mathcal{P}_{\mathcal{D}}\phi - \phi \|_{L^2(\mathcal{O})}
                  + \alpha \| \nabla_{\mathcal{D}} \mathcal{P}_{\mathcal{D}}\phi - \nabla\phi \|_{L^p(\mathcal{O})}
                  \leq C_{\mathcal{P}}  \alpha  h_{\mathcal{D}}\|\phi\|_{W_s}.
              \end{equation}
        \item There exists \(C > 0\), depending only on the regularity factors of the mesh, with
              \(h_{\mathcal{D}} \leq C h_{\mathcal{M}}\).
    \end{enumerate}
    Denote
    interpolation error of the initial condition by
    \begin{equation*}
        e_{\mathcal{D}}^{\mathrm{ini}} := \| \Pi_{\mathcal{D}} u_{\mathcal{D}}^{(0)} - \Pi_{\mathcal{D}} \mathcal{P}_{\mathcal{D}} \overline{u}_0 \|_{L^2(\Omega;L^2(\mathcal{O}))},
    \end{equation*}
    and set
    \(\mathfrak{E} := e_{\mathcal{D}}^{\mathrm{ini}} + k_{\mathcal{D}}^{-\frac12} \alpha h_{\mathcal{D}}\),
    \begin{align*}
        \Theta_{1}                    & := \mathfrak{E} + h_{\mathcal{D}}^{\,p-1} + k_{\mathcal{D}}^{\,\gamma(p-1)},
                                      &
        \Theta_{\Pi_{\mathcal{D}}}    & := \mathfrak{E} + h_{\mathcal{D}}^{\,\frac{p'}{2}} + k_{\mathcal{D}}^{\,\frac{p'\gamma}{2}},
                                      &
        \Theta_{\nabla_{\mathcal{D}}} & := \mathfrak{E}^{\frac{2}{p}} + h_{\mathcal{D}}^{\,\frac{1}{p-1}} + k_{\mathcal{D}}^{\,\frac{\gamma}{p-1}} .
    \end{align*}
    Then there is \(C > 0\) independent of \( e_{\mathcal{D}}^{\operatorname{ini}}\), \(h_{\mathcal{D}}\), \(h_{\mathcal{M}}\), and \(k_{\mathcal{D}}\),
    such that
    \begin{align*}
        \mathbb{E}_{\mathfrak{T}} \bigg[ \sum_{n=0}^{N-1} \kappa^{(n+1)}
        \big\| \Pi_{\mathcal{D}} u_{\mathcal{D}}^{(n+1)} - \overline{u}(\mathfrak{t}^{(n+1)}) \big\|_{L^2(\mathcal{O})}^2 \bigg]^{\frac12}
         & \;\le\;
        \begin{cases}
            C\, \Theta_{1}                 & \text{if } 1 < p \le 2, \\
            C\, \Theta_{\Pi_{\mathcal{D}}} & \text{if } p > 2 ,
        \end{cases}
        \\[6pt]
        \mathbb{E}_{\mathfrak{T}} \bigg[ \sum_{n=0}^{N-1} \kappa^{(n+1)}
        \big\| \nabla_{\mathcal{D}} u_{\mathcal{D}}^{(n+1)} - \nabla \overline{u}(\mathfrak{t}^{(n+1)}) \big\|_{L^p(\mathcal{O})}^p \bigg]^{\frac1p}
         & \;\le\;
        \begin{cases}
            C\, \Theta_{1}                    & \text{if } 1 < p \le 2, \\
            C\, \Theta_{\nabla_{\mathcal{D}}} & \text{if } p > 2 .
        \end{cases}
    \end{align*}
\end{theorem}
\begin{proof}
    See Section \ref{sec:proof-LC cv error}.
\end{proof} 
\begin{remark}\label{rem:where the averaging sits}
    Randomising the time grid and averaging the scheme are two tools for purpose of relaxation of regularity assumption of the exact solution,
    and they essentially differ in where the averaging is considered.
    The study  \cite{diening2023averaged} uses a time-average algorithm that is driven by averaged Wiener increments,
    their discrete unknowns approximate the time-averaged of the exact solution rather than its nodal values.
    Their averaging is considered on the numerical scheme.
    In our work, 
    we adapted the strategy in \cite{breit2021space}
    and consider the same scheme that we used to establish convergence by compactness, while the nodes are random,
    our averaging is considered when we compare the exact solution and approximate solution.
    Applying the techniques in \cite{diening2023averaged} for nonconforming schemes remains an interesting topic for future work. 
\end{remark}

\section{SPDE with transport noise }\label{sec:transport}
In this section, we consider 
\begin{equation}
\label{model:sde gradient}
    \begin{cases}
        d u = \operatorname{div} (\bm{a}(x, \nabla u)) dt  +  \frac{1}{2} \, \bm{\sigma} \cdot \nabla  (\bm{\sigma} \cdot \nabla  \overline{u}) \, dt +  (\bm{\sigma} \cdot \nabla   \overline{u} ) \, d\beta_t  & \text{in \((0,T) \times \mathcal{O}\)},          \\
        u = 0                                                               & \text{on \((0,T) \times \partial \mathcal{O}\)}, \\
        u(0, \cdot) = u_0                                                   & \text{in \(\mathcal{O}\)},
    \end{cases}
\end{equation}
where 
\(\bm{\sigma} \in C^1(\mathcal{O})\) with \(\text{div} \, \bm{\sigma} = 0\).

By the Itô-Stratonovich correction (see, e.g.  \cite{barbu2018stochastic}),
we can obtain 
\begin{equation*}
    (\bm{\sigma} \cdot \nabla   \overline{u} ) \circ \, d\beta_t 
    =
    \frac{1}{2} \, \bm{\sigma} \cdot \nabla  (\bm{\sigma} \cdot \nabla  \overline{u}) \, dt +  (\bm{\sigma} \cdot \nabla   \overline{u} ) \, d\beta_t.
\end{equation*}
Model \eqref{model:sde gradient} can be written in the Stratonovich form as 
\begin{subequations} \label{model:sde gradient 2}
    \begin{alignat}{3}
        d\overline{u} = \,          & \operatorname{div} (\textbf{a} (x, \nabla  \overline{u})) \, dt +   (\bm{\sigma} \cdot \nabla   \overline{u} ) \circ \, d\beta_t. &  & \quad \text{on \((0,T) \times \mathcal{O}\)},            \\
        \overline{u} = \,           & 0                                                                                          &  & \quad  \text{on \( (0,T) \times \partial \mathcal{O}\)}, \\
        \overline{u}(0, \cdot) = \, & u_0                                                                                        &  & \quad \text{in \(\mathcal{O}\)}.
    \end{alignat}
\end{subequations} 
The stochastic perturbation in \eqref{model:sde gradient 2} is of transport type,
which models the advection of \(\overline{u}\) by a randomly fluctuating velocity field.
Stochastic models of this kind have been derived
and analysed from several viewpoints in the fluid dynamics literature, for which we refer to
\cite{flandoli2011random, holm2015variational, flandoli2022additive} and the references therein.

By the zero-divergence property of \(\bm{\sigma}\), we can write 
\begin{equation*}
     \bm{\sigma} \cdot \nabla  (\bm{\sigma} \cdot \nabla  \overline{u}) 
     = 
     \operatorname{div} (\bm{\sigma} (\bm{\sigma} \cdot \nabla  \overline{u}) ).
\end{equation*}
Therefore, we consider a solution in the following sense:
\begin{definition}\label{def:solution SPDE with grad noise}
     Let \(p \in [2, \infty)\) and let \(u_0 \in L^2(\Omega, \mathcal{F}_0; L^2(\mathcal{O}))\).
     A continuous \(\mathbb{F}\)-adapted \(L^2(\mathcal{O})\)-valued process \(u\) is a variational solution to \eqref{model:sde gradient}
     if \(u \in L^2(\Omega_T; L^2(\mathcal{O})) \cap L^p(\Omega_T; W_0^{1,p}(\mathcal{O}))\)
     with \(u_0 \in  L^2(\Omega; L^2(\mathcal{O})) \cap L^p(\Omega; W_0^{1,p}(\mathcal{O}))\) 
     and satisfies
     \begin{equation*}
          \mathbb{E} \Big[
          \sup_{t \in (0,T)} \| u(t)\|_{L^2(\mathcal{O})}^2
          \Big]
          < \infty,
     \end{equation*}
     and
     for every \(t \in [0,T]\) and \(\varphi \in W_0^{1, p}(\mathcal{O})\),
     \begin{equation}\label{eq-def:solution SPDE with grad noise}
          \begin{aligned}
                    & \langle u(t), \varphi\rangle
               +
               \int_0^t \langle \bm{a}(x, \nabla  u), \nabla  \varphi \rangle
               +
               \frac{1}{2} \langle \bm{\sigma} \cdot \nabla  u, \bm{\sigma} \cdot \nabla  \varphi \rangle \, dt \\
               = \, &
               \langle u_0, \varphi \rangle
               +
               \int_0^t \langle \varphi , (\bm{\sigma} \cdot \nabla ) \, u \, d\beta_t \rangle
               \quad \text{\(\mathbb{P}\)-almost surely.}
          \end{aligned}
     \end{equation}
\end{definition}

We consider the filtration \(\mathbb{F}_{\mathcal{D}} := (\mathcal{F}_{t^{(n)}})_{n = 1}^{N}\) where
\begin{equation*}
    \mathcal{F}_{t^{(n)} }:= \sigma (\beta(t^{(n')}): n' = 0, \ldots,  n )
\end{equation*}
and propose the following fully discretised scheme:
the flux term and the Itô correction term are implicit in time, the noise term is explicit in time which corresponds to definition of the stochastic Itô integral.
\begin{scheme}\label{def:GS SPDE with transport noise}
    Take \(X_{\mathcal{D}}\)-valued and \(\mathcal{F}_0\)-measurable \(u_{\mathcal{D}}^{(0)} = \mathcal{I}_{\mathcal{D}} u_0: \Omega \rightarrow X_{\mathcal{D}}\).
    Find  \(u_{\mathcal{D}} = (u_{\mathcal{D}}^{(n)})_{n= 1}^N: \Omega \rightarrow X_{\mathcal{D}}^N\) such that
    \(u_{\mathcal{D}}\) is adapted to the filtration \(\mathbb{F}_{\mathcal{D}}\) and,
    for any \(\phi_{\mathcal{D}} \in X_{\mathcal{D}}\),  and for \(n = 0, \ldots, N-1,\)
    \begin{equation}\label{model:GS transport noise}
        \begin{aligned}
                 & \langle \delta_{\mathcal{D}}^{(n+1)}u_{\mathcal{D}}, \Pi_{\mathcal{D}} \phi_{\mathcal{D}} \rangle
            +
            k_{\mathcal{D}} \langle \bm{a}(x, \nabla _{\mathcal{D}}u_{\mathcal{D}}^{(n+1)}), \nabla _{\mathcal{D}} \phi_{\mathcal{D}} \rangle
            \\
            = \, & 
            -
             \frac{1}{2} k_{\mathcal{D}} \langle \bm{\sigma} \cdot \nabla _{\mathcal{D}} u_{\mathcal{D}}^{(n+1)}, \bm{\sigma} \cdot \nabla _{\mathcal{D}} \phi_{\mathcal{D}} \rangle 
             +
            \langle
            \bm{\sigma} \cdot \nabla_{\mathcal{D}} u_{\mathcal{D}}^{(n)} \Delta^{(n+1)}\beta, \Pi_{\mathcal{D}} \phi_{\mathcal{D}}
            \rangle
        \end{aligned}
    \end{equation}
    \(\mathbb{P}\)-almost surely,
    where
    \(\delta_{\mathcal{D}}^{(n+1)}u_{\mathcal{D}} := \Pi_{\mathcal{D}}(u_{\mathcal{D}}^{(n+1)} - u_{\mathcal{D}}^{(n)})\)
    and
    \(\Delta^{(n+1)} \beta := \beta{(t^{(n+1)})} - \beta{(t^{(n)})}\).
\end{scheme}

\begin{theorem}\label{thm:cv 2}
    Suppose that the flux \(\bm{a}\) satisfies the Assumption \ref{def:assumption opertator a} with \(p \geq 2\)
    and  \(\bm{\sigma} \in C^1(\mathcal{O})\) with \(\text{div} \, \bm{\sigma} = 0\).
    Let \(((\mathcal{D}_T)_m)_{m \in \mathbb{N}}\) be a sequence of gradient discretisations that is consistent, limit-conforming, and coercive.
    For every \(m \in\mathbb{N}\), there exists a random variable \(u_{\mathcal{D}_m}\) that is a solution to the gradient scheme \ref{def:GS SPDE with transport noise} with \(\mathcal{D}= \mathcal{D}_m\)
    and
    along a subsequence as \(m \rightarrow \infty\),
    \begin{equation*}
        \begin{aligned}
             & \Pi_{\mathcal{D}_m} u_{\mathcal{D}_m} \rightarrow \overline{u} \text{ weakly in \(L^2( \Omega_T;L^2(\mathcal{O}))\),}                  \\
             & \Pi_{\mathcal{D}_m} u_{\mathcal{D}_m}(T) \rightarrow \overline{u}(T) \text{ weakly in \(L^2( \Omega;L^2(\mathcal{O}))\),}              \\
             & \nabla _{\mathcal{D}_m} u_{\mathcal{D}_m} \rightarrow \nabla  \overline{u}  \text{ weakly in \(L^p( \Omega_T;  L^p(\mathcal{O}))^d\)},
        \end{aligned}
    \end{equation*}
    where \(\overline{u}\) is a solution to \eqref{model:sde gradient} in the sense of  Definition \ref{def:solution SPDE with grad noise}.
\end{theorem}
\begin{proof}
    See Section \ref{sec:proof-transport cv compactness}.
\end{proof}
\section{Proofs for Theorem \ref{thm:cv 1} and Theorem \ref{thm:cv 1.1}}\label{sec:proof-LC cv compactness}
\subsection{A Priori estimates}
In this section, we analyse the gradient scheme \ref{def:GS SPDE with continuous ito noise}
\begin{lemma}\label{lem:existence of solution lip conti}
    Suppose that the flux \(\bm{a}\) satisfies Assumption \ref{def:assumption opertator a}.
    There exists at least one \(\mathbb{F}_{\mathcal{D}}\)-adapted process \(u_{\mathcal{D}}\) that satisfies the gradient scheme \ref{def:GS SPDE with continuous ito noise}.
\end{lemma}
\begin{proof}
    Endow \( X_{\mathcal{D}}\) with product \((\cdot, \cdot)\) and denote \(|\cdot|\) the corresponding norm.
    Let \(F^{(n+1)}\) map \(X_{\mathcal{D}}\)-valued random variable to \(X_{\mathcal{D}}\)-valued random variable and such that, for all \(X_{\mathcal{D}}\)-valued random variable \(v_{\mathcal{D}}\),
    for all \(\omega \in \Omega\), for all \(\phi_{\mathcal{D}} \in X_{\mathcal{D}}\),
    \begin{equation*}
        (F^{(n+1)}(v_{\mathcal{D}}), \phi_{\mathcal{D}})
        :=
        \langle \Pi_{\mathcal{D}}v_{\mathcal{D}}, \Pi_{\mathcal{D}}\phi_{\mathcal{D}} \rangle
        +
        k_{\mathcal{D}} \langle \bm{a}(x,  \nabla_{\mathcal{D}} v_{\mathcal{D}}) , \nabla_{\mathcal{D}} \phi_{\mathcal{D}} \rangle.
    \end{equation*}
    Set \(b^{(n+1)}: \Omega \rightarrow X_{\mathcal{D}}\) such that for all \(\omega \in \Omega\), for all \(\phi_{\mathcal{D}}\in X_{\mathcal{D}}\)
    \begin{equation*}
        (b^{(n+1)},\phi_{\mathcal{D}}) =
        \langle \Pi_{\mathcal{D}}u_{\mathcal{D}}^{(n)}, \Pi_{\mathcal{D}}\phi_{\mathcal{D}} \rangle +  \langle B(\Pi_{\mathcal{D}} u_{\mathcal{D}}^{(n) }) \Delta^{(n+1)}W , \Pi_{\mathcal{D}}\phi_{\mathcal{D}} \rangle
    \end{equation*}
    By the coercivity property of \(\bm{a}\) and the equivalence of norms in finite dimensional space,
    there exists a constant \(C\) such that
    \begin{align*}
        (F^{(n+1)}(v_{\mathcal{D}}), v_{\mathcal{D}}) 
        \geq
        \underline{a} k_{\mathcal{D}}\| \nabla_{\mathcal{D}} v_{\mathcal{D}} \|_{L^p(\mathcal{O})}^p \geq C |v_{\mathcal{D}} |^p.
    \end{align*}
    This implies that, for every \(\omega \in \Omega\), \((F^{(n+1)}(v_{\mathcal{D}}),  v_{\mathcal{D}} )/ |v_{\mathcal{D}} | \rightarrow \infty\) as \(|v_{\mathcal{D}}| \rightarrow \infty\).
    By \cite[Theorem 3.1]{Deimling1985NonlinearFA}, for every \(\omega \in \Omega\), \(F^{(n+1)}\) is surjective; that is, for any \(b^{(n+1)} \in X_{\mathcal{D}}\),
    there exists an \(u \in X_{\mathcal{D}}\) such that \(F^{(n+1)}(u) = b^{(n+1)}\).
    Hence, the gradient scheme \ref{def:GS SPDE with continuous ito noise} has at least one solution.

    It remains to show that there exists a \(\mathbb{F}_{\mathcal{D}}\)-adapted \(u_{\mathcal{D}}\).
    Define \(\mathcal{G}_n:X_{\mathcal{D}} \times X_{\mathcal{D}} \times \mathbb{R} \rightarrow X_{\mathcal{D}}\) via
    \begin{equation*}
        \mathcal{G}_n(\phi, u_{\mathcal{D}}^{(n)}, W)(\omega) = \mathcal{G}_n(\phi(\omega), u_{\mathcal{D}}^{(n)}(\omega), W(\omega)) := F^{(n+1)}(\phi(\omega)) - b(\omega)
    \end{equation*}
    Since there exists at least one solution to \eqref{model:GS Lip continuous noise} for every \(\omega\), then for each \(u_{\mathcal{D}}^{(n)}(\omega) \in X_{\mathcal{D}}\),
    there exists \(u_{\mathcal{D}}^{(n+1)}(\omega) \in X_{\mathcal{D}}\) such that \(\mathcal{G}_n(u_{\mathcal{D}}^{(n+1)}(\omega), u_{\mathcal{D}}^{(n)}(\omega), W(\omega)) = 0.\)

    Define a mapping \(\Gamma: X_{\mathcal{D}} \times \mathbb{R} \rightarrow \text{Power}(X_{\mathcal{D}})\) via
    \begin{equation*}
        (u_{\mathcal{D}}^{(n)}, W) \mapsto \{\phi \in X_{\mathcal{D}}: \mathcal{G}_n(\phi, u_{\mathcal{D}}^{(n)}, W)= 0\},
    \end{equation*}
    where \(\text{Power} (X_{\mathcal{D}})\) is the power set of \(X_{\mathcal{D}}\) which contains all subsets of \(X_{\mathcal{D}}\).
    By assumption,
    \(|\phi|\) has an upper bound that depend on \(k_{\mathcal{D}}\) and \(u_{\mathcal{D}}^{(n)}\), for any \(\phi \in \Gamma(u_{\mathcal{D}}^{(n)}, W)\).
    Moreover, since \(F^{(n+1)}\) is continuous, the set \(\Gamma(u_{\mathcal{D}}^{(n)}, W)\) is closed for any \((u_{\mathcal{D}}^{(n)}, W) \in X_{\mathcal{D}} \times \mathbb{R}\).
    And so, the mapping \(\Gamma\) is well-defined and its values lie in closed and bounded subsets of \(X_{\mathcal{D}}\).
    By \cite[Theorem 12]{jureczko2023remarks}, the mapping \(\Gamma\) admits a Borel-measurable selector \(\mathcal{S}_n: X_{\mathcal{D}} \times \mathbb{R} \rightarrow X_{\mathcal{D}}\)
    such that,
    \begin{equation*}
        \mathcal{S}_n(\psi,W) \in \Gamma(\psi, W), \quad \forall \psi \in X_{\mathcal{D}}.
    \end{equation*}

    Therefore, given \(F_{t^{(n)}}\)-measurable \(u_{\mathcal{D}}^{(n)}\),
    and since \(\Delta^{(n+1)}W\) is \(F_{t^{(n+1)}}\)-measurable, by measurability properties of \(\mathcal{S}_{n}\),
    we have \(S_n(u_{\mathcal{D}}^{(n)}, \Delta^{(n+1)}W)\) is \(F_{t^{(n+1)}}\)-measurable.
    We then construct a sequence \((u_{\mathcal{D}}^{(n)})_{n = 1}^{N}\) by
    \begin{equation*}
        u_{\mathcal{D}}^{(n)}:= \mathcal{S}_{n-1}(u_{\mathcal{D}}^{(n-1)},\Delta^{(n)}W), \quad  n = 1, \ldots, N.
    \end{equation*}
    Thus \(u_{\mathcal{D}}\) is \(\mathbb{F}_{\mathcal{D}}\)-adapted.
\end{proof}

\begin{lemma}\label{lem:continuous noise apriori estimate}
    Given \(\mathcal{F}_0\)-measurable \(u^{(0)}:\Omega \rightarrow X_{\mathcal{D}}\),
    every solution to the gradient scheme \ref{def:GS SPDE with continuous ito noise} satisfies:
    there exists constants \(C(T,  u_{\mathcal{D}}^{(0)}) > 0\),
    \begin{equation}\label{eq-lem:continuous noise apriori estimate}
        \mathbb{E}\bigg[\max_{1\leq n \leq N} \|\Pi_{\mathcal{D}} u_{\mathcal{D}}^{(n)}\|_{L^2(\mathcal{O})}^2  + \sum_{n=0}^{N-1} k_{\mathcal{D}} \|\nabla_{\mathcal{D}}u_{\mathcal{D}}^{(n+1)}\|_{L^p(\mathcal{O})}^p \bigg]
        \leq
        C(T, u_{\mathcal{D}}^{(0)}).
    \end{equation}
    And there exists a constant \(C\) depending on \(\mathcal{O}\), \(T\), \(\overline{\bm{a}}\), \(\lambda\), and \(u_{\mathcal{D}}^{(0)}\) such that
    \begin{equation}\label{eq-lem:continuous noise apriori estimate time translate}
        \mathbb{E}
        \bigg[
        \sum_{n = 1}^{N - 1} k_{\mathcal{D}} \| \Pi_{\mathcal{D}} u^{(n + 1)}  - \Pi_{\mathcal{D}}u^{(n)}\|_{L^2(\mathcal{O})}^2
        \bigg]
        \leq
        C k_{\mathcal{D}}.
    \end{equation}
\end{lemma}

\begin{proof}[Proof of \eqref{eq-lem:continuous noise apriori estimate}]
    Let \(\phi_{\mathcal{D}} = u_{\mathcal{D}}^{(n+1)}\) in \eqref{def:GS SPDE with continuous ito noise}, and use the relation
    \begin{equation}\label{eq:relation 1}
        2(a-b)a =  (a^2 - b^2) +  (a-b)^2 \quad \forall a,b \in \mathbb{R}
    \end{equation}
    to write
    \begin{equation}\label{eq-proof-lem:GS abstract noise apriori 1}
        \begin{aligned}
              & \frac{1}{2}\| \Pi_{\mathcal{D}}u_{\mathcal{D}}^{(n+1)}\|_{L^2(\mathcal{O})}^2
            +
            \frac{1}{2}\| \Pi_{\mathcal{D}}u_{\mathcal{D}}^{(n+1)}
            -
            \Pi_{\mathcal{D}}u_{\mathcal{D}}^{(n)}\|_{L^2(\mathcal{O})}^2
            +
            k_{\mathcal{D}} \langle \bm{a} (x, \nabla_{\mathcal{D}} u_{\mathcal{D}}^{(n+1)}), \nabla_{\mathcal{D}}u_{\mathcal{D}}^{(n+1)} \rangle \\
            = & \,
            \frac{1}{2} \| \Pi_{\mathcal{D}}u_{\mathcal{D}}^{(n)}\|_{L^2(\mathcal{O})}^2
            +
            \langle  B(\Pi_{\mathcal{D}} u_{\mathcal{D}}^{(n)}) \Delta^{(n+1)}W, \delta_{\mathcal{D}}^{(n+1)} u_{\mathcal{D}} \rangle
            +
            \langle  B(\Pi_{\mathcal{D}} u_{\mathcal{D}}^{(n)}) \Delta^{(n+1)}W, \Pi_{\mathcal{D}} u_{\mathcal{D}}^{(n)}\rangle.
        \end{aligned}
    \end{equation}
    Taking the sum from \(n = 0\) to \(n = \ell \in \{0, \ldots, N-1\}\) and the expectation,
    by the coercivity property of \(\bm{a}\),
    we obtain
    \begin{align*}
             & \mathbb{E} \big[\| \Pi_{\mathcal{D}}u_{\mathcal{D}}^{(\ell+1)}\|_{L^2(\mathcal{O})}^2 \big]
        + 2\mathbb{E} \bigg[\sum_{n=0}^{\ell} k_{\mathcal{D}} \|\nabla_{\mathcal{D}}u_{\mathcal{D}}^{(n+1)}\|_{L^p(\mathcal{O})}^p\bigg]                                           \\ 
        \leq & \,
        \mathbb{E} \big[\| \Pi_{\mathcal{D}}u_{\mathcal{D}}^{(0)}\|_{L^2(\mathcal{O})}^2 \big]
        +
        \mathbb{E} \bigg[\sum_{n = 0}^{\ell}\langle  B(\Pi_{\mathcal{D}} u_{\mathcal{D}}^{(n)}) \Delta^{(n+1)}W, \delta_{\mathcal{D}}^{(n+1)} u_{\mathcal{D}} \rangle\bigg]        \\
             & \, + \mathbb{E} \bigg[\sum_{n = 0}^{\ell}\langle  B(\Pi_{\mathcal{D}} u_{\mathcal{D}}^{(n)}) \Delta^{(n+1)}W, \Pi_{\mathcal{D}} u_{\mathcal{D}}^{(n)}\rangle\bigg].
    \end{align*}

    The third term on the right-hand side vanishes since \(u_{\mathcal{D}}^{(n)}\) is \(\mathcal{F}_{t^{(n)}}\)-measurable and thus is independent of \(\Delta^{(n+1)}W\); hence the discrete stochastic integral is a martingale.

    Now we estimate the second term on the right-hand side.
    Applying the Cauchy-Schwarz inequality and the Young inequality to the term,
    we obtain for any \(\xi > 0\),
    \begin{align*}
        \mathbb{E} \bigg[\sum_{n = 0}^{\ell}\langle  B(\Pi_{\mathcal{D}} u_{\mathcal{D}}^{(n)}) \Delta^{(n+1)}W, \delta_{\mathcal{D}}^{(n+1)} u_{\mathcal{D}} \rangle\bigg]
        \leq \, &
        \mathbb{E} \bigg[\frac{\xi}{2}\sum_{n = 0}^{\ell} \| B(\Pi_{\mathcal{D}} u_{\mathcal{D}}^{(n)})  \Delta^{(n+1)}W \|_{L^2(\mathcal{O})}^2\bigg] \\
                & +
        \mathbb{E} \bigg[\frac{1}{2\xi}\sum_{n = 0}^{\ell} \|\delta_{\mathcal{D}}^{(n+1)} \phi_{\mathcal{D}}\|_{L^2(\mathcal{O})}^2\bigg],
    \end{align*}
    where the first can be estimated using the Itô isometry (see, e.g. \cite{da2014stochastic})
    \begin{equation}\label{eq-proof-lem:continuous noise analysis 1 1}
        \mathbb{E}
        \big[
        \| B(\Pi_{\mathcal{D}} u_{\mathcal{D}}^{(n)})  \Delta^{(n+1)}W \|_{L^2(\mathcal{O})}^2
        \big]
        =
        \mathbb{E} \big[ k_{\mathcal{D}} \| B(\Pi_{\mathcal{D}} u_{\mathcal{D}}^{(n)}) \|_{L_2^0}^2 \big]. 
    \end{equation}
    Then by the boundedness of \(B\),
    we have
    \begin{align*}
                & \mathbb{E} \big[\| \Pi_{\mathcal{D}}u_{\mathcal{D}}^{(\ell+1)}\|_{L^2(\mathcal{O})}^2 \big]
        +
        2\mathbb{E} \bigg[\sum_{n=0}^{\ell} k_{\mathcal{D}} \|\nabla_{\mathcal{D}}u_{\mathcal{D}}^{(n+1)}\|_{L^p(\mathcal{O})}^p\bigg] \\ 
        \leq \, &
        \mathbb{E} \big[\| \Pi_{\mathcal{D}}u_{\mathcal{D}}^{(0)}\|_{L^2(\mathcal{O})}^2 \big]
        +
        \mathbb{E}\bigg[T +\sum_{n = 0}^{\ell} k_{\mathcal{D}}\|\Pi_{\mathcal{D}} u_{\mathcal{D}}^{(n)}\|_{L^2(\mathcal{O})}^2 \bigg]
    \end{align*}
    Using the discrete Grönwall inequality, there exists a constant \(C\) such that
    \begin{equation}\label{eq-proof-lem:continuous noise apriori estimate eq 3.1}
        \max_{1 \leq n \leq \ell + 1}\mathbb{E} \big[\| \Pi_{\mathcal{D}}u_{\mathcal{D}}^{(n)}\|_{L^2(\mathcal{O})}^2 \big]
        +
        \mathbb{E} \bigg[\sum_{n=0}^{\ell} k_{\mathcal{D}} \|\nabla_{\mathcal{D}}u_{\mathcal{D}}^{(n+1)}\|_{L^p(\mathcal{O})}^p\bigg]
        \leq
        C \mathbb{E} \big[1 + \| \Pi_{\mathcal{D}}u_{\mathcal{D}}^{(0)}\|_{L^2(\mathcal{O})}^2 \big].
    \end{equation}

    Next,
    summing \eqref{eq-proof-lem:GS abstract noise apriori 1} from \(n = 0\) to \(\ell \),
    taking the maximum over \(0 \leq \ell \leq N-1\),
    taking the expectation,
    and by the coercivity property of \(\bm{a}\),
    we obtain the following
    \begin{align*}
                & \mathbb{E}\Big[\max_{1\leq n \leq N} \|\Pi_{\mathcal{D}} u_{\mathcal{D}}^{(n)}\|_{L^2(\mathcal{O})}^2\Big]
        + \mathbb{E} \bigg[\sum_{n = 0}^{N-1}\| \delta_{\mathcal{D}}^{(n+1)}u_{\mathcal{D}} \|_{L^2(\mathcal{O})}^2 \bigg]                                                                                                \\
        \leq \, &
        \mathbb{E} \big[\| \Pi_{\mathcal{D}}u_{\mathcal{D}}^{(0)}\|_{L^2(\mathcal{O})}^2 \big]
        +
        2\mathbb{E} \bigg[\max_{1 \leq \ell \leq N-1}\sum_{n = 0}^{\ell}\big\langle  B(\Pi_{\mathcal{D}} u_{\mathcal{D}}^{(n)}) \Delta^{(n+1)}W, \Pi_{\mathcal{D}} u_{\mathcal{D}}^{(n+1)}\big\rangle \bigg]              \\
        \leq \, &
        \mathbb{E} \big[\| \Pi_{\mathcal{D}}u_{\mathcal{D}}^{(0)}\|_{L^2(\mathcal{O})}^2 \big]
        +
        \mathbb{E} \bigg[ \max_{1 \leq \ell \leq N-1} \sum_{n = 0}^{\ell}\langle B(\Pi_{\mathcal{D}} u_{\mathcal{D}}^{(n)}) \Delta^{(n+1)}W, \Pi_{\mathcal{D}} \delta_{\mathcal{D}}^{(n+1)} u_{\mathcal{D}} \rangle\bigg] \\
                & +
        \mathbb{E} \bigg[ \max_{1 \leq \ell \leq N-1} \sum_{n = 0}^{\ell}\langle B(\Pi_{\mathcal{D}} u_{\mathcal{D}}^{(n)}) \Delta^{(n+1)}W, \Pi_{\mathcal{D}} u_{\mathcal{D}}^{(n)}\rangle\bigg]
    \end{align*}
    Consider the second term on the right-hand side.
    By the Cauchy-Schwarz inequality, the Young inequality, and \eqref{eq-proof-lem:continuous noise analysis 1 1},
    we can obtain
    \begin{align*}
                & \mathbb{E} \bigg[ \max_{1 \leq \ell \leq N-1} \sum_{n = 0}^{\ell}\langle B(\Pi_{\mathcal{D}} u_{\mathcal{D}}^{(n)}) \Delta^{(n+1)}W, \Pi_{\mathcal{D}} \delta_{\mathcal{D}}^{(n+1)} u_{\mathcal{D}} \rangle\bigg] \\
        \leq \, &
        \mathbb{E} \bigg[ \sum_{n = 0}^{N-1}\| B(\Pi_{\mathcal{D}} u_{\mathcal{D}}^{(n)}) \Delta^{(n+1)}W\|_{L^2(\mathcal{O})} \|\delta_{\mathcal{D}}^{(n+1)} u_{\mathcal{D}}\|_{L^2(\mathcal{O})}\bigg]                            \\
        \leq \, &
        \frac{1}{2}\mathbb{E} \bigg[ \sum_{n = 0}^{N-1}\|B(\Pi_{\mathcal{D}} u_{\mathcal{D}}^{(n)}) \Delta^{(n+1)}W\|_{L^2(\mathcal{O})}^2\bigg]
        +
        \frac{1}{2}\mathbb{E} \Bigg[ \sum_{n = 0}^{N-1}\|\delta_{\mathcal{D}}^{(n+1)} u_{\mathcal{D}}\|_{L^2(\mathcal{O})}^2\Bigg]                                                                                                  \\
        \leq \, &
        \frac{1}{2} \mathbb{E} \bigg[ \sum_{n = 0}^{N-1} k_{\mathcal{D}}  \| B(\Pi_{\mathcal{D}} u_{\mathcal{D}}^{(n)})\|_{L^2(\mathcal{O})}^2 \bigg]
        +
        \frac{1}{2} \mathbb{E} \bigg[ \sum_{n = 0}^{N-1}\|\delta_{\mathcal{D}}^{(n+1)} u_{\mathcal{D}}\|_{L^2(\mathcal{O})}^2\bigg].
    \end{align*}
    Now consider the third term on the right-hand side.
    We treat the sum as a stochastic integral of a piecewise constant integrand and use the Burkholder-Davis-Gundy inequality (see,  e.g.  \cite[Theorem 2.4]{brzezniak1997stochastic}), then there exists a constant \(C_B\) such that
    \begin{align*}
                & \mathbb{E} \bigg[ \max_{1 \leq \ell \leq N-1} \sum_{n = 0}^{\ell}\langle B(\Pi_{\mathcal{D}} u_{\mathcal{D}}^{(n)}) \Delta^{(n+1)}W, \Pi_{\mathcal{D}} u_{\mathcal{D}}^{(n)}\rangle\bigg]                                                                     \\
        \leq \, & C_B \mathbb{E}\bigg[ \bigg( \sum_{n = 1}^{N-1} k_{\mathcal{D}} \|B(\Pi_{\mathcal{D}}u_{\mathcal{D}}^{(n)})\|_{L_2^0}^2 \|\Pi_{\mathcal{D}} u_{\mathcal{D}}^{(n)}\|_{L^2(\mathcal{O})}^2 \bigg)^{1/2} \bigg]                                                   \\
        \leq \, & C_B \mathbb{E}\bigg[\max_{0 \leq n \leq N - 1} \|\Pi_{\mathcal{D}} u_{\mathcal{D}}^{(n)}\|_{L^2(\mathcal{O})} \bigg(\sum_{n = 1}^{N-1} k_{\mathcal{D}} \|B(\Pi_{\mathcal{D}} u_{\mathcal{D}}^{(n)}) \|_{L_2^0}^2 \bigg)^{1/2}\bigg]                           \\
        \leq \, & \frac{1}{4} \mathbb{E}\Big[\max_{0 \leq n \leq N - 1} \|\Pi_{\mathcal{D}} u_{\mathcal{D}}^{(n)}\|_{L^2(\mathcal{O})}^2 \Big] + C_B^2 \mathbb{E}\bigg[ \sum_{n = 1}^{N-1} k_{\mathcal{D}}  \|B (\Pi_{\mathcal{D}} u_{\mathcal{D}}^{(n)}) \|_{L_2^0}^2  \bigg]. 
    \end{align*}

    Therefore,
    we have
    \begin{align*}
        \mathbb{E}\Big[\max_{1\leq n \leq N} \|\Pi_{\mathcal{D}} u_{\mathcal{D}}^{(n)}\|_{L^2(\mathcal{O})}^2\Big]
        \leq \, &
        \mathbb{E} \Big[\| \Pi_{\mathcal{D}}u_{\mathcal{D}}^{(0)}\|_{L^2(\mathcal{O})}^2 \Big]
        +
        \mathbb{E}\bigg[\sum_{n = 0}^{N-1} k_{\mathcal{D}}\|B( \Pi_{\mathcal{D}} u_{\mathcal{D}}^{(n)})\|_{L_2^0}^2 \bigg]                        \\
                & + \frac{1}{4} \mathbb{E}\Big[\max_{0 \leq n \leq N - 1} \|\Pi_{\mathcal{D}} u_{\mathcal{D}}^{(n)}\|_{L^2(\mathcal{O})}^2 \Big].
    \end{align*}
    Bounding the second term and the last term on the right-hand side by \eqref{eq-proof-lem:continuous noise apriori estimate eq 3.1}, then by the discrete Grönwall inequality,
    there exists a exists a constant \(C\) such that
    \begin{equation*}
        \mathbb{E}\Big[\max_{1\leq n \leq N} \|\Pi_{\mathcal{D}} u_{\mathcal{D}}^{(n)}\|_{L^2(\mathcal{O})}^2\Big]
        \leq
        C \big(1 + \mathbb{E} \big[\| \Pi_{\mathcal{D}}u_{\mathcal{D}}^{(0)}\|_{L^2(\mathcal{O})}^2 \big]\big).
        \qedhere
    \end{equation*}
\end{proof}

\begin{proof}[Proof of \eqref{eq-lem:continuous noise apriori estimate time translate}]
    Let \(\phi_{\mathcal{D}} = k_{\mathcal{D}}  (u_{\mathcal{D}}^{(n+1)} - u_{\mathcal{D}}^{(n)}) \) in
    \eqref{model:GS Lip continuous noise}
    and sum over \(n\) from \(1\) to \(N - 1\),
    \begin{align*}
        \sum_{n = 1}^{N-1} k_{\mathcal{D}} \| \Pi_{\mathcal{D}} u_{\mathcal{D}}^{(n+ 1)} - \Pi_{\mathcal{D}} u_{\mathcal{D}}^{(n)}\|_{L^2(\mathcal{O})}^2
        =  & \,
        - \sum_{n=1}^{N - 1}  k_{\mathcal{D}}^2
        \langle \bm{a}(x, \nabla_{\mathcal{D}} u_{\mathcal{D}}^{(n+1)}), \nabla_{\mathcal{D}} u_{\mathcal{D}}^{(n+ 1)} - \nabla_{\mathcal{D}} u_{\mathcal{D}}^{(n)} \rangle                                 \\
           & \, +
        \sum_{n=1}^{N - 1}  k_{\mathcal{D}} \langle B(\Pi_{\mathcal{D}} u_{\mathcal{D}}^{(n)}) \Delta^{(n+1)}W, \Pi_{\mathcal{D}} u_{\mathcal{D}}^{(n+1)} - \Pi_{\mathcal{D}} u_{\mathcal{D}}^{(n)} \rangle \\
        =: & \, T_1 + T_2.
    \end{align*}
    We now estimate the expectation of \(T_1\).
    By the Hölder inequality, the boundedness property of \(\bm{a}\), and the triangle inequality,
    we have
    \begin{align*}
        \mathbb{E}[T_1]
        \leq & \, \mathbb{E} \bigg[
        \sum_{n=1}^{N - 1}k_{\mathcal{D}}^2 \langle |\overline{\bm{a}}| + \lambda |\nabla_{\mathcal{D}} u_{\mathcal{D}}^{(n+1)}|^{p-1}, |\nabla_{\mathcal{D}} u_{\mathcal{D}}^{(n+1)} - \nabla_{\mathcal{D}} u_{\mathcal{D}}^{(n)}|   \rangle
        \bigg]                                                                                                \\
        \leq & \, \mathbb{E} \bigg[
        \sum_{n=1}^{N - 1}k_{\mathcal{D}}^2 \|\overline{\bm{a}}\|_{L^{p'}(\mathcal{O})} \|\nabla_{\mathcal{D}} u_{\mathcal{D}}^{(n+1)} - \nabla_{\mathcal{D}} u_{\mathcal{D}}^{(n)}\|_{L^p(\mathcal{O})}
        \Bigg]                                                                                                \\
             & +
        \mathbb{E} \bigg[
        \sum_{n=1}^{N - 1}k_{\mathcal{D}}^2
        \lambda \| \nabla_{\mathcal{D}} u_{\mathcal{D}}^{(n+1)}\|_{L^p(\mathcal{O})}^{p-1}
        \|\nabla_{\mathcal{D}} u_{\mathcal{D}}^{(n+1)} - \nabla_{\mathcal{D}} u_{\mathcal{D}}^{(n)}\|_{L^p(\mathcal{O})}
        \bigg]                                                                                                \\
        \leq & \,
        2k_{\mathcal{D}} \|\overline{\bm{a}}\|_{L^{p'}(\mathcal{O})}  \|\nabla_{\mathcal{D}} u_{\mathcal{D}}\|_{L^1(\Omega_T;L^p(\mathcal{O}))}
        +
        k_{\mathcal{D}} \lambda \|\nabla_{\mathcal{D}} u_{\mathcal{D}}\|_{L^p(\Omega_T;L^p(\mathcal{O}))}^{p} \\
             & +
        k_{\mathcal{D}} \lambda \mathbb{E} \bigg[
        \sum_{n=1}^{N - 1}k_{\mathcal{D}}
        \| \nabla_{\mathcal{D}} u_{\mathcal{D}}^{(n+1)}\|_{L^p(\mathcal{O})}^{p}
        \bigg]^{(p-1)/p}
        \mathbb{E} \bigg[
        \sum_{n=1}^{N - 1}k_{\mathcal{D}}
        \| \nabla_{\mathcal{D}} u_{\mathcal{D}}^{(n)}\|_{L^p(\mathcal{O})}^{p}
        \bigg]^{1/p}                                                                                          \\
        \leq & \,
        2k_{\mathcal{D}} \|\overline{a}\|_{L^{p'}(\mathcal{O})}  \|\nabla_{\mathcal{D}} u_{\mathcal{D}}\|_{L^1(\Omega_T;L^p(\mathcal{O}))}
        +
        2k_{\mathcal{D}} \lambda \|\nabla_{\mathcal{D}} u_{\mathcal{D}}\|_{L^p(\Omega_T;L^p(\mathcal{O}))}^{p}.
    \end{align*}
    Now we consider the expectation of \(T_2\). Using the Cauchy-Schwarz inequality, the Young inequality, and the Itô isometry to write
    \begin{align*}
        \mathbb{E}[T_2]
        =    & \,
        k_{\mathcal{D}} \mathbb{E}\bigg[
            \sum_{n=1}^{N - 1}  \langle B(\Pi_{\mathcal{D}} u_{\mathcal{D}}^{(n)}) \Delta^{(n+1)}W, \Pi_{\mathcal{D}} u_{\mathcal{D}}^{(n+1)} - \Pi_{\mathcal{D}} u_{\mathcal{D}}^{(n)} \rangle
        \bigg]    \\
        \leq & \,
        \frac{1}{2}k_{\mathcal{D}} \mathbb{E}\bigg[
        \sum_{n=1}^{N - 1} \| B(\Pi_{\mathcal{D}} u_{\mathcal{D}}^{(n)}) \Delta^{(n+1)}W\|_{L^2(\mathcal{O})}^2
        \bigg]
        +
        \frac{1}{2}k_{\mathcal{D}} \mathbb{E}\bigg[
        \sum_{n=1}^{N - 1} \|\Pi_{\mathcal{D}} u_{\mathcal{D}}^{(n+1)} - \Pi_{\mathcal{D}} u_{\mathcal{D}}^{(n)}\|_{L^2(\mathcal{O})}^2
        \bigg]    \\
        =    & \,
        \frac{1}{2}k_{\mathcal{D}} \mathbb{E}\bigg[
        \sum_{n=1}^{N - 1} k_{\mathcal{D}} \| B(\Pi_{\mathcal{D}} u_{\mathcal{D}}^{(n)}) \|_{L_2^0}^2
        \bigg]
        +
        \frac{1}{2}k_{\mathcal{D}} \mathbb{E}\bigg[
        \sum_{n=1}^{N - 1} \|\Pi_{\mathcal{D}} u_{\mathcal{D}}^{(n+1)} - \Pi_{\mathcal{D}} u_{\mathcal{D}}^{(n)}\|_{L^2(\mathcal{O})}^2
        \bigg].
    \end{align*}
    Collecting estimates for \(\mathbb{E}[T_1]\) and \(\mathbb{E}[T_2]\), and absorbing the second term on the upper bounded of \(\mathbb{E}[T_2]\) to the left-hand side, yield the \eqref{eq-lem:continuous noise apriori estimate time translate}.
\end{proof}

\subsection{Proof of convergence results}\label{section: stoc LL conti noise cvproof compactness}
\begin{lemma}\label{lem:continuous noise cv1}
    Let \(((\mathcal{D}_T)_m)_{m\in \mathbb{N}}\) be a consistent and limit-conforming sequence of space-time gradient discretisation.
    Let \(u_{\mathcal{D}_m}: \Omega \rightarrow X_{\mathcal{D}}^{N+1}\) be a solution to the gradient scheme \ref{def:GS SPDE with continuous ito noise} with \(\mathcal{D} \equiv \mathcal{D}_m\).
    Up to a subsequence, as \(m \rightarrow \infty\), there exists \(dt \otimes \mathbb{P}\)-progressively measurable \(\overline{u} \in L^2(\Omega_T; L^2(\mathcal{O})) \cap L^{p}(\Omega_T; W_0^{1,p}(\mathcal{O}))\) with \(p \in (1, \infty)\) such that,
    \begin{subequations}\label{eq-lem:continuous noise cv1}
        \begin{align}
             & \Pi_{\mathcal{D}_m} u_{\mathcal{D}_m}   \rightarrow \overline{u}  \text{ weakly in \(L^{\max(2,p)}(\Omega_T; L^{\max(2,p)}(\mathcal{O}))\)};\label{eq-lem:continuous noise cv1 1} \\
             & \nabla_{\mathcal{D}_m} u_{\mathcal{D}_m}   \rightarrow \nabla \overline{u} \text{ weakly  in \(L^p( \Omega_T; L^p(\mathcal{O})^d)\)}.\label{eq-lem:continuous noise cv1 2}
        \end{align}
    \end{subequations}
    There exists a \(dt \otimes \mathbb{P}\)-progressively measurable \(\bm{A} \in L^{p'}( \Omega_T; L^{p'}(\mathcal{O}))^d\) where \(p' = p/(p-1)\) such that
    \begin{equation}\label{eq-lem:continuous noise cv1 3}
        \bm{a}(x,  \nabla_{\mathcal{D}_m} u_{\mathcal{D}_m}) \rightarrow \bm{A} \text{ weakly  in \(L^{p'}( \Omega_T; L^{p'}(\mathcal{O}))^d\)}.
    \end{equation}
    There exists \(dt \otimes \mathbb{P}\)-progressively measurable \(\bm{B} \in L^2(\Omega_T; L_2^0)\) such that
    \begin{equation}\label{eq-lem:continuous noise cv1 B}
        B(\Pi_{\mathcal{D}_m} \underline{u}_{\mathcal{D}_m}) \rightarrow \bm{B} \text{ weakly in \(L^2(\Omega_T; L_2^0)\).}
    \end{equation}
\end{lemma}

\begin{proof}
    The convergences \eqref{eq-lem:continuous noise cv1 1} and \eqref{eq-lem:continuous noise cv1 2} follow from Lemma \ref{lem:continuous noise apriori estimate} and Lemma \ref{lem: stoc regularity of the limit}.
    
    The boundedness property of \(\bm{a}\) and estimate \eqref{eq-lem:continuous noise apriori estimate}
    imply that \(\bm{a}(x,\nabla_{\mathcal{D}_m} u_{\mathcal{D}_m})\) is bounded in \(L^{p'}( \Omega_T; L^{p'}(\mathcal{O}))^d\) where \(p' = p/(p-1)\).
    Therefore, up to a subsequence, \(\bm{a}(x, \nabla_{\mathcal{D}_m} u_{\mathcal{D}_m})\) converges to some \(\bm{A}\) weakly in \(L^{p'}( \Omega_T; L^{p'}(\mathcal{O}))^d\) as \(m \rightarrow \infty\).
    The progressive measurability of \(\bm{A}\) follows from the same argument in Lemma \ref{lem: stoc regularity of the limit}.

    The growth property of \(B\) implies that \(B(\Pi_{\mathcal{D}_m} u_{\mathcal{D}_m})\) is bounded in \(L^2(\Omega_T; L_2^0)\). Therefore,
    \(B(\Pi_{\mathcal{D}_m} u_{\mathcal{D}_m})\) converges to some \(\bm{B}\) weakly in \(L^2(\Omega_T; L_2^0)\) as \( m \rightarrow \infty\).
\end{proof}

\begin{lemma}\label{lem:continuous noise cv2}
    Suppose that \ref{def:assumption opertator a}  holds.
    Let \(((\mathcal{D}_T)_m)_{m \in \mathbb{N}}\) be a consistent and limit-conforming sequence of gradient discretisation.
    Let \(\overline{u}\), \(\bm{A}\), and \(\bm{B}\) be the limits that are associated to solution \(u_{\mathcal{D}_m}\) of the gradient scheme \eqref{model:GS Lip continuous noise} in Lemma \ref{lem:continuous noise cv1}.
    Then,
    \(\overline{u}\) is a continuous \(L^2(\mathcal{O})\)-valued process and
    for any \(t \in [0,T]\),
    \begin{equation}\label{eq-lem:continuous noise cv2 2}
        \mathbb{E} \big[ \| \overline{u}(t)\|_{L^2(\mathcal{O})}^2 \big]
        +
        \mathbb{E} \Big[\int_0^t
            2 \langle \bm{A}, \nabla \overline{u}  \rangle  \, ds \Big]
        = \mathbb{E}\big[\| \overline{u}_0\|_{L^2(\mathcal{O})}^2 \big]
        +
        \mathbb{E} \Big[\int_0^t
        \| \bm{B}\|_{L_2^0}^2
        \, ds \Big].
    \end{equation}
    In addition, for any \(t \in [0,T]\), along a subsequence as \(m \rightarrow \infty\),
    \begin{equation}\label{eq-lem:continuous noise cv2 2 2}
        \Pi_{\mathcal{D}_m} u_{\mathcal{D}_m} (t)
        \rightarrow
        \overline{u}(t)
        \text{ weakly in \(L^2(\Omega; L^2(\mathcal{O}))\)},
    \end{equation}
    and satisfies
    \begin{equation}\label{eq-lem:continuous noise uniform bound}
        \mathbb{E} \Big[
        \sup_{t \in [0,T]} \| \overline{u}(t)\|_{L^2(\mathcal{O})}^2
        \Big]
        < \infty.
    \end{equation}
\end{lemma}
\begin{proof}
    Let \(V:= L^2(\mathcal{O})\cap W_0^{1,p}(\mathcal{O})\).
    Take any \(t \in (0,T]\), then there exists \(\ell = 1, \ldots, N_m-1\) such that \(t \in (t^{(\ell)}, t^{(\ell+1)}]\).
    Take any \(\psi \in L^{\max(2,p)}(\Omega_T)\) and any \(\phi \in V\).
    Set \(\Phi := \psi \phi\) and consider \(X_{\mathcal{D}_m}\)-valued piecewise constant process \(\Phi_{\mathcal{D}_m} := \psi_m \phi_{\mathcal{D}_m}\) in the sense of Lemma \ref{lem:GD-stoc interpolation strong cv}, then
    \begin{subequations}\label{eq-proof-lem:GD interpolation strong cv}
        \begin{align}
             & \Pi_{\mathcal{D}_m} \Phi_{\mathcal{D}_m} \rightarrow \Phi  \text{ strongly in {\(L^{\max(2,p)}(\Omega_T; L^{\max(2,p)}(\mathcal{O}))\)}}, \label{eq-proof-lem:GD interpolation strong cv 1} \\
             & \nabla_{\mathcal{D}_m} \Phi_{\mathcal{D}_m} \rightarrow \nabla \Phi \text{ strongly in \(L^p(\Omega_T; L^p(\mathcal{O})^d)\)}. \label{eq-proof-lem:GD interpolation strong cv 2}
        \end{align}
    \end{subequations}

    Take summation of \eqref{model:GS Lip continuous noise} (with \(\mathcal{D} = \mathcal{D}_m\)) from \(n = 0\) to \(\ell\) with \(\phi_{\mathcal{D}} = \Phi_{\mathcal{D}_m}\),
    \begin{align*}
             & \langle \Pi_{\mathcal{D}_m} u_{\mathcal{D}_m}^{(\ell)} , \Pi_{\mathcal{D}_m}\Phi_{\mathcal{D}_m} \rangle
        + \sum_{n = 0}^{\ell} k_{\mathcal{D}_m} \langle \bm{a} (x, \nabla_{\mathcal{D}_m} {u}_m^{(n+1)}), \nabla_{\mathcal{D}_m} \Phi_{\mathcal{D}_m} \rangle \\
        = \, &
        \langle \Pi_{\mathcal{D}_m}   u_{\mathcal{D}_m}^{(0)}, \Pi_{\mathcal{D}_m}\Phi_{\mathcal{D}_m} \rangle
        +
        \sum_{n = 0}^{\ell}  \langle B(\Pi_{\mathcal{D}_m} u_{\mathcal{D}_m}^{(n)}) \Delta^{(n+1)}W, \Pi_{\mathcal{D}_m} \Phi_{\mathcal{D}_m} \rangle
    \end{align*}
    Applying the relations: for any piecewise constant in time function \(f\)
    \begin{subequations}\label{eq:relation sum to int}
        \begin{align}
             & \sum_{n=0}^{\ell} k_{\mathcal{D}_m}f =  \int_0^t f \, ds + \int_t^{\lceil t / k_{\mathcal{D}_m} \rceil k_{\mathcal{D}_m}} f \, ds         \\
             & \sum_{n=0}^{\ell} f \Delta^{(\ell+1)} W =  \int_0^t f \, dW_s + \int_t^{\lceil t / k_{\mathcal{D}_m} \rceil k_{\mathcal{D}_m}} f \, dW_s,
        \end{align}
    \end{subequations}
    to the summation terms  yields
    \begin{align}
             & \langle \Pi_{\mathcal{D}_m} u_{\mathcal{D}_m}(t) , \Pi_{\mathcal{D}_m}\Phi_{\mathcal{D}_m} \rangle
        + \int_0^t \langle \bm{a} (x, \nabla_{\mathcal{D}_m} {u}_m), \nabla_{\mathcal{D}_m} \Phi_{\mathcal{D}_m} \rangle \, ds \nonumber \\
        = \  & \langle \Pi_{\mathcal{D}_m} u_{\mathcal{D}_m}^{(0)}, \Pi_{\mathcal{D}_m} \Phi_{\mathcal{D}_m} \rangle
        +  \int_0^t  \langle B(\Pi_{\mathcal{D}_m} \underline{u}_{\mathcal{D}_m})dW_s, \Pi_{\mathcal{D}_m} \Phi_{\mathcal{D}_m} \rangle + R_m(t), \label{eq-proof-lem:continuous noise cv2 1.0}
    \end{align}
    where
    \begin{align*}
        R_m(t) =   \int_t^{\lceil t / k_{\mathcal{D}_m} \rceil k_{\mathcal{D}_m}} - \langle \bm{a} (x,  \nabla_{\mathcal{D}_m} {u}_{\mathcal{D}_m}), \nabla_{\mathcal{D}_m} \Phi_{\mathcal{D}_m} \rangle
        + \int_t^{\lceil t / k_{\mathcal{D}_m} \rceil k_{\mathcal{D}_m}} \langle B(\Pi_{\mathcal{D}_m} \underline{u}_{\mathcal{D}_m}) dW_s, \Pi_{\mathcal{D}_m} \Phi_{\mathcal{D}_m} \rangle.
    \end{align*}

    Integrating \eqref{eq-proof-lem:continuous noise cv2 1.0} from \(0\) to \(T\) with respect to \(dt\) and taking the expectation leads to
    \begin{align}
             & \mathbb{E}\int_0^T \langle \Pi_{\mathcal{D}_m} u_{\mathcal{D}_m}(t) , \Pi_{\mathcal{D}_m}\Phi_{\mathcal{D}_m} \rangle \, dt
        + \mathbb{E}\int_0^T \int_0^t \langle \bm{a} (x, \nabla_{\mathcal{D}_m} {u}_{\mathcal{D}_m}), \nabla_{\mathcal{D}_m} \Phi_{\mathcal{D}_m} \rangle \, ds \, dt \nonumber \\
        = \  & \mathbb{E}\int_0^T  \langle \Pi_{\mathcal{D}_m} u_{\mathcal{D}_m}^{(0)}, \Pi_{\mathcal{D}_m} \Phi_{\mathcal{D}_m} \rangle \, dt
        + \mathbb{E}\int_0^T \int_0^t  \langle B(\Pi_{\mathcal{D}_m} \underline{u}_{\mathcal{D}_m}) dW_s, \Pi_{\mathcal{D}_m} \Phi_{\mathcal{D}_m} \rangle \, dt + \overline{R}_m, \label{eq-proof-lem:continuous noise cv2 1}
    \end{align}
    where
    \(\overline{R}_m := \mathbb{E}\int_0^T R_m(t) \, dt.\)
    Note that for any \(t \in [0,T]\), the integrands of \(R_m(t)\) are constant on the interval \((t, \lceil t / k_{\mathcal{D}_m} \rceil k_{\mathcal{D}_m}]\).
    By the Hölder inequality and the boundedness of \(\bm{a}\), we have
    \begin{align*}
                & \mathbb{E} \int_0^T \int_{t}^{\lceil t / k_{\mathcal{D}_m} \rceil k_{\mathcal{D}_m}}  \langle \bm{a} (x,  \nabla_{\mathcal{D}_m} u_{\mathcal{D}_m}),  \nabla_{\mathcal{D}_m} \Phi_{\mathcal{D}_m} \rangle \, ds \, dt \\
        \leq \, & k_{\mathcal{D}_m}   \bm{a} (x,  \nabla_{\mathcal{D}_m} u_{\mathcal{D}_m}) \|_{L^{p'}(\Omega_T; L^{p'}(\mathcal{O}))} \| \nabla_{\mathcal{D}_m}\Phi_{\mathcal{D}_m} \|_{L^p(\Omega_T;L^p(\mathcal{O}))};
    \end{align*}
    by the Cauchy-Schwarz inequality and the Itô isometry,  
    the following can be obtained
    \begin{align*}
                & \mathbb{E} \int_0^T \int_t^{\lceil t / k_{\mathcal{D}_m} \rceil k_{\mathcal{D}_m}}  \langle B(\Pi_{\mathcal{D}_m} \underline{u}_{\mathcal{D}_m}) dW_s, \Pi_{\mathcal{D}_m} \Phi_{\mathcal{D}_m} \rangle \, dt \\
        \leq \, &
        k_{\mathcal{D}_m}  \| B(\Pi_{\mathcal{D}_m} \underline{u}_{\mathcal{D}_m})\|_{L^2(\Omega_T; L^2(\mathcal{O}))} \|\Pi_{\mathcal{D}_m} \Phi_{\mathcal{D}_m}\|_{L^2(\Omega_T; L^2(\mathcal{O}))}.
    \end{align*} 
    Therefore, if \(m \rightarrow \infty\), then \(k_{\mathcal{D}_m} \rightarrow 0\) and hence
    \begin{equation}\label{eq-proof-lem:continuous noise cv2 cv tail}
        \overline{R}_m \rightarrow 0 \text{ as \(m \rightarrow \infty\)}.
    \end{equation}
    Letting \(m \rightarrow \infty\) in \eqref{eq-proof-lem:continuous noise cv2 1},
    on the left-hand side,
    recalling the weak convergence \eqref{eq-lem:continuous noise cv1 1} for the first term,
    recalling the weak convergence \eqref{eq-lem:continuous noise cv1 3} for the second term;
    on the right-hand side,
    recalling the approximation property of \(\mathcal{I}_{\mathcal{D}_m}\) (from consistency of GD) for the first term,
    recalling weak convergence \eqref{eq-lem:continuous noise cv1 B} for the second term,
    then passing these weak limits together with the strong convergences \eqref{eq-proof-lem:GD interpolation strong cv},
    we obtain
    \begin{equation*}
        \mathbb{E} \int_0^T \langle \overline{u}(t), \Phi \rangle \, dt
        + \mathbb{E} \int_0^T  \int_0^t   \langle \bm{A}, \nabla \Phi  \rangle \, ds \, dt
        =  \mathbb{E} \int_0^T  \langle  \overline{u}_0, \Phi \rangle \, dt
        + \mathbb{E} \int_0^T \int_0^t  \langle \bm{B} \, dW_s,  \Phi  \rangle \, dt,
    \end{equation*}
    In particular, for any \(\psi \in {L^{\max(2,p)}(\Omega_T)}\),
    \begin{equation*}
        \mathbb{E} \int_0^T \langle \overline{u}(t), \phi \rangle \, \psi \, dt
        + \mathbb{E} \int_0^T  \int_0^t   \langle \bm{A}, \nabla \phi  \rangle \, ds \, \psi \, dt
        =  \mathbb{E} \int_0^T  \langle  \overline{u}_0, \phi \rangle \, \psi \, dt
        + \mathbb{E} \int_0^T  \int_0^t  \langle \bm{B} \, dW_s,  \phi  \rangle \, \psi \, dt.
    \end{equation*}
    Therefore, for \(dt \otimes \mathbb{P}\)-almost all \((t, \omega) \in \Omega_T\),
    \begin{equation}\label{eq-proof-lem:continuous noise cv2 weak formulation}
        \langle\overline{u}(t), \phi \rangle
        + \int_0^t   \langle \bm{A}, \nabla \phi  \rangle \, ds
        =   \langle  \overline{u}_0, \phi \rangle
        +   \int_0^t  \langle \bm{B} \, dW_s,  \phi  \rangle \,
        \quad
        {\forall \phi \in V.}
    \end{equation} 
    Recalling the Sobolev embeddings
    {\(V \hookrightarrow  L^2(\mathcal{O}) \hookrightarrow V^*  \)}
    and
    invoking \cite[Theorem 4.2.5]{liu2015stochastic},
    then there is a \(L^2(\mathcal{O})\)-valued continuous version of \(\overline{u}\) such that,
    for any \(t \in [0,T]\), \eqref{eq-proof-lem:continuous noise cv2 weak formulation} holds \(\mathbb{P}\)-almost surely,
    and
    \begin{align*}
        \| \overline{u}(t)\|_{L^2(\mathcal{O})}^2
        +
        \int_0^t  2\langle \bm{A}, \nabla \hat{u} \rangle \, ds
        = \| \overline{u}_0\|_{L^2(\mathcal{O})}^2
        + \int_0^t  \|\bm{B}\|_{L_0^2}^2 \, ds
        + 2\int_0^t  \langle \overline{u},  \bm{B} dW(s)\rangle
    \end{align*}
    where \(\hat{u}\) is any  {\(V\)}-valued progressively measurable \(dt \otimes \mathbb{P}\)-version of \(\overline{u}\).
    Then taking the expectation leads to \eqref{eq-lem:continuous noise cv2 2}.

    Next, we prove the convergence \eqref{eq-lem:continuous noise cv2 2 2}.
    Take any \(t \in [0,T]\), then there exists \(\ell = 1, \ldots, N_m-1\) such that \(t \in (t^{(\ell)}, t^{(\ell+1)}]\).
    Since \eqref{eq-lem:continuous noise apriori estimate} holds for every \(m \in \mathbb{N}\),
    we have
    \begin{equation*}
        \mathbb{E}\big[  \|\Pi_{\mathcal{D}_m} u_{\mathcal{D}_m}(t)\|_{L^2(\mathcal{O})}^2\big]
        \leq
        \mathbb{E}\Big[\max_{1\leq n \leq N_m} \|\Pi_{\mathcal{D}_m} u_{\mathcal{D}_m}^{(n)}\|_{L^2(\mathcal{O})}^2\Big]
        \leq
        C(T,  u_{\mathcal{D}}^{(0)}).
    \end{equation*}
    Then there exists \(U_t \in L^2(\Omega; L^2(\mathcal{O}))\) such that along a subsequence,
    as \(m \rightarrow \infty\),
    \begin{equation}
        \Pi_{\mathcal{D}_m} u_{\mathcal{D}_m} (t)
        \rightarrow
        U_t
        \text{ weakly in \(L^2(\Omega; L^2(\mathcal{O}))\)}.\label{eq-proof-lem:continuous noise cv2 pointwise in time cv}
    \end{equation}

    In \eqref{eq-proof-lem:continuous noise cv2 1.0}, take \(\Phi_{\mathcal{D}_m} = \psi \phi_{\mathcal{D}_m}\) with \(\psi \in L^{\max(2,p)}(\Omega)\).
    Since for any \(t \in [0,T]\),
    \begin{align*}
        \mathbb{E}[R_m(t)]
        \leq \, &
        k_{\mathcal{D}_m}^{1/p} \| \bm{a}(x, \nabla_{\mathcal{D}_m} u_{\mathcal{D}_m})\|_{L^{p'}(\Omega_T; L^{p'}(\mathcal{O}))} \| \nabla_{\mathcal{D}_m} \Phi_{\mathcal{D}_m} \|_{L^p(\Omega;L^p(\mathcal{O}))} \\
                & +
        k_{\mathcal{D}_m}^{1/2} \| B(\Pi_{\mathcal{D}_m} \underline{u}_{\mathcal{D}_m})\|_{L^{2}(\Omega_T; L_2^0)} \| \Pi_{\mathcal{D}_m} \Phi_{\mathcal{D}_m} \|_{L^2(\Omega;L^2(\mathcal{O}))}
        \text{ \(\rightarrow 0\) as \(m \rightarrow \infty\)},
    \end{align*}
    taking the expectation and letting \(m \rightarrow \infty\),
    recalling convergences (\ref{eq-proof-lem:continuous noise cv2 pointwise in time cv}, \ref{eq-lem:continuous noise cv1 3}, \ref{eq-lem:continuous noise cv1 B}, \ref{eq-proof-lem:GD interpolation strong cv})
    and approximation properties of \(\mathcal{I}_{\mathcal{D}_m}\),
    we have
    \begin{equation*}
        \mathbb{E}\Big[ \psi \langle U_t, \phi \rangle \Big]
        + \mathbb{E}\Big[ \psi \int_0^t   \langle \bm{A}, \nabla \phi  \rangle \, d s \Big]
        = \mathbb{E}\Big[ \psi \langle  \overline{u}_0, \phi \rangle \Big]
        + \mathbb{E}\Big[ \psi \int_0^t  \langle \bm{B} \, dW_{s},  \phi  \rangle \Big].
    \end{equation*}
    Since this holds for any \(\psi \in L^{\max(2,p)}(\Omega)\), we deduce that for \(\mathbb{P}\)-almost all \(\omega \in \Omega\),
    \begin{equation*}
        \langle U_t, \phi \rangle
        + \int_0^t   \langle \bm{A}, \nabla \phi  \rangle \, d s
        =   \langle  \overline{u}_0, \phi \rangle
        +   \int_0^t  \langle \bm{B} \, dW_{s},  \phi  \rangle \,
        \quad
        \forall \phi \in V.
    \end{equation*}
    And hence \(U_t = \overline{u}(t)\) for every \(t\in [0,T]\).

    It remains to prove \eqref{eq-lem:continuous noise uniform bound}.
    Take \(t_1, \ldots, t_J \in [0,T]\).
    Since the mapping
    \[(L^2(\Omega; L^2(\mathcal{O})))^J \ni (\overline{u}(t_1), \ldots, \overline{u}(t_J)) \mapsto \mathbb{E}\big[\max_{1 \leq j \leq J} \|\overline{u}(t_j)\|_{L^2(\mathcal{O})}^2\big]^{1/2}\]
    forms a norm, then by its lower semicontinuous (see, e.g. \cite[Proposition 3.5]{BrezisHaim2010FASS})
    and \eqref{eq-lem:continuous noise apriori estimate} and \eqref{eq-lem:continuous noise cv2 2 2},
    \begin{equation*}
        \mathbb{E}\Big[\max_{1 \leq j \leq J} \|\overline{u}(t_j)\|_{L^2(\mathcal{O})}^2\Big]^{1/2}
        \leq \liminf_{m \rightarrow \infty} \mathbb{E}\Big[\max_{1 \leq j \leq J} \|\Pi_{\mathcal{D}_m}u_{\mathcal{D}_m}(t_j)\|_{L^2(\mathcal{O})}^2\Big]^{1/2}
        \leq C ,
    \end{equation*}
    with \(C\) independent of \(J\) and \(t_j\).
    Taking \(\{t_1, \ldots, t_J\}\) to be the first \(J\) elements of a countable dense subset \(D\) of \([0,T]\)
    and letting \(J \rightarrow \infty\), the monotone convergence theorem gives
    \[\mathbb{E}\Big[\sup_{s \in D} \|\overline{u}(s)\|_{L^2(\mathcal{O})}^2\Big] \leq C. \]
    Since \(\overline{u}\) has \(\mathbb{P}\)-almost surely continuous paths in \(L^2(\mathcal{O})\),
    the supremum over \(D\) coincides with the supremum over \([0,T]\),
    and therefore we obtain \eqref{eq-lem:continuous noise uniform bound}.
\end{proof}

\subsection*{Proof of Theorem \ref{thm:cv 1}}

Summing \eqref{eq-proof-lem:GS abstract noise apriori 1} over \(n\) from \(0\) to \(\ell \in \{0, 1, \ldots, N_m - 1\}\),
taking the expectation, and invoking \eqref{eq-proof-lem:continuous noise analysis 1 1},
we can obtain
\begin{align*}
            & \mathbb{E} \big[\| \Pi_{\mathcal{D}_m}u_{\mathcal{D}_m}^{(\ell+1)}\|_{L^2(\mathcal{O})}^2 \big]
    + 2\mathbb{E} \bigg[\sum_{n=0}^{\ell} k_{\mathcal{D}_m} \big\langle \bm{a} (x, \nabla_{\mathcal{D}_m} u_{\mathcal{D}_m}^{(n+1)}), \nabla_{\mathcal{D}_m}u_{\mathcal{D}_m}^{(n+1)} \big\rangle\bigg] \\
    \leq \, & \mathbb{E} \big[\| \Pi_{\mathcal{D}_m}u_{\mathcal{D}_m}^{(0)}\|_{L^2(\mathcal{O})}^2 \big]
    +
    \mathbb{E} \bigg[\sum_{n = 0}^{\ell} k_{\mathcal{D}_m}\| B(\Pi_{\mathcal{D}_m} u_{\mathcal{D}_m}^{(n)})\|_{L_2^0}^2\bigg].
\end{align*}
This, and the relations \eqref{eq:relation sum to int}, imply that
\begin{align*}
             & \mathbb{E} \Big[\| \Pi_{\mathcal{D}_m}u_{\mathcal{D}_m}(t)\|_{L^2(\mathcal{O})}^2
    + 2 \int_0^t \langle \bm{a} (x,  \nabla_{\mathcal{D}_m} u_{\mathcal{D}_m}), \nabla_{\mathcal{D}_m}u_{\mathcal{D}_m} \rangle \, ds \Big] \\ 
    \leq  \  & \mathbb{E}\big[  \| \Pi_{\mathcal{D}_m}u_{\mathcal{D}_m}^{(0)}\|_{L^2(\mathcal{O})}^2 \big]
    +
    \mathbb{E}\Big[\int_0^t  \|B (\Pi_{\mathcal{D}_m} \underline{u}_{\mathcal{D}_m}) \|_{L_2^0}^2 \, ds \Big]
    + R_m(t),
\end{align*}
where
\begin{equation*}
    R_{m}(t) := \mathbb{E} \Big[ \int_t^{\lceil t/ k_{\mathcal{D}_m} \rceil k_{\mathcal{D}_m}}
    \| B(\Pi_{\mathcal{D}_m} \underline{u}_{\mathcal{D}_m})\|_{L_2^0}^2
    - 2 \big\langle \bm{a} (x, \nabla_{\mathcal{D}_m} u_{\mathcal{D}_m}), \nabla_{\mathcal{D}_m} u_{\mathcal{D}_m} \big\rangle  \, ds \Big].
\end{equation*}
By Lemma \ref{lem:stoc auxiliary 1}, for all \(\kappa \geq 0\),
\begin{align}\label{eq-proof-thm:cv 1 identify limits 1}
             & e^{-\kappa t} \mathbb{E} \big[\| \Pi_{\mathcal{D}_m}u_{\mathcal{D}_m}(t)\|_{L^2(\mathcal{O})}^2 \big]
    + \kappa \mathbb{E} \Big[ \int_0^t e^{-\kappa s} \| \Pi_{\mathcal{D}_m}u_{\mathcal{D}_m}(s)\|_{L^2(\mathcal{O})}^2 ds \Big]                                                     \nonumber \\
             & + 2 \mathbb{E} \Big[\int_0^t e^{-\kappa s} \langle \bm{a} (x,  \nabla_{\mathcal{D}_m} u_{\mathcal{D}_m}), \nabla_{\mathcal{D}_m}u_{\mathcal{D}_m} \rangle \, ds \Big]          \\ 
    \leq  \  & \mathbb{E}\big[  \| \Pi_{\mathcal{D}_m}u_{\mathcal{D}_m}^{(0)}\|_{L^2(\mathcal{O})}^2 \big]
    +
    \mathbb{E}\Big[\int_0^t e^{-\kappa s} \|B (\Pi_{\mathcal{D}_m} \underline{u}_{\mathcal{D}_m}) \|_{L_2^0}^2 \, ds \Big]
    + \tilde{R}_m(t),\nonumber
\end{align}
where
\begin{equation*}
    \tilde{R}_m(t) := \kappa \mathbb{E}\Big[\int_0^t e^{-\kappa s} R_m(s) \, ds\Big]  \rightarrow 0 \text{ as \(m \rightarrow \infty\),}
\end{equation*}
which follows by a similar argument for \eqref{eq-proof-lem:continuous noise cv2 cv tail}.

Applying Lemma \ref{lem:stoc auxiliary 1} to \eqref{eq-lem:continuous noise cv2 2} yields, for all \(\kappa \geq 0\),
\begin{align*}
         & e^{-\kappa t}\mathbb{E} \big[ \| \overline{u}(t)\|_{L^2(\mathcal{O})}^2 \big]
    +
    \kappa \mathbb{E} \Big[ \int_0^t e^{-\kappa s} \|\overline{u}(s)\|_{L^2(\mathcal{O})}^2 ds\Big]
    +
    \mathbb{E} \Big[\int_0^t 2 e^{-\kappa s} \langle \bm{A}, \nabla \overline{u} \rangle  \, ds \Big] \\
    = \, & \mathbb{E}\big[\| \overline{u}_0\|_{L^2(\mathcal{O})}^2 \big]
    +
    \mathbb{E} \Big[\int_0^t e^{-\kappa s} \| \bm{B}\|_{L_2^0}^2 \, ds \Big].
\end{align*}
Subtracting this from \eqref{eq-proof-thm:cv 1 identify limits 1}, we can obtain
\begin{align}
            & e^{-\kappa t} \mathbb{E} \big[
    \| \Pi_{\mathcal{D}_m}u_{\mathcal{D}_m}(t) - \overline{u}(t)\|_{L^2(\mathcal{O})}^2
    \big]
    + \kappa \mathbb{E} \Big[ \int_0^t e^{-\kappa s} \| \Pi_{\mathcal{D}_m}u_{\mathcal{D}_m}(s) -  \overline{u}(s)\|_{L^2(\mathcal{O})}^2 ds\Big]
    \nonumber                                \\
    +       &
    2\mathbb{E} \Big[\int_0^t e^{-\kappa s}\langle \bm{a} (x,  \nabla_{\mathcal{D}_m} u_{\mathcal{D}_m}) - \bm{a} (x, \nabla \overline{u}) , \nabla_{\mathcal{D}_m}u_{\mathcal{D}_m} - \nabla \overline{u} \rangle \Big]
    \nonumber                                \\
    \leq \, &
    \mathbb{E}\Big[\int_0^t  e^{-\kappa s}\| B(\Pi_{\mathcal{D}_m} \underline{u}_{\mathcal{D}_m}) \|_{L_2^0}^2 -  e^{-\kappa s}\| \bm{B}\|_{L_2^0}^2 \, ds \Big]
    +
    \varepsilon_m(t)
    \label{eq-proof-thm:cv 1 identify limits 2}
\end{align}
where
\begin{align*}
    \varepsilon_m (t)
    := \, &
    \mathbb{E}\big[  \| \Pi_{\mathcal{D}_m}u_{\mathcal{D}_m}^{(0)}\|_{L^2(\mathcal{O})}^2 - \| \overline{u}_0\|_{L^2(\mathcal{O})}^2\big]
    +
    2e^{-\kappa t}\mathbb{E}\big[\langle \overline{u}(t) -  \Pi_{\mathcal{D}_m}u_{\mathcal{D}_m}(t), \overline{u}(t) \rangle \big]
    \\
          & +
    2\kappa \mathbb{E} \Big[ \int_0^t e^{-\kappa s} \langle  \overline{u} - \Pi_{\mathcal{D}_m}u_{\mathcal{D}_m}, \overline{u} \rangle\Big]
    + 2\mathbb{E} \Big[\int_0^t  e^{-\kappa s} \Big\langle \bm{A} - \bm{a} (x,  \nabla_{\mathcal{D}_m} u_{\mathcal{D}_m}) , \nabla \overline{u} \Big\rangle  \Big]
    \\
          &
    - 2\mathbb{E} \Big[\int_0^t e^{-\kappa s} \langle  \bm{a} (x, \nabla \overline{u}) , \nabla_{\mathcal{D}_m}u_{\mathcal{D}_m} - \nabla \overline{u} \rangle \Big]
    + \tilde{R}_m(t).
\end{align*} 
Consider the first term on the right-hand side.
Using the relation
\begin{equation*}
    |a|^2 - |b|^2 = |a-c|^2 - |b-c|^2 + 2 (a-b)c,
    \quad \forall a, b, c \in \mathbb{R}.
\end{equation*}
and the Lipschitz continuity of \(B\), 
we can have
\begin{align*}
            & \mathbb{E}\Big[\int_0^t  e^{-\kappa s}\| B(\Pi_{\mathcal{D}_m} \underline{u}_{\mathcal{D}_m}) \|_{L_2^0}^2 -  e^{-\kappa s}\| \bm{B}\|_{L_2^0}^2 \, ds \Big]   \\
    = \,    &
    \mathbb{E}\Big[\int_0^t e^{-\kappa s} \| B(\Pi_{\mathcal{D}_m} \underline{u}_{\mathcal{D}_m}) -  B(\overline{u})\|_{L_2^0}^2 \, ds \Big]
    - \mathbb{E}\Big[\int_0^t e^{-\kappa s}  \| \bm{B} -  B(\overline{u})\|_{L_2^0}^2 \, ds \Big]                                                                            \\
            & + 2\mathbb{E}\Big[\int_0^t e^{-\kappa s}  \langle B(\Pi_{\mathcal{D}_m} \underline{u}_{\mathcal{D}_m}) -  \bm{B},  B(\overline{u}) \rangle_{L_2^0} \, ds \Big] \\
    \leq \, &
    2C_{\operatorname{Lip}}\mathbb{E}\Big[\int_0^t e^{-\kappa s} \|\Pi_{\mathcal{D}_m} \underline{u}_{\mathcal{D}_m} -  \Pi_{\mathcal{D}_m} {u}_{\mathcal{D}_m}\|_{L^2(\mathcal{O})}^2 \, ds \Big]
    + 2C_{\operatorname{Lip}}\mathbb{E}\Big[\int_0^t e^{-\kappa s} \|\Pi_{\mathcal{D}_m} {u}_{\mathcal{D}_m} -  \overline{u}\|_{L^2(\mathcal{O})}^2 \, ds \Big]              \\
            & - \mathbb{E}\Big[\int_0^t e^{-\kappa s}  \| \bm{B} -  B(\overline{u})\|_{L_2^0}^2 \, ds \Big]
    + 2\mathbb{E}\Big[\int_0^t e^{-\kappa s}  \langle B(\Pi_{\mathcal{D}_m} \underline{u}_{\mathcal{D}_m}) -  \bm{B},  B(\overline{u}) \rangle_{L_2^0} \, ds \Big].
\end{align*}

In \eqref{eq-proof-thm:cv 1 identify limits 2},
choosing \(\kappa > 2 C_{\operatorname{Lip}}\),
we can then write
\begin{align}
            & e^{-\kappa t} \mathbb{E} \big[
    \| \Pi_{\mathcal{D}_m}u_{\mathcal{D}_m}(t) - \overline{u}(t)\|_{L^2(\mathcal{O})}^2
    \big]                              \nonumber                                                                                                                                                                                   \\
    +       & (\kappa - 2C_{\operatorname{Lip}})\mathbb{E} \Big[ \int_0^t e^{-\kappa s} \| \Pi_{\mathcal{D}_m}u_{\mathcal{D}_m}(s) -  \overline{u}(s)\|_{L^2(\mathcal{O})}^2 ds\Big]
    \nonumber                                                                                                                                                                                                                      \\
    +       &
    2\mathbb{E} \Big[\int_0^t e^{-\kappa s}\langle \bm{a} (x,  \nabla_{\mathcal{D}_m} u_{\mathcal{D}_m}) - \bm{a} (x, \nabla \overline{u}) , \nabla_{\mathcal{D}_m}u_{\mathcal{D}_m} - \nabla \overline{u} \rangle \Big] \nonumber \\
    +       &
    \mathbb{E}\Big[\int_0^t e^{-\kappa s}  \| \bm{B} -  B(\overline{u})\|_{L_2^0}^2 \, ds \Big]
    \nonumber                                                                                                                                                                                                                      \\
    \leq \, &
    \widehat{\varepsilon}_m(t), \label{eq-proof-thm:cv 1 identify limits 3}
\end{align}
where
\begin{align*}
    \widehat{\varepsilon}_m(t)
    := \, &
    \varepsilon_m(t)
    +
    2C_{\operatorname{Lip}}\mathbb{E}\Big[\int_0^t e^{-\kappa s} \|\Pi_{\mathcal{D}_m} \underline{u}_{\mathcal{D}_m} -  \Pi_{\mathcal{D}_m} {u}_{\mathcal{D}_m}\|_{L^2(\mathcal{O})}^2 \, ds \Big] \\ 
          & +
    2\mathbb{E}\Big[\int_0^t e^{-\kappa s}  \langle B(\Pi_{\mathcal{D}_m} \underline{u}_{\mathcal{D}_m}) -  \bm{B},  B(\overline{u}) \rangle_{L_2^0} \, ds \Big].
\end{align*} 
Let \(s \in [0,T]\) and let \((s_m)_m \subset [0,T]\) be any sequence such that \(s_m \rightarrow s\).
Suppose, for now, that
\begin{equation}\label{eq-proof-thm:cv 1 eps claim}
    \widehat{\varepsilon}_m(s_m) \rightarrow 0
    \quad \text{as } m \rightarrow \infty,
\end{equation}
the proof of which is postponed to the end of this proof.
Since all the terms on the left-hand side of \eqref{eq-proof-thm:cv 1 identify limits 3} are non-negative,
evaluating this relation at \(t = s_m\) gives
\begin{equation}\label{eq-proof-thm:cv 1 moving}
    \mathbb{E} \big[
    \| \Pi_{\mathcal{D}_m}u_{\mathcal{D}_m}(s_m) - \overline{u}(s_m)\|_{L^2(\mathcal{O})}^2
    \big]
    \leq e^{\kappa T} \widehat{\varepsilon}_m(s_m)
    \rightarrow 0 .
\end{equation}
In addition,
in \eqref{eq-proof-thm:cv 1 identify limits 3}, 
taking \(t = T\) leads to
\[\text{\(\bm{B} = B(\overline{u})\), \(dt \otimes d\mathbb{P}\)-almost everywhere,}\]
and
\begin{equation}\label{eq-proof-thm:cv 1 monotonicity}
    \lim_{m\rightarrow \infty} \mathbb{E} \Big[\int_0^T \langle \bm{a} (x,  \nabla_{\mathcal{D}_m} u_{\mathcal{D}_m}) - \bm{a} (x, \nabla \overline{u}) , \nabla_{\mathcal{D}_m}u_{\mathcal{D}_m} - \nabla \overline{u} \rangle \Big] = 0.
\end{equation}

Now consider sequence \((s_m)_m\) with \(s_m \rightarrow s\).
Recalling that
\(\overline{u}\) has \(\mathbb{P}\)-almost surely continuous paths in \(L^2(\mathcal{O})\),
so that \(\|\overline{u}(s_m) - \overline{u}(s)\|_{L^2(\mathcal{O})}^2 \rightarrow 0\) \(\mathbb{P}\)-almost surely as \(s_m \rightarrow s\),
combining with the estimate \eqref{eq-lem:continuous noise uniform bound},
the dominated convergence theorem gives
\begin{equation}\label{eq-proof-thm:cv 1 moving limit}
    \lim_{m \rightarrow \infty}\| \Pi_{\mathcal{D}_m}u_{\mathcal{D}_m}(s_m) - \overline{u}(s)  \|_{L^2(\Omega; L^2(\mathcal{O}))}
    = 0
\end{equation}
Then, by \cite[Lemma C.13]{droniou2018gradient} with compact set \([0,T]\) endowed with the metric induced by \(L^2(\Omega; L^2(\mathcal{O}))\)-norm,
we have
\begin{equation}\label{eq-proof-thm:cv 1 unifrom on time}
    \lim_{m \rightarrow \infty} \sup_{s\in [0,T]}\| \Pi_{\mathcal{D}_m}u_{\mathcal{D}_m}(s) - \overline{u}(s)  \|_{L^2(\Omega; L^2(\mathcal{O}))}
    = 0.
\end{equation}
\color{black}

Next, we identify \(\bm{A}\).
Rewrite \eqref{eq-proof-thm:cv 1 identify limits 1} into
\begin{equation}\label{equation:ID limits eq1}
    \begin{aligned}
                & -  \mathbb{E} \Big[ \int_0^T  2\langle \bm{a} (x,  \nabla_{\mathcal{D}_m} u_{\mathcal{D}_m}), \nabla_{\mathcal{D}_m}u_{\mathcal{D}_m} \rangle \, ds\Big]  \\
        \geq \, & \mathbb{E} \Big[ \| \Pi_{\mathcal{D}_m}u_{\mathcal{D}_m}(T)\|_{L^2(\mathcal{O})}^2 -  \| \Pi_{\mathcal{D}_m}u_{\mathcal{D}_m}^{(0)}\|_{L^2(\mathcal{O})}^2
        - \int_0^T  \| B(\Pi_{\mathcal{D}_m} \underline{u}_{\mathcal{D}_m})\|_{L_2^0}^2 \, ds \Big].
    \end{aligned}
\end{equation}
Now it is possible to apply Minty's trick.
Consider for  \(\Phi \in L^p(\Omega_T; L^p(\mathcal{O}))^d\),
\begin{equation}\label{equation:ID limits functional}
    0 
    \geq 
    -\mathbb{E}
    \Big[
        \int_0^T 2
        \langle
        \nabla_{\mathcal{D}_m} u_{\mathcal{D}_m} - \Phi,
        \bm{a} (x,  \nabla_{\mathcal{D}_m} u_{\mathcal{D}_m}) - \bm{a} (x, \Phi)
        \rangle
        \Big]
    \, ds.
\end{equation}
Expanding \eqref{equation:ID limits functional} then combining it with \eqref{equation:ID limits eq1}, we obtain
\begin{align*}
    0 \geq  \, & \mathbb{E} \Big[ \int_0^T 2  
        \langle \Phi, \bm{a} (x, \nabla_{\mathcal{D}_m} u_{\mathcal{D}_m}) - \bm{a}(x,  \Phi) \rangle
    +  2 \langle \nabla_{\mathcal{D}_m} u_{\mathcal{D}_m} , \bm{a}(x, \Phi)  \rangle  \, ds \bigg]                                                                                   \\
               & -  \mathbb{E} \Big[ \int_0^T 2  
               \langle \nabla_{\mathcal{D}_m} u_{\mathcal{D}_m}  , \bm{a}(x, \nabla_{\mathcal{D}_m} u_{\mathcal{D}_m}) \rangle  \, ds 
               \Big] \\
    \geq \,    &
    \mathbb{E} \Big[\int_0^T 2  \langle \Phi, \bm{a}(x, \nabla_{\mathcal{D}_m} u_{\mathcal{D}_m}) - \bm{a}(x, \Phi) \rangle
    +  2 \langle \nabla_{\mathcal{D}_m} u_{\mathcal{D}_m} , \bm{a}(x, \Phi) \rangle  \,ds \Big]                                                                                 \\
               & + \mathbb{E} \Big[ \| \Pi_{\mathcal{D}_m}u_{\mathcal{D}_m}(T)\|_{L^2(\mathcal{O})}^2
    - \| \Pi_{\mathcal{D}_m}u_{\mathcal{D}_m}^{(0)}\|_{L^2(\mathcal{O})}^2
    - \int_0^T  \| B(\Pi_{\mathcal{D}_m} \underline{u}_{\mathcal{D}_m})\|_{L_2^0}^2 \, ds \Big].
\end{align*}
Taking the limit as \(m \rightarrow \infty\),
recalling the weak convergences (\ref{eq-lem:continuous noise cv1 2}, \ref{eq-lem:continuous noise cv1 3}),
leads to
\begin{align*}\label{equation:ID limits eq2}
    0  \geq  \mathbb{E} \Big[ \int_0^T 2  \langle \Phi, \bm{A} - \bm{a}(x, \Phi) \rangle
    +  2 \langle \nabla \overline{u} , \bm{a} (x,   \Phi) \rangle  \,ds 
    + \| \overline{u}(T)\|_{L^2(\mathcal{O})}^2
    - \| \overline{u}_0\|_{L^2(\mathcal{O})}^2
    - \int_0^T  \| B(\overline{u})\|_{L_2^0}^2  ds\Big].
\end{align*}
Combining this and \eqref{eq-lem:continuous noise cv2 2}
leads to
\begin{equation*}\label{equation:ID limits eq4}
    0
    \geq - \mathbb{E} \Big[\int_0^T 2  \langle \nabla \overline{u} - \Phi, \bm{A} - \bm{a}(x, \Phi) \rangle \, ds\Big],
    \quad \forall \Phi \in  L^p(\Omega_T;L^p(\mathcal{O}))^d.
\end{equation*}
Let \(\Phi =  \nabla \overline{u} + \alpha \bm{\phi}\) with \(\bm{\phi} \in L^p(\Omega_T;L^p(\mathcal{O}))^d\) and \(\alpha \geq 0\), then
\begin{equation*}
    0  \leq \mathbb{E} \Big[ \int_0^T   \langle - \alpha \bm{\phi}, \bm{A} - \bm{a}(x, \nabla \overline{u} + \alpha \bm{\phi}) \rangle \, ds \Big].
\end{equation*}
Dividing both sides by \(\alpha\) and letting \(\alpha \rightarrow 0\), by the dominated convergence theorem and the continuity of \(\bm{a}(x,\cdot)\),
we have
\begin{equation*}
    0 \leq -  \mathbb{E} \Big[ \int_0^T  \langle \bm{\phi}, \bm{A} - \bm{a}(x, \nabla \overline{u}) \rangle \, ds \Big].
\end{equation*}
Repeating the same argument with \(\alpha < 0\) leads to
\begin{equation*}
    \mathbb{E} \Big[ \int_0^T  \langle \bm{\phi} , \bm{A} - \bm{a}(x, \nabla \overline{u}) \rangle \, ds \Big]  = 0, \quad \forall \bm{\phi} \in   L^p(\Omega_T;L^p(\mathcal{O})^d).
\end{equation*}
Hence, \( \bm{A} = \bm{a}(x, \nabla \overline{u})\), for almost every \(x \in \mathcal{O}\), \(dt \otimes d\mathbb{P}\)-almost surely. 

It remains to verify that \(\widehat{\varepsilon}_m(t_m) \rightarrow 0\) as \(m \rightarrow \infty\),
for every \(t \in [0,T]\) and every sequence \((t_m)_m \subset [0,T]\) such that \(t_m \rightarrow t\).
The first term of \(\varepsilon_m(t_m)\)  vanishes by the strong convergence \(\Pi_{\mathcal{D}_m}u_{\mathcal{D}_m}^{(0)} \rightarrow \overline{u}_0\) in \(L^2(\Omega; L^2(\mathcal{O}))\).
The term involving \(\langle \Pi_{\mathcal{D}_m}u_{\mathcal{D}_m}(t_m) - \overline{u}(t_m), \overline{u}(t_m) \rangle\)
is considered by writing the following relation,
\begin{equation*}
    \langle a_m - b_m, b_m \rangle
    = \langle a_m - b, b \rangle
    + \langle b - b_m, b \rangle
    + \langle a_m - b_m, b_m - b \rangle,
\end{equation*}
with \(a_m := \Pi_{\mathcal{D}_m}u_{\mathcal{D}_m}(t_m)\),
\(b_m := \overline{u}(t_m)\),
and \(b := \overline{u}(t)\).
Since \(t_m \rightarrow t\),
\(a_m \rightarrow b\) weakly in \(L^2(\Omega; L^2(\mathcal{O}))\) by Lemma \ref{lem:weak cv moving times},
\(b_m \rightarrow b\) strongly in  \(L^2(\Omega; L^2(\mathcal{O}))\) by continuity of \(\overline{u}\),
\(a_m \rightarrow b_m\) weakly in \(L^2(\Omega; L^2(\mathcal{O}))\) by \eqref{eq-lem:continuous noise cv2 2 2},
the expectation of these terms vanishes.
For the time shift term,
since \(e^{-\kappa s} \leq 1\), and since \(\Pi_{\mathcal{D}_m} \underline{u}_{\mathcal{D}_m}(s) = \Pi_{\mathcal{D}_m} u_{\mathcal{D}_m}^{(n)}\) and \(\Pi_{\mathcal{D}_m} u_{\mathcal{D}_m}(s) = \Pi_{\mathcal{D}_m} u_{\mathcal{D}_m}^{(n+1)}\) for \(s \in (t^{(n)}, t^{(n+1)}]\),
by the estimates \eqref{eq-lem:continuous noise apriori estimate time translate} and \eqref{eq-lem:continuous noise apriori estimate},
we have
\begin{align*}
            & \mathbb{E}\Big[\int_0^t e^{-\kappa s} \|\Pi_{\mathcal{D}_m} \underline{u}_{\mathcal{D}_m} -  \Pi_{\mathcal{D}_m} {u}_{\mathcal{D}_m}\|_{L^2(\mathcal{O})}^2 \, ds \Big]
    \leq
    \mathbb{E} \bigg[ \sum_{n=0}^{N_m-1} k_{\mathcal{D}_m} \| \Pi_{\mathcal{D}_m}u_{\mathcal{D}_m}^{(n+1)} -  \Pi_{\mathcal{D}_m} u_{\mathcal{D}_m}^{(n)}\|_{L^2(\mathcal{O})}^2 \bigg] \\
    = \,    &
    \mathbb{E} \bigg[ \sum_{n=1}^{N_m-1} k_{\mathcal{D}_m} \| \Pi_{\mathcal{D}_m}u_{\mathcal{D}_m}^{(n+1)} -  \Pi_{\mathcal{D}_m} u_{\mathcal{D}_m}^{(n)}\|_{L^2(\mathcal{O})}^2 \bigg]
    + k_{\mathcal{D}_m}\mathbb{E} \Big[ \| \Pi_{\mathcal{D}_m}u_{\mathcal{D}_m}^{(1)} -  \Pi_{\mathcal{D}_m} u_{\mathcal{D}_m}^{(0)}\|_{L^2(\mathcal{O})}^2 \Big]                       \\
    \leq \, & C k_{\mathcal{D}_m}
    + 4k_{\mathcal{D}_m} \mathbb{E} \Big[\max_{0\leq n \leq N_m} \|\Pi_{\mathcal{D}_m} u_{\mathcal{D}_m}^{(n)} \|_{L^2(\mathcal{O})}^2 \Big]
    \leq C k_{\mathcal{D}_m}
    \rightarrow 0 \text{ as \(m \rightarrow \infty\)}.
\end{align*}
The remaining terms all
tend to \(0\) as \(m \rightarrow \infty\) by Lemma \ref{lem:uniform weak cv}.

\subsection*{Proof of Theorem \ref{thm:cv 1.1}}
Recalling \eqref{eq-proof-thm:cv 1 monotonicity},
since the integrand is nonnegative,
\begin{equation*}
    [\bm{a} (x,  \nabla_{\mathcal{D}_m} u_{\mathcal{D}_m}) - \bm{a} (x, \nabla \overline{u})]\cdot (\nabla_{\mathcal{D}_m}u_{\mathcal{D}_m} - \nabla \overline{u}) \rightarrow 0
    \text{ in \(L^1(\Omega \times (0,T) \times \mathcal{O})\).}
\end{equation*}
Thanks to the strict monotonicity, by \cite[Lemma 2.47]{droniou2018gradient},
along a subsequence as \(m \rightarrow \infty\),
\begin{equation*}
    \nabla_{\mathcal{D}_m} u_{\mathcal{D}_m} \rightarrow \nabla \overline{u}
    \text{ a.e. on \(\Omega \times (0,T) \times \mathcal{O}\).}
\end{equation*}
This convergence implies that 
\begin{equation*}
    \bm{a} (x,  \nabla_{\mathcal{D}_m} u_{\mathcal{D}_m}) \cdot \nabla_{\mathcal{D}_m} u_{\mathcal{D}_m} 
    \rightarrow
    \bm{a} (x,  \nabla \overline{u}) \cdot \nabla \overline{u}
    \text{ a.e. on \(\Omega \times (0,T) \times \mathcal{O}\).}
\end{equation*}
Recalling \eqref{equation:ID limits eq1} and the convergence \eqref{eq-proof-thm:cv 1 unifrom on time},
taking the limit superior as \(m \rightarrow \infty\),
\begin{equation*}
    \limsup_{m \rightarrow \infty} \mathbb{E} \Big[ \int_0^T \langle \bm{a} (x,  \nabla_{\mathcal{D}_m} u_{\mathcal{D}_m}), \nabla_{\mathcal{D}_m}u_{\mathcal{D}_m} \rangle \, ds\Big] 
    \leq 
    \mathbb{E} \Big[ \int_0^T \langle \bm{a} (x,  \nabla \overline{u}), \nabla \overline{u} \rangle \, ds\Big];
\end{equation*}
since the integrand is nonnegative and converges almost everywhere, by Fatou's lemma, 
\begin{equation*}
    \mathbb{E} \Big[ \int_0^T \langle \bm{a} (x,  \nabla \overline{u}), \nabla \overline{u} \rangle \, ds\Big]
    \leq
    \liminf_{m \rightarrow \infty} \mathbb{E} \Big[ \int_0^T \langle \bm{a} (x,  \nabla_{\mathcal{D}_m} u_{\mathcal{D}_m}), \nabla_{\mathcal{D}_m}u_{\mathcal{D}_m} \rangle \, ds\Big];
\end{equation*}
then
\begin{equation*}
    \lim_{m \rightarrow \infty} \mathbb{E} \Big[ \int_0^T \langle \bm{a} (x,  \nabla_{\mathcal{D}_m} u_{\mathcal{D}_m}), \nabla_{\mathcal{D}_m}u_{\mathcal{D}_m} \rangle \, ds\Big] 
    =
    \mathbb{E} \Big[ \int_0^T \langle \bm{a} (x,  \nabla \overline{u}), \nabla \overline{u} \rangle \, ds\Big].
\end{equation*}
By \cite[Lemma 2.48]{droniou2018gradient},
\begin{equation*}
    \bm{a} (x,  \nabla_{\mathcal{D}_m} u_{\mathcal{D}_m}) \cdot \nabla_{\mathcal{D}_m} u_{\mathcal{D}_m}
    \rightarrow
    \bm{a}(x, \nabla \overline{u}) \cdot \nabla \overline{u}
    \text{ in \(L^1(\Omega \times (0,T) \times \mathcal{O})\).}
\end{equation*}
This convergence gives the equi-integrability of the sequence \((\bm{a} (x,  \nabla_{\mathcal{D}_m} u_{\mathcal{D}_m}) \cdot \nabla_{\mathcal{D}_m} u_{\mathcal{D}_m})_m\),
which further implies the equi-integrability of \((|\nabla_{\mathcal{D}_m} u_{\mathcal{D}_m}|^p)_m\) by the coercivity property of \(\bm{a}\).
Then the strong convergence of \(\nabla_{\mathcal{D}_m} u_{\mathcal{D}_m}\) to \(\nabla \overline{u}\) is a consequence of Vitali's theorem.

\section{Proof of Theorem \ref{thm:LC noise error}}\label{sec:proof-LC cv error}
Error estimates can be split into two parts, one is the interpolation error that occurs when projecting the exact solution, the other is the consistency error that occurs when plugging the interpolation of the exact solution to the gradient scheme.
Since interpolation properties of \(\mathcal{P}_{\mathcal{D}}\) are known, it remains to investigate the consistency error.
In this section, we denote  \(f \lesssim g\) to mean \(f \leq C g\),
where \(C\) depending on \(\mathcal{O}\), \(T\), \(\alpha_{\bm{a}}\), \(\beta_{\bm{a}}\), \(\gamma_{\bm{a}}\), and the exact solution \(\overline{u}\),
but independent of quantities \(f\) and \(g\).
\begin{lemma}[Stability]\label{lem:stability unified}
    Let \(\mathfrak{T}\) be admissible, \(u_1, u_2 \in X_{\mathcal{D}}^N\), \(u_{12} := u_1 - u_2\), and
    suppose that for all \(n\), all \(\phi_{\mathcal{D}} \in X_{\mathcal{D}}\), 
    \begin{align*}
        \big\langle \delta_{\mathcal{D}}^{(n+1)} u_{12}, \Pi_{\mathcal{D}}\phi_{\mathcal{D}} \big\rangle
         & + \kappa^{(n+1)} \big\langle \bm{a}(x,\nabla_{\mathcal{D}} u_1^{(n+1)}) - \bm{a}(x,\nabla_{\mathcal{D}} u_2^{(n+1)}), \nabla_{\mathcal{D}} \phi_{\mathcal{D}} \big\rangle \\
         & = \big\langle [B(\Pi_{\mathcal{D}} u_1^{(n)}) - B(\Pi_{\mathcal{D}} u_2^{(n)})] \Delta^{(n+1)}W, \Pi_{\mathcal{D}}\phi_{\mathcal{D}} \big\rangle + R^{(n+1)}(u_{12}).
    \end{align*}
    Then, with \(\mathscr{E}_{\bm{a}}\) defined below in \eqref{eq:stability monotonicity term},
    \begin{align}\label{eq:stab master}
           & \mathbb{E}_{\mathfrak{T}}\Big[ \max_{1 \le n \le N} \|\Pi_{\mathcal{D}} u_{12}^{(n)}\|_{L^2(\mathcal{O})}^2 \Big]
        + \mathbb{E}_{\mathfrak{T}}\bigg[ \sum_{n=0}^{N-1} \|\delta_{\mathcal{D}}^{(n+1)} u_{12}\|_{L^2(\mathcal{O})}^2 \bigg]
        + \mathscr{E}_{\bm{a}}(u_1,u_2)                                          \nonumber                                     \\
        \lesssim
        \, & \mathbb{E}\big[ \|\Pi_{\mathcal{D}} u_{12}^{(0)}\|_{L^2(\mathcal{O})}^2 \big]
        + \mathbb{E}_{\mathfrak{T}}\bigg[ \max_{0 \le \ell \le N-1} \sum_{n=0}^{\ell} R^{(n+1)}(u_{12}) \bigg].
    \end{align}
\end{lemma}

\begin{proof}
    This can be obtained \(\overline{\mathbb{P}}\)-a.s.\ exactly as in Lemma \ref{lem:continuous noise apriori estimate}. 
    Then applying \(\overline{\mathbb{E}}\) concludes the proof.
\end{proof} 

Take \(v_1, v_2 \in X_{\mathcal{D}}\),
set \(v_{12}  := v_1 - v_2\),
\begin{align}
    \mathscr{E}_{\bm{a}}(v_1,v_2) & := \mathbb{E}_{\mathfrak{T}}\bigg[ \sum_{n=0}^{N-1} \kappa^{(n+1)}
    \big\langle \bm{a}(x,\nabla_{\mathcal{D}} v_1^{(n+1)}) - \bm{a}(x,\nabla_{\mathcal{D}} v_2^{(n+1)}), \nabla_{\mathcal{D}} v_{12}^{(n+1)} \big\rangle \bigg],   \label{eq:stability monotonicity term} \\
    \mathscr{G}(v_{12})           & := \mathbb{E}_{\mathfrak{T}}\bigg[ \sum_{n=0}^{N-1} \kappa^{(n+1)} \|\nabla_{\mathcal{D}} v_{12}^{(n+1)}\|_{L^p(\mathcal{O})}^p \bigg],                           \label{eq:stability monotonicity term 2}     \\
    \mathscr{U}(v_1,v_2)          & := \mathbb{E}_{\mathfrak{T}}\bigg[ \sum_{n=0}^{N-1} \kappa^{(n+1)}
    \Big( \|\nabla_{\mathcal{D}} v_1^{(n+1)}\|_{L^p(\mathcal{O})}^p + \|\nabla_{\mathcal{D}} v_2^{(n+1)}\|_{L^p(\mathcal{O})}^p \Big) \bigg].
\end{align}

\begin{lemma}\label{cor:integrated monotonicity}
    Let Assumption \ref{def:assumption opertator a 2}(iii) holds.
    For all \(v_1, v_2 \in X_{\mathcal{D}}^N\):
    \begin{enumerate}[label={(\roman*)}, leftmargin=*, align=left]
        \item if \(p \ge 2\),
              \begin{equation}\label{eq:stab gradient p geq 2}
                  \mathscr{G}(v_{12}) \;\le\; \gamma_{\bm{a}}^{-1}\, \mathscr{E}_{\bm{a}}(v_1,v_2)
              \end{equation}
        \item if \(1 < p < 2\),
              \begin{equation}\label{eq:stab gradient}
                  \mathscr{G}(v_{12}) \;\le\; \gamma_{\bm{a}}^{-p/2}\, 2^{(p-1)(2-p)/2}\;
                  \mathscr{E}_{\bm{a}}(v_1,v_2)^{p/2}\; \mathscr{U}(v_1,v_2)^{(2-p)/2}.
              \end{equation}
    \end{enumerate}
    In addition,  
    if \(\mathscr{U}(v_1,v_2) \lesssim 1\), then
    \begin{equation*}
        \mathscr{G}(v_{12})^{\sigma} \lesssim  \mathscr{E}_{\bm{a}}(v_1,v_2),
        \quad 
        \sigma :=
        \begin{cases}
            2/p & \text{if \(1 < p < 2\),} \\
            1   & \text{if \(p \geq 2\).}
        \end{cases}
    \end{equation*}
\end{lemma}

\begin{proof}
    Estimates (\ref{eq:stab gradient p geq 2}, \ref{eq:stab gradient}) follows from Hölder inequality and Lemma \ref{lem:ineq LL op2} with \(\bm{\mu} = \nabla_{\mathcal{D}} v_1^{(n+1)}\) and \(\bm{\nu} = \nabla_{\mathcal{D}} v_2^{(n+1)}\). 
\end{proof} 

Set
\begin{equation}\label{eq:master A}
    \mathfrak{E}_{P_D} :=
    \mathbb{E}_{\mathfrak{T}} \bigg[ \sum_{n=0}^{N-1} \int_{\mathfrak{t}^{(n)}}^{\mathfrak{t}^{(n+1)}}
    \big\| \nabla_{\mathcal{D}} \mathcal{P}_{\mathcal{D}} \overline{u}(\mathfrak{t}^{(n+1)}) - \nabla \overline{u}(s) \big\|_{L^p(\mathcal{O})}^p \, ds \bigg].
\end{equation}

\begin{lemma}\label{lem:master A}
    Under (H2) and (H3),
    \begin{equation}\label{eq:A bound}
        \mathfrak{E}_{P_D} \lesssim h_{\mathcal{D}}^p + k_{\mathcal{D}}^{\gamma p}.
    \end{equation}
\end{lemma}

\begin{proof}
    Consider
    \[\nabla_{\mathcal{D}} \mathcal{P}_{\mathcal{D}} \overline{u}(\mathfrak{t}^{(n+1)}) - \nabla\overline{u}(s)
        = \big[ \nabla_{\mathcal{D}} \mathcal{P}_{\mathcal{D}} \overline{u}(\mathfrak{t}^{(n+1)}) - \nabla\overline{u}(\mathfrak{t}^{(n+1)}) \big]
        + \big[ \nabla\overline{u}(\mathfrak{t}^{(n+1)}) - \nabla\overline{u}(s) \big].\]
    Hypothesis (H3) implies that 
    \begin{equation*}
        \|  \nabla_{\mathcal{D}} \mathcal{P}_{\mathcal{D}} \overline{u}(\mathfrak{t}^{(n+1)}) - \nabla\overline{u}(\mathfrak{t}^{(n+1)})   \|_{L^p(\mathcal{O})} \le C h_{\mathcal{D}} \|\overline{u}(\mathfrak{t}^{(n+1)})\|_{W_s}.
    \end{equation*}
    Together with \eqref{eq:G1} leads to 
    \begin{equation*}
        \mathbb{E}_{\mathfrak{T}}\bigg[ \sum_{n=0}^{N-1} \int_{\mathfrak{t}^{(n)}}^{\mathfrak{t}^{(n+1)}}
                  \|  \nabla_{\mathcal{D}} \mathcal{P}_{\mathcal{D}} \overline{u}(\mathfrak{t}^{(n+1)}) - \nabla\overline{u}(\mathfrak{t}^{(n+1)}) \|_{L^p(\mathcal{O})}^p \, ds \bigg]
                  \leq 
                  C h_{\mathcal{D}}^p \|\overline{u}\|_{\mathbb{W}^{\gamma,p}(W_s)}^p
    \end{equation*} 
    Next, by \eqref{eq:G2} with \(g = \nabla \overline{u}\), 
    we have  
    \begin{equation*}
        \mathbb{E}_{\mathfrak{T}}\bigg[ \sum_{n=0}^{N-1} \int_{\mathfrak{t}^{(n)}}^{\mathfrak{t}^{(n+1)}}
                  \|  \nabla \overline{u}(\mathfrak{t}^{(n+1)}) -  \nabla \overline{u}(s) \|_{L^p(\mathcal{O})}^p \, ds \bigg]
                  \;\le\; C \, k_{\mathcal{D}}^{\gamma p} \, \|  \nabla \overline{u} \|_{\mathbb{W}^{\gamma,p}(L^p(\mathcal{O}))}^p.
    \end{equation*}
    Now the proof completes by the Minkowski's Inequality.
\end{proof}

\begin{lemma}[Consistency]\label{lem:consistency unified}
    Let \(\mathcal{P}_{\mathcal{D}}\) satisfy (H3) and set, for \(n = 0, \ldots, N-1\),
    \begin{align*}
        \mathcal{E}^{(n+1)}(\phi_{\mathcal{D}}^{(n+1)})
        := \, & \big\langle \delta_{\mathcal{D}}^{(n+1)} \mathcal{P}_{\mathcal{D}} \overline{u}, \Pi_{\mathcal{D}} \phi_{\mathcal{D}}^{(n+1)} \big\rangle
        + \kappa^{(n+1)} \big\langle \bm{a}(x, \nabla_{\mathcal{D}} \mathcal{P}_{\mathcal{D}} \overline{u}(\mathfrak{t}^{(n+1)})), \nabla_{\mathcal{D}} \phi_{\mathcal{D}}^{(n+1)} \big\rangle \\
              & - \big\langle B(\Pi_{\mathcal{D}} \mathcal{P}_{\mathcal{D}} \overline{u}(\mathfrak{t}^{(n)})) \Delta^{(n+1)} W, \Pi_{\mathcal{D}} \phi_{\mathcal{D}}^{(n+1)} \big\rangle .
    \end{align*}
    Then, for every \(\xi > 0\),
    \begin{align}\label{eq:consistency unified}
        \mathbb{E}_{\mathfrak{T}} \bigg[ \sum_{n=0}^{N-1} \mathcal{E}^{(n+1)}(\phi_{\mathcal{D}}^{(n+1)}) \bigg]
        \lesssim \; & \xi^{-1} \Big( k_{\mathcal{D}}^{-1} \alpha^2 h_{\mathcal{D}}^2 + T_{\bm{a}}^{\,q} + h_{\mathcal{D}}^{\,q} \Big)
        + \xi \, \|\Pi_{\mathcal{D}} \phi_{\mathcal{D}}^{(0)}\|_{L^2(\Omega;L^2(\mathcal{O}))}^2 + \xi \, \mathscr{G}(\phi_{\mathcal{D}})^{\sigma}\nonumber                                                 \\
                    & + \xi \, \mathbb{E}_{\mathfrak{T}} \bigg[ \max_{1 \le n \le N-1} \|\Pi_{\mathcal{D}} \phi_{\mathcal{D}}^{(n)}\|_{L^2(\mathcal{O})}^2
        + \sum_{n=0}^{N-1} \|\delta_{\mathcal{D}}^{(n+1)} \phi_{\mathcal{D}}\|_{L^2(\mathcal{O})}^2 \bigg],
    \end{align}
    where \(\sigma\) is as in Lemma \ref{cor:integrated monotonicity}, 
    \begin{equation}\label{eq:Ta unified}
        T_{\bm{a}} \;\lesssim\;
        \begin{cases}
            h_{\mathcal{D}}^{\,p-1} + k_{\mathcal{D}}^{\,\gamma(p-1)} & \text{if } 1 < p \le 2, \\[4pt]
            h_{\mathcal{D}} + k_{\mathcal{D}}^{\,\gamma}               & \text{if } p > 2,
        \end{cases}
        \quad \text{and} \quad 
        q := \begin{cases} 
            2 & \text{if \(1 < p \le 2\)},\\
            p' & \text{ if \(p > 2\).}
        \end{cases}
    \end{equation}
\end{lemma}

\begin{proof}
    For \(dt \otimes \mathbb{P}\)-almost any \((t,\omega) \in \Omega_T\), any \(\phi_{\mathcal{D}} \in X_{\mathcal{D}}\),
    invoking Definition \ref{def:GD-limit-conformity}, we write
    \begin{align*}
        \langle \text{div}(\bm{a}(x, \nabla \overline{u})), \Pi_{\mathcal{D}} \phi_{\mathcal{D}} \rangle
        = - \langle \bm{a}(x, \nabla \overline{u}), \nabla_{\mathcal{D}} \phi_{\mathcal{D}} \rangle + \widetilde{W}_{\mathcal{D}}(\bm{a}(x, \nabla \overline{u}), \phi_{\mathcal{D}}).
    \end{align*} For any \(n = 0, \ldots, N-1\), for any \(\phi_{\mathcal{D}}^{(n+1)} \in X_{\mathcal{D}}\), we have
    \begin{align*}
             & \langle \overline{u}(\mathfrak{t}^{(n+1)}) - \overline{u}(\mathfrak{t}^{(n)}), \Pi_{\mathcal{D}}\phi_{\mathcal{D}}^{(n+1)} \rangle
        +
        \int_{\mathfrak{t}^{(n)}}^{\mathfrak{t}^{(n+1)}} \langle \bm{a}(x, \nabla \overline{u}), \nabla_{\mathcal{D}} \phi_{\mathcal{D}}^{(n+1)} \rangle \, ds \\
        = \, &
        \int_{\mathfrak{t}^{(n)}}^{\mathfrak{t}^{(n+1)}} \langle B(\overline{u})\,  dW_s, \Pi_{\mathcal{D}} \phi_{\mathcal{D}}^{(n+1)} \rangle
        +
        \int_{\mathfrak{t}^{(n)}}^{\mathfrak{t}^{(n+1)}} \widetilde{W}_{\mathcal{D}}(\bm{a}(x, \nabla \overline{u}), \phi_{\mathcal{D}}^{(n+1)}) \, ds.
    \end{align*}
    Taking the difference between these and \(\mathcal{E}\) respectively,
    we can write, for any \(n = 0, \ldots, N-1\), for any \(\phi_{\mathcal{D}}^{(n+1)} \in X_{\mathcal{D}}\),
    \begin{align*}
        \mathcal{E}^{(n + 1)}(\phi^{(n + 1)})
        =
         & \langle \delta_{\mathcal{D}}^{(n+1)} \mathcal{P}_{\mathcal{D}} \overline{u} - (\overline{u}(\mathfrak{t}^{(n+1)}) - \overline{u}(\mathfrak{t}^{(n)})) , \Pi_{\mathcal{D}}\phi_{\mathcal{D}}^{(n+1)} \rangle                                          \\
         & +
        \int_{\mathfrak{t}^{(n)}}^{\mathfrak{t}^{(n+1)}} \langle \bm{a}(x,\nabla_{\mathcal{D}} \mathcal{P}_{\mathcal{D}} \overline{u}(\mathfrak{t}^{(n + 1)})) - \bm{a}(x, \nabla \overline{u}) , \nabla_{\mathcal{D}} \phi_{\mathcal{D}}^{(n+1)} \rangle \, ds \\
         & -
        \int_{\mathfrak{t}^{(n)}}^{\mathfrak{t}^{(n+1)}} \langle (B(\Pi_{\mathcal{D}}\mathcal{P}_{\mathcal{D}} \overline{u}) - B(\overline{u})) \,  dW_s, \Pi_{\mathcal{D}} \phi_{\mathcal{D}}^{(n+1)} \rangle                                                  \\
         & -
        \int_{\mathfrak{t}^{(n)}}^{\mathfrak{t}^{(n+1)}} \widetilde{W}_{\mathcal{D}}(\bm{a}(x, \nabla \overline{u}), \phi_{\mathcal{D}}^{(n+1)}) \, ds.
    \end{align*}
    Summing over \(n\) from \(0\) to \(N-1\) and taking \(\mathbb{E}_{\mathfrak{T}}\) leads to
    \begin{align*}
           & \mathbb{E}_{\mathfrak{T}} \bigg[ \sum_{n=0}^{N-1} \mathcal{E}_{(i)}^{(n + 1)}(\phi^{(n + 1)}) \bigg]                                                                                                                                                                                                         \\
        =
        \, & \mathbb{E}_{\mathfrak{T}} \bigg[ \sum_{n=0}^{N-1} \langle \delta_{\mathcal{D}}^{(n+1)} \mathcal{P}_{\mathcal{D}} \overline{u} - (\overline{u}(\mathfrak{t}^{(n+1)}) - \overline{u}(\mathfrak{t}^{(n)})), \Pi_{\mathcal{D}}\phi_{\mathcal{D}}^{(n+1)} \rangle \bigg]                                          \\
           & +
        \mathbb{E}_{\mathfrak{T}} \bigg[ \sum_{n=0}^{N-1}  \int_{\mathfrak{t}^{(n)}}^{\mathfrak{t}^{(n+1)}} \langle \bm{a}(x,\nabla_{\mathcal{D}} \mathcal{P}_{\mathcal{D}} \overline{u}(\mathfrak{t}^{(n + 1)})) - \bm{a}(x, \nabla \overline{u}) , \nabla_{\mathcal{D}} \phi_{\mathcal{D}}^{(n+1)} \rangle \, ds \bigg] \\
           & -
        \mathbb{E}_{\mathfrak{T}} \bigg[ \sum_{n=0}^{N-1} \int_{\mathfrak{t}^{(n)}}^{\mathfrak{t}^{(n+1)}} \langle (B(\Pi_{\mathcal{D}}\mathcal{P}_{\mathcal{D}} \overline{u}) - B(\overline{u})) \,  dW_s, \Pi_{\mathcal{D}} \phi_{\mathcal{D}}^{(n+1)} \rangle \bigg]                                                   \\
           & -
        \mathbb{E}_{\mathfrak{T}} \bigg[ \sum_{n=0}^{N-1} \int_{\mathfrak{t}^{(n)}}^{\mathfrak{t}^{(n+1)}} \widetilde{W}_{\mathcal{D}}(\bm{a}(x, \nabla \overline{u}), \phi_{\mathcal{D}}^{(n+1)}) \, ds \bigg]
        =: T_1 + T_2 + T_3 + T_4.
    \end{align*}

    \noindent \underline{Estimate for \(T_1\).}
    Write
    \begin{align*}
             & \langle \delta_{\mathcal{D}}^{(n+1)} \mathcal{P}_{\mathcal{D}} \overline{u} - (\overline{u}(\mathfrak{t}^{(n+1)}) - \overline{u}(\mathfrak{t}^{(n)})) , \Pi_{\mathcal{D}}\phi_{\mathcal{D}}^{(n+1)} \rangle                                      \\
        = \, &
        \langle \delta_{\mathcal{D}}^{(n+1)} \mathcal{P}_{\mathcal{D}} \overline{u} - (\overline{u}(\mathfrak{t}^{(n+1)}) - \overline{u}(\mathfrak{t}^{(n)})) , \Pi_{\mathcal{D}}\phi_{\mathcal{D}}^{(n+1)} - \Pi_{\mathcal{D}} \phi_{\mathcal{D}}^{(n)}\rangle \\
             & +
        \langle \delta_{\mathcal{D}}^{(n+1)} \mathcal{P}_{\mathcal{D}} \overline{u} - (\overline{u}(\mathfrak{t}^{(n+1)}) - \overline{u}(\mathfrak{t}^{(n)})) , \Pi_{\mathcal{D}}\phi_{\mathcal{D}}^{(n)} \rangle.
    \end{align*}
    By the Cauchy-Schwarz inequality and the Young inequality, the following can be obtained:
    for any \(\xi > 0\),
    \begin{align*}
        T_1 \lesssim \, & \mathbb{E}_{\mathfrak{T}}  \bigg[ \xi^{-1}
        \| \Pi_{\mathcal{D}} \mathcal{P}_{\mathcal{D}} \overline{u}_0 - \overline{u}_0 \|_{L^2(\mathcal{O})}^2
        + \xi^{-1}  \sum_{n=0}^{N-1} \| \Pi_{\mathcal{D}} \mathcal{P}_{\mathcal{D}} \overline{u} (\mathfrak{t}^{(n+1)} ) -  \overline{u} (\mathfrak{t}^{(n+1)} )\|_{L^2(\mathcal{O})}^2 \bigg] \\
                        & +  \xi\mathbb{E}_{\mathfrak{T}} \bigg[ \sum_{n=0}^{N-1}\| \delta_{\mathcal{D}}^{(n+1)}\phi_{\mathcal{D}} \|_{L^2(\mathcal{O})}^2
        + \| \Pi_{\mathcal{D}}\phi_{\mathcal{D}}^{(0)} \|_{L^2(\mathcal{O})}^2
        + \max_{1 \leq n \leq N-1} \| \Pi_{\mathcal{D}}\phi_{\mathcal{D}}^{(n)} \|_{L^2(\mathcal{O})}^2 \bigg]                                                                                 \\
        \lesssim \,     &
        \xi^{-1} h_{\mathcal{D}} \| \overline{u}_0 \|_{L^2(\Omega:W_s)}^2
        +
        \xi^{-1} k_{\mathcal{D}}^{-1} \alpha^2 h_{\mathcal{D}}^2 \|  \overline{u} \|_{L^2(\Omega_T;W_s)}^2                                                                                     \\
                        & + \xi \| \Pi_{\mathcal{D}} \phi_{\mathcal{D}}^{(0)}\|_{L^2(\Omega; L^2(\mathcal{O}))}^2
        + \xi \mathbb{E}_{\mathfrak{T}}  \bigg[
        \sum_{n=0}^{N-1}\| \delta_{\mathcal{D}}^{(n+1)}\phi_{\mathcal{D}}\|_{L^2(\mathcal{O})}^2
        +
        \max_{1 \leq n \leq N-1} \| \Pi_{\mathcal{D}}\phi_{\mathcal{D}}^{(n)} \|_{L^2(\mathcal{O})}^2
        \bigg].
    \end{align*}

    \noindent \underline{Estimate for \(T_2\).}
    For \(T_2\), the H\"older inequality and Young inequality give
    \[T_2 \lesssim \xi^{-1} T_{\bm{a}}^{\,q} + \xi \mathscr{G}(\phi_{\mathcal{D}})^{\sigma}\]
    with
    \[
        T_{\bm{a}} := \mathbb{E}_{\mathfrak{T}} \bigg[ \sum_{n=0}^{N-1} \int_{\mathfrak{t}^{(n)}}^{\mathfrak{t}^{(n+1)}}
        \big\| \bm{a}(x, \nabla_{\mathcal{D}} \mathcal{P}_{\mathcal{D}} \overline{u}(\mathfrak{t}^{(n+1)})) - \bm{a}(x, \nabla \overline{u}(s)) \big\|_{L^{p'}(\mathcal{O})}^{p'} \, ds \bigg]^{1/p'} .
    \]
    It remains to prove \eqref{eq:Ta unified}. Write
    \[\mathcal{A}_n(s) := \| \nabla_{\mathcal{D}} \mathcal{P}_{\mathcal{D}} \overline{u}(\mathfrak{t}^{(n+1)}) - \nabla \overline{u}(s) \|_{L^p(\mathcal{O})}.\]

    Consider \(1 < p \le 2\).
    By \eqref{eq-lem:ineq LL op1} (with \(\bm{\mu} = \nabla_{\mathcal{D}} \mathcal{P}_{\mathcal{D}} \overline{u}(\mathfrak{t}^{(n+1)})\) and \(\bm{\nu} = \nabla \overline{u}(s)\))
    and
    \((p-1)p' = p\),
    we have
    \[
        T_{\bm{a}}^{p'}
        \lesssim
        \mathbb{E}_{\mathfrak{T}} \bigg[ \sum_n \int_{\mathfrak{t}^{(n)}}^{\mathfrak{t}^{(n+1)}} \mathcal{A}_n(s)^p \, ds \bigg] = \mathfrak{E}_{P_D},
        \]
        which is 
    \[
        T_{\bm{a}} \lesssim \mathfrak{E}_{P_D}^{1/p'} .
    \]

    Consider \(p > 2\). By Assumption \ref{def:assumption opertator a 2}(ii) and the generalised
    H\"older inequality with exponents \((p', p, p/(p-2))\),
    \begin{align*}
         & \big\| \bm{a}(x, \nabla_{\mathcal{D}} \mathcal{P}_{\mathcal{D}} \overline{u}(\mathfrak{t}^{(n+1)})) - \bm{a}(x, \nabla \overline{u}(s)) \big\|_{L^{p'}(\mathcal{O})}
        \lesssim \mathcal{A}_n(s) \, \mathcal{B}_n(s),
    \end{align*}
    where
    \begin{align*}
        \mathcal{B}_n(s) := \| \nabla_{\mathcal{D}} \mathcal{P}_{\mathcal{D}} \overline{u}(\mathfrak{t}^{(n+1)}) \|_{L^p(\mathcal{O})}^{p-2} + \| \nabla \overline{u}(s) \|_{L^p(\mathcal{O})}^{p-2}.
    \end{align*}

    Integrate over \(\overline{\Omega} \times \Omega \times \mathfrak{T}\),
    apply the H\"older inequality with the conjugate pairwith exponents \((p-1, (p-1)/(p-2))\); since
    \(p'(p-1) = p\) and \(p'(p-1)/(p-2) = p/(p-2)\),
    \[
        T_{\bm{a}}^{p'}
        \lesssim
        \mathbb{E}_{\mathfrak{T}} \Big[ \sum_n \int_{\mathfrak{t}^{(n)}}^{\mathfrak{t}^{(n+1)}} \mathcal{A}_n(s)^p \, ds \Big]^{\!\frac{1}{p-1}}
        \mathbb{E}_{\mathfrak{T}} \Big[ \sum_n \int_{\mathfrak{t}^{(n)}}^{\mathfrak{t}^{(n+1)}} \mathcal{B}_n(s)^{\frac{p}{p-2}} ds \Big]^{\!\frac{p-2}{p-1}} .
    \]
    On the right-hand side, the first factor is \(\mathfrak{E}_{P_D}^{1/(p-1)}\);
    the second factor is bounded since
    \[\mathcal{B}_n(s)^{p/(p-2)} \lesssim \| \nabla_{\mathcal{D}} \mathcal{P}_{\mathcal{D}} \overline{u}(\mathfrak{t}^{(n+1)}) \|_{L^p(\mathcal{O})}^{p} + \| \nabla \overline{u}(s) \|_{L^p(\mathcal{O})}^{p}.\]
    Hence
    \[T_{\bm{a}} \lesssim \mathfrak{E}_{P_D}^{1/(p'(p-1))} = \mathfrak{E}_{P_D}^{1/p}.\]

    \noindent \underline{Estimate for \(T_3\).}
    By the Cauchy-Schwarz inequality and the Young inequality,
    and the Itô isometry,
    we have that
    \begin{align*}
        T_3
        \lesssim \, &
        \xi^{-1} \alpha h_{\mathcal{D}} \| \overline{u} \|_{L^2(\Omega_T; W_s)}^2
        +
        \xi^{-1}k_{\mathcal{D}}^{2\gamma} \|\overline{u} \|_{L^2(\Omega; C^{\gamma}([0,T]; L^2(\mathcal{O})))}^2 \\
                    & +
        \xi \mathbb{E}_{\mathfrak{T}} \Big[ \max_{1 \leq n \leq N-1} \|\Pi_{\mathcal{D}} \phi_{\mathcal{D}}^{(n)} \|_{L^2(\mathcal{O})}^2\Big]
        +
        \xi \mathbb{E}_{\mathfrak{T}}
        \bigg[\sum_{n=0}^{N-1} \kappa^{(n-1)} || \delta_{\mathcal{D}}^{(n+1)} \phi_{\mathcal{D}}||_{L^2(\mathcal{O})}^2 \bigg].
    \end{align*}

    \noindent \underline{Estimate for \(T_4\).}
    Invoking Definition \ref{def:GD-limit-conformity} for the limit-conformity,
    applying the Hölder inequality,
    and the Young inequality to the term \(T_4\)
    lead to
    \begin{align*}
        T_4 \lesssim
        \begin{cases}
            \displaystyle
            \xi^{-1}T_W^{2} + \xi \mathbb{E}_{\mathfrak{T}}  \bigg[ \sum_{n=0}^{N-1} \kappa^{(n+1)} \| \nabla_{\mathcal{D}} \phi_{\mathcal{D}}^{(n+1)}\|_{ L^p(\mathcal{O})}^p \bigg]^{\frac{2}{p}}
             & \quad  \text{if \(1 < p < 2\),}
            \\
            \displaystyle
            \xi^{-1}T_W^{p'} + \xi \mathbb{E}_{\mathfrak{T}}  \bigg[ \sum_{n=0}^{N-1} \kappa^{(n+1)} \| \nabla_{\mathcal{D}} \phi_{\mathcal{D}}^{(n+1)}\|_{ L^p(\mathcal{O})}^p \bigg]
             & \quad \text{if \(2 \geq p \),}
        \end{cases}
    \end{align*}
    where
    \begin{align*}
        T_W :=
        \mathbb{E}_{\mathfrak{T}}  \bigg[ \sum_{n=0}^{N-1} \int_{\mathfrak{t}^{(n)}}^{\mathfrak{t}^{(n+1)}} {W}_{\mathcal{D}}(\bm{a}(x, \nabla \overline{u}))^{p'} \, ds \bigg]^{\frac{1}{p'}}
        \lesssim \,
        h_{\mathcal{D}} \|\bm{a}(x, \nabla \overline{u})\|_{L^{p'}(\Omega_T;\bm{W}_w)}.
    \end{align*}
    Combining these estimates concludes the proof.
\end{proof}

\subsection*{Proof of Theorem \ref{thm:LC noise error}}

\begin{proof}
    Apply Lemma \ref{lem:stability unified} and Lemma \ref{cor:integrated monotonicity}
    with \(u_1 = u_{\mathcal{D}}\), \(u_2 = \mathcal{P}_{\mathcal{D}} \overline{u}\),
    so that \(R^{(n+1)} = \mathcal{E}^{(n+1)}(u_{\mathcal{D}} - \mathcal{P}_{\mathcal{D}} \overline{u})\),
    and insert Lemma \ref{lem:consistency unified} with \(\phi_{\mathcal{D}} = u_{\mathcal{D}} - \mathcal{P}_{\mathcal{D}} \overline{u}\).
    Choosing \(\xi\) small enough to absorb the terms with factor \(\xi\) into the left-hand side of
    \eqref{eq:stab master} and \eqref{eq:stab gradient} yields,
    \[
        \mathbb{E}_{\mathfrak{T}}\big[\max_n \|\Pi_{\mathcal{D}} u_{\mathcal{D}}^{(n)} - \Pi_{\mathcal{D}} \mathcal{P}_{\mathcal{D}} \overline{u}(\mathfrak{t}^{(n)})\|_{L^2}^2\big]
        + \mathscr{G}(u_{\mathcal{D}} - \mathcal{P}_{\mathcal{D}} \overline{u})^{\sigma}
        \lesssim (e_{\mathcal{D}}^{\mathrm{ini}})^2 + k_{\mathcal{D}}^{-1} \alpha^2 h_{\mathcal{D}}^2 + T_{\bm{a}}^{\,q} + h_{\mathcal{D}}^{\,q}
        =: \mathcal{Q} .
    \]
    By \eqref{eq:Ta unified},
    \begin{equation*}
        \mathcal{Q} \lesssim
        \begin{cases}
            \mathfrak{E}^2 + h_{\mathcal{D}}^{2(p-1)} + k_{\mathcal{D}}^{2\gamma(p-1)} & \text{if \(p \le 2\) (\(q = 2\)),} \\
            \mathfrak{E}^2 + h_{\mathcal{D}}^{p'} + k_{\mathcal{D}}^{p'\gamma}         & \text{if \(p > 2\) (\(q = p'\)).}
        \end{cases}
    \end{equation*}
    Then, we can have
    \begin{equation*}
        \mathbb{E}_{\mathfrak{T}}\big[\max_n \|\Pi_{\mathcal{D}} u_{\mathcal{D}}^{(n)} - \Pi_{\mathcal{D}} \mathcal{P}_{\mathcal{D}} \overline{u}(\mathfrak{t}^{(n)})\|_{L^2}^2\big]^{1/2}
        \lesssim
        \begin{cases}
            \mathfrak{E} + h_{\mathcal{D}}^{p-1} + k_{\mathcal{D}}^{\gamma(p-1)} & \text{if \(p \le 2\),} \\
            \mathfrak{E} + h_{\mathcal{D}}^{p'/2} + k_{\mathcal{D}}^{p'\gamma/2} & \text{if \(p > 2\);}
        \end{cases}
    \end{equation*}
    and
    \begin{equation*}
        \mathbb{E}_{\mathfrak{T}}\bigg[ \sum_{n=0}^{N-1} \kappa^{(n+1)} \|\nabla_{\mathcal{D}}u_{\mathcal{D}}^{(n+1)} - \nabla_{\mathcal{D}}\mathcal{P}_{\mathcal{D}} \overline{u}(\mathfrak{t}^{(n+1)})\|_{L^p(\mathcal{O})}^p \bigg]^{1/p}
        \lesssim
        \begin{cases}
            \mathfrak{E} + h_{\mathcal{D}}^{p-1} + k_{\mathcal{D}}^{\gamma(p-1)},           & \text{if \(p \leq 2\) ,} \\
            \mathfrak{E}^{2/p} + h_{\mathcal{D}}^{1/(p-1)} + k_{\mathcal{D}}^{\gamma/(p-1)} & \text{if \(p > 2\).}
        \end{cases}
    \end{equation*}
    Next, recall the interpolation property in \((H3)\),
    \begin{align*}
         & \mathbb{E}_{\mathfrak{T}}\Big[\sum_n \kappa^{(n+1)}
        \big\|\Pi_{\mathcal{D}} \mathcal{P}_{\mathcal{D}}\overline{u}(\mathfrak{t}^{(n+1)})
        - \overline{u}(\mathfrak{t}^{(n+1)})\big\|_{L^2(\mathcal{O})}^2\Big]^{1/2}
        \lesssim \alpha h_{\mathcal{D}} \|\overline{u}\|_{L^2(\Omega_T;W_s)}, \\
         & \mathbb{E}_{\mathfrak{T}}\Big[\sum_n \kappa^{(n+1)}
        \big\|\nabla_{\mathcal{D}} \mathcal{P}_{\mathcal{D}}\overline{u}(\mathfrak{t}^{(n+1)})
        - \nabla\overline{u}(\mathfrak{t}^{(n+1)})\big\|_{L^p(\mathcal{O})}^p\Big]^{1/p}
        \lesssim h_{\mathcal{D}} \|\overline{u}\|_{L^p(\Omega_T;W_s)} .
    \end{align*}
    Then the results follow by triangle inequality.
\end{proof}
\section{Proofs for Theorem \ref{thm:cv 2}}\label{sec:proof-transport cv compactness}
In this proof, we follow the same strategy as the proof for Theorem \ref{thm:cv 1}.

\noindent\textbf{Step 1} \textit{Apriori} estimates and existence of solutions.
For the sake of legibility, in this step, we drop the index \(m\) and write \(\mathcal{D}\) instead of \(\mathcal{D}_m\).

\begin{lemma}
    Suppose that the flux \(\bm{a}\) satisfies Assumption \ref{def:assumption opertator a}.
    There exists at least one \(\mathbb{F}_{\mathcal{D}}\)-adapted \(u_{\mathcal{D}}\) that satisfies the gradient scheme \ref{model:GS transport noise}.
\end{lemma}
\begin{proof}
    Endow \( X_{\mathcal{D}}\) with product \((\cdot, \cdot)\) and denote \(|\cdot|\) the corresponding norm.
    Let \(F^{(n+1)}\) be a mapping that maps \(X_{\mathcal{D}}\)-valued random variable to \(X_{\mathcal{D}}\)-valued random variable and such that,
    for all \(X_{\mathcal{D}}\)-valued random variable \(v_{\mathcal{D}}\),
    for all \(\omega \in \Omega\),
    for all \(\phi_{\mathcal{D}} \in X_{\mathcal{D}}\),
    \begin{equation*}
        (F^{(n+1)}(v_{\mathcal{D}}), \phi_{\mathcal{D}})
        =
        \langle \Pi_{\mathcal{D}}v_{\mathcal{D}}, \Pi_{\mathcal{D}}\phi_{\mathcal{D}} \rangle
        +
        k_{\mathcal{D}} \langle \bm{a}(x, \nabla_{\mathcal{D}} v_{\mathcal{D}}) , \nabla_{\mathcal{D}} \phi_{\mathcal{D}} \rangle
        +
        \frac{k_{\mathcal{D}}}{2} \langle (\bm{\sigma}\cdot \nabla_{\mathcal{D}}) v_{\mathcal{D}} , (\bm{\sigma}\cdot \nabla_{\mathcal{D}})\phi_{\mathcal{D}} \rangle.
    \end{equation*}
    Set \(X_{\mathcal{D}}\)-valued random variable \(b^{(n+1)}\) such that, for all \(\phi_{\mathcal{D}}\in X_{\mathcal{D}}\),
    \begin{equation*}
        (b^{(n+1)},\phi_{\mathcal{D}}) = \langle \Pi_{\mathcal{D}}u_{\mathcal{D}}^{(n)}, \Pi_{\mathcal{D}}\phi_{\mathcal{D}} \rangle +  \langle (\bm{\sigma} \cdot \nabla_{\mathcal{D}})u_{\mathcal{D}}^{(n) } \Delta^{(n+1)}\beta, \Pi_{\mathcal{D}}\phi_{\mathcal{D}} \rangle.
    \end{equation*}
    By the coercivity property of \(\bm{a}\) and the equivalence of norms in finite dimensional space,
    there exists a constant \(C\) such that
    \begin{align*}
        (F^{(n+1)}(v_{\mathcal{D}}),  v_{\mathcal{D}})
        = \,    &
        \|\Pi_{\mathcal{D}}v_{\mathcal{D}}\|_{L^2(\mathcal{O})}^2
        +
        k_{\mathcal{D}} \langle \bm{a}(x, \nabla_{\mathcal{D}} v_{\mathcal{D}}) , \nabla_{\mathcal{D}}v_{\mathcal{D}} \rangle
        +
        \frac{k_{\mathcal{D}}}{2} \|(\bm{\sigma}\cdot \nabla_{\mathcal{D}}) v_{\mathcal{D}} \|_{L^2(\mathcal{O})}^2 \\
        \geq \, &
        \underline{a} k_{\mathcal{D}}\| \nabla_{\mathcal{D}} v_{\mathcal{D}} \|_{L^p(\mathcal{O})}^p \geq C |v_{\mathcal{D}} |^p.
    \end{align*}
    This implies that, for every \(\omega \in \Omega\), \((F^{(n+1)}(v_{\mathcal{D}}),  v_{\mathcal{D}} )/ |v_{\mathcal{D}} | \rightarrow \infty\) as \(|v_{\mathcal{D}}| \rightarrow \infty\).
    By \cite[Theorem 3.1]{Deimling1985NonlinearFA}, for every \(\omega \in \Omega\), \(F^{(n+1)}\) is surjective; that is, for any \(b^{(n+1)} \in X_{\mathcal{D}}\),
    there exists an \(u \in X_{\mathcal{D}}\) such that \(F^{(n+1)}(u) = b^{(n+1)}\).
    Hence, the gradient scheme \ref{model:GS transport noise} has at least one solution. The \(\mathbb{F}_{\mathcal{D}}\)-adaptivity property can be established using the same argument in the proof of Lemma \ref{lem:existence of solution lip conti}.
\end{proof}
\begin{lemma}\label{lem:gradient noise apriori estimate}
    Suppose that the flux \(\bm{a}\) satisfies Assumption \ref{def:assumption opertator a}.
    There exists a constant \(C\) depending only on \(T\), \(\bm{\sigma}\), \( u_{\mathcal{D}}^{(0)}\)
    such that the solution \(u\) to the gradient scheme \ref{model:GS transport noise} satisfies
    \begin{equation}\label{eq-lem:gradient noise apriori estimate 1}
        \mathbb{E}\bigg[\max_{1\leq n \leq N} \|\Pi_{\mathcal{D}} u_{\mathcal{D}}^{(n)}\|_{L^2(\mathcal{O})}^2  + \|\nabla_{\mathcal{D}} u_{\mathcal{D}} \|_{L^p((0,T; L^p(\mathcal{O}))}^p \bigg]
        \leq C.
    \end{equation}
\end{lemma}
\begin{proof}
    Let \(\phi = u_{\mathcal{D}}^{(n+1)}\) in \eqref{model:GS transport noise}, and use the relation \eqref{eq:relation 1} to write
    \begin{equation}\label{eq-proof-lem:gradient noise apriori estimate 1}
        \begin{aligned}
              & \frac{1}{2}\| \Pi_{\mathcal{D}}u_{\mathcal{D}}^{(n+1)}\|_{L^2(\mathcal{O})}^2
            -
            \frac{1}{2} \| \Pi_{\mathcal{D}}u_{\mathcal{D}}^{(n)}\|_{L^2(\mathcal{O})}^2
            +
            \frac{1}{2}\| \Pi_{\mathcal{D}}u_{\mathcal{D}}^{(n+1)}
            -
            \Pi_{\mathcal{D}}u_{\mathcal{D}}^{(n)}\|_{L^2(\mathcal{O})}^2                     \\
              & \, +
            k_{\mathcal{D}} \langle \bm{a} (x, \nabla_{\mathcal{D}} u_{\mathcal{D}}^{(n+1)}), \nabla_{\mathcal{D}}u_{\mathcal{D}}^{(n+1)} \rangle
            +
            \frac{1}{2}k_{\mathcal{D}} \| (\bm{\sigma} \cdot \nabla_{\mathcal{D}}) u_{\mathcal{D}}^{(n+1)} \|_{L^2(\mathcal{O})}^2
            \\
            = & \,
            \langle  (\bm{\sigma} \cdot \nabla_{\mathcal{D}}) u_{\mathcal{D}}^{(n)} \Delta^{(n+1)}\beta, \delta_{\mathcal{D}}^{(n+1)} u_{\mathcal{D}} \rangle
            +
            \langle  (\bm{\sigma} \cdot \nabla_{\mathcal{D}}) u_{\mathcal{D}}^{(n)} \Delta^{(n + 1)} \beta, \Pi_{\mathcal{D}} u_{\mathcal{D}}^{(n)}\rangle.
        \end{aligned}
    \end{equation}
    Taking the sum from \(n = 0\) to \(n = \ell \in \{0, \ldots, N - 1\}\) and the  expectation,
    invoking the coercivity property of \(\bm{a}\), we have
    \begin{align*}
             & \mathbb{E} \big[\| \Pi_{\mathcal{D}}u_{\mathcal{D}}^{(\ell+1)}\|_{L^2(\mathcal{O})}^2 \big]
        +
        \mathbb{E} \bigg[\sum_{n=0}^{\ell}\| \delta_{\mathcal{D}}^{(n + 1)}u_{\mathcal{D}} \|_{L^2(\mathcal{O})}^2 \bigg]                                                                             \\
             & +
        2\mathbb{E} \bigg[\sum_{n=0}^{\ell} k_{\mathcal{D}} \|\nabla_{\mathcal{D}}u_{\mathcal{D}}^{(n+1)}\|_{L^p(\mathcal{O})}^p\bigg]
        +
        \mathbb{E} \bigg[\sum_{n=0}^{\ell} k_{\mathcal{D}} \| (\bm{\sigma} \cdot \nabla_{\mathcal{D}}) u_{\mathcal{D}}^{(n+1)} \|_{L^2(\mathcal{O})}^2 \bigg]                                         \\
        \leq & \,
        \mathbb{E} \big[\| \Pi_{\mathcal{D}}u_{\mathcal{D}}^{(0)}\|_{L^2(\mathcal{O})}^2 \big]
        +
        2\mathbb{E}\bigg[  \sum_{n=0}^{\ell} \langle  (\bm{\sigma} \cdot \nabla_{\mathcal{D}}) u_{\mathcal{D}}^{(n)} \Delta^{(n+1)}\beta, \delta_{\mathcal{D}}^{(n+1)} u_{\mathcal{D}} \rangle \bigg] \\
             & +
        2\mathbb{E}\bigg[ \sum_{n=0}^{\ell} \langle  (\bm{\sigma} \cdot \nabla_{\mathcal{D}}) u_{\mathcal{D}}^{(n)} \Delta^{(n + 1)} \beta, \Pi_{\mathcal{D}} u_{\mathcal{D}}^{(n)}\rangle \bigg].
    \end{align*}
    Consider the second term on the right-hand side.
    By the Cauchy-Schwarz inequality, the Young inequality, and the Itô isometry,
    \begin{align*}
             & \mathbb{E}\bigg[  \sum_{n=0}^{\ell} \langle  (\bm{\sigma} \cdot \nabla_{\mathcal{D}}) u_{\mathcal{D}}^{(n)} \Delta^{(n+1)}\beta, \delta_{\mathcal{D}}^{(n+1)} u_{\mathcal{D}} \rangle \bigg] \\
        \leq & \,
        \frac{1}{2}\mathbb{E}\bigg[  \sum_{n=0}^{\ell} \| (\bm{\sigma} \cdot \nabla_{\mathcal{D}}) u_{\mathcal{D}}^{(n)} \Delta^{(n+1)}\beta \|_{L^2(\mathcal{O})}^2 \bigg]
        +
        \frac{1}{2}\mathbb{E}\bigg[  \sum_{n=0}^{\ell}  \|\delta_{\mathcal{D}}^{(n+1)} u_{\mathcal{D}}\|_{L^2(\mathcal{O})}^2\bigg]                                                                         \\
        =    & \,
        \frac{1}{2}\mathbb{E}\bigg[  \sum_{n=0}^{\ell} k_{\mathcal{D}} \| (\bm{\sigma} \cdot \nabla_{\mathcal{D}}) u_{\mathcal{D}}^{(n)} \|_{L^2(\mathcal{O})}^2 \bigg]
        +
        \frac{1}{2}\mathbb{E}\bigg[  \sum_{n=0}^{\ell}  \|\delta_{\mathcal{D}}^{(n+1)} u_{\mathcal{D}}\|_{L^2(\mathcal{O})}^2\bigg].
    \end{align*}
    The thrid term on the right-hand side vanishes since \(u_{\mathcal{D}}^{(n)}\) is \(\mathcal{F}_{t^{(n)}}\)-measurable and thus is independent of \(\Delta^{(n+1)}\beta\); hence the discrete stochastic integral is a martingale.

    Therefore, we arrive at
    \begin{align*}
             & \mathbb{E} \big[\| \Pi_{\mathcal{D}}u_{\mathcal{D}}^{(\ell+1)}\|_{L^2(\mathcal{O})}^2 \big]
        +
        \mathbb{E} \big[ k_{\mathcal{D}} \| (\bm{\sigma} \cdot \nabla_{\mathcal{D}}) u_{\mathcal{D}}^{(\ell + 1)} \|_{L^2(\mathcal{O})}^2 \big]
        +
        2\mathbb{E} \bigg[\sum_{n=0}^{\ell} k_{\mathcal{D}} \|\nabla_{\mathcal{D}}u_{\mathcal{D}}^{(n+1)}\|_{L^p(\mathcal{O})}^p\bigg] \\
        \leq & \,
        \mathbb{E}\big[ k_{\mathcal{D}} \| (\bm{\sigma} \cdot \nabla_{\mathcal{D}}) u_{\mathcal{D}}^{(0)} \|_{L^2(\mathcal{O})}^2 \big]
        +
        \mathbb{E} \big[\| \Pi_{\mathcal{D}}u_{\mathcal{D}}^{(0)}\|_{L^2(\mathcal{O})}^2 \big].
    \end{align*}

    Next,
    summing \eqref{eq-proof-lem:gradient noise apriori estimate 1} from \(n = 0\) to \(\ell \),
    taking the maximum over \(0 \leq \ell \leq N-1\),
    taking the expectation,
    and by the coercivity property of \(\bm{a}\),
    we obtain the following
    \begin{align*}
                & \mathbb{E}\Big[\max_{1\leq n \leq N} \|\Pi_{\mathcal{D}}u_{\mathcal{D}}^{(n)}\|_{L^2(\mathcal{O})}^2\Big]
        +
        \mathbb{E} \bigg[\sum_{n = 0}^{N-1}\| \delta_{\mathcal{D}}^{(n+1)}u_{\mathcal{D}} \|_{L^2(\mathcal{O})}^2 \bigg]
        +
        \mathbb{E} \bigg[\sum_{n=0}^{N-1} k_{\mathcal{D}} \| (\bm{\sigma} \cdot \nabla_{\mathcal{D}}) u_{\mathcal{D}}^{(n+1)} \|_{L^2(\mathcal{O})}^2 \bigg] \\
        \leq \, &
        \mathbb{E} \big[\| \Pi_{\mathcal{D}}u_{\mathcal{D}}^{(0)}\|_{L^2(\mathcal{O})}^2 \big]
        + T_1 + T_2,
    \end{align*}
    where
    \begin{align*}
        T_1 := &
        2\mathbb{E}\bigg[ \max_{0 \leq \ell \leq N - 1} \sum_{n=0}^{\ell} \langle  (\bm{\sigma} \cdot \nabla_{\mathcal{D}}) u_{\mathcal{D}}^{(n)} \Delta^{(n+1)}\beta, \delta_{\mathcal{D}}^{(n+1)} u_{\mathcal{D}} \rangle \bigg], \\
        T_2 := &
        2\mathbb{E}\bigg[ \max_{0 \leq \ell \leq N - 1}  \sum_{n=0}^{\ell} \langle  (\bm{\sigma} \cdot \nabla_{\mathcal{D}}) u_{\mathcal{D}}^{(n)} \Delta^{(n + 1)} \beta, \Pi_{\mathcal{D}} u_{\mathcal{D}}^{(n)}\rangle \bigg].
    \end{align*}
    Consider the term \(T_1\), by the Cauchy-Schwarz inequality and the Young inequality,
    \begin{align*}
        T_1 \leq \, & \mathbb{E} \bigg[ \sum_{n=0}^{N - 1} k_{\mathcal{D}}\| (\bm{\sigma}\cdot \nabla_{\mathcal{D}}) \, u_{\mathcal{D}}^{(n)}\|_{L^2(\mathcal{O})}^2  \bigg]
        + \mathbb{E} \bigg[ \sum_{n=0}^{N - 1}\| \delta_{\mathcal{D}}^{(n + 1)} u_{\mathcal{D}} \|_{L^2(\mathcal{O})}^2\bigg]
    \end{align*}
    The term \(T_2\) can be treated as a stochastic integral of a piecewise constant integrand and use the Burkholder-Davis-Gundy inequality (see,  e.g.  \cite[Theorem 6.1.2]{liu2015stochastic}),
    then there exists a constance \(C_B\) such that
    \begin{align*}
        T_2 \leq \, & C_B \mathbb{E}\bigg[ \bigg( \sum_{n = 0}^{N-1} k_{\mathcal{D}} \|(\bm{\sigma} \cdot \nabla_{\mathcal{D}}) \, u_{\mathcal{D}}^{(n)}\|_{L^2(\mathcal{O})}^2 \|\Pi_{\mathcal{D}} u_{\mathcal{D}}^{(n)}\|_{L^2(\mathcal{O})}^2 \bigg)^{1/2} \bigg]                                                                                    \\
        \leq \,     & C_B \mathbb{E}\bigg[\max_{0 \leq n \leq N - 1} \|\Pi_{\mathcal{D}} u_{\mathcal{D}}^{(n)}\|_{L^2(\mathcal{O})} \bigg(\|\bm{\sigma}\|_{L^{\infty}(\mathcal{O})}^2 \sum_{n = 1}^{N-1} k_{\mathcal{D}} \|\nabla_{\mathcal{D}} u_{\mathcal{D}}^{(n)}\|_{L^2(\mathcal{O})}^2 \bigg)^{1/2}\bigg]                                         \\
        \leq \,     & \frac{1}{4} \mathbb{E}\Big[\max_{0 \leq n \leq N - 1} \|\Pi_{\mathcal{D}} u_{\mathcal{D}}^{(n)}\|_{L^2(\mathcal{O})}^2 \Big] + C_B^2 \|\bm{\sigma}\|_{L^{\infty}(\mathcal{O})}^2 \mathbb{E}\bigg[ \sum_{n = 0}^{N-1} k_{\mathcal{D}} \|\nabla_{\mathcal{D}} u_{\mathcal{D}}^{(n)} \|_{L^2(\mathcal{O})}^2 \bigg]                  \\
        \leq \,     & \frac{1}{4} \mathbb{E}\Big[\max_{0 \leq n \leq N - 1} \|\Pi_{\mathcal{D}} u_{\mathcal{D}}^{(n)}\|_{L^2(\mathcal{O})}^2 \Big] + C_B^2 \|\bm{\sigma}\|_{L^{\infty}(\mathcal{O})}^2 T^{1-2/p} \mathbb{E}\bigg[ \sum_{n = 0}^{N-1} k_{\mathcal{D}} \|\nabla_{\mathcal{D}} u_{\mathcal{D}}^{(n)} \|_{L^p(\mathcal{O})}^p \bigg]^{2/p}.
    \end{align*}
    Combining these estimates and cancelling our the matching terms, we have
    \begin{align*}
        \mathbb{E}\Big[\max_{1\leq n \leq N} \|\Pi_{\mathcal{D}}u_{\mathcal{D}}^{(n)}\|_{L^2(\mathcal{O})}^2\Big] 
        +
        \mathbb{E} \big[ k_{\mathcal{D}} \| (\bm{\sigma} \cdot \nabla_{\mathcal{D}}) u_{\mathcal{D}}^{(N)} \|_{L^2(\mathcal{O})}^2 \big]
        \leq 
        C_1 + \frac{1}{4} \mathbb{E}\Big[\max_{0 \leq n \leq N - 1} \|\Pi_{\mathcal{D}} u_{\mathcal{D}}^{(n)}\|_{L^2(\mathcal{O})}^2 \Big].
    \end{align*}
    where 
    \begin{align*}
        C_1 := \, &
        \mathbb{E} \big[\| \Pi_{\mathcal{D}}u_{\mathcal{D}}^{(0)}\|_{L^2(\mathcal{O})}^2 \big]
        +
        \mathbb{E} [ k_{\mathcal{D}} \| (\bm{\sigma} \cdot \nabla_{\mathcal{D}}) u_{\mathcal{D}}^{(0)} \|_{L^2(\mathcal{O})}^2]\\
        & +
        C_B^2 \|\bm{\sigma}\|_{L^{\infty}(\mathcal{O})}^2 T^{1-2/p} \mathbb{E}\bigg[ \sum_{n = 0}^{N-1} k_{\mathcal{D}} \|\nabla_{\mathcal{D}} u_{\mathcal{D}}^{(n)} \|_{L^p(\mathcal{O})}^p \bigg]^{2/p}
        < \infty.
    \end{align*}
    Therefore, 
    we can find a constant \( C(T, \bm{\sigma},  u_{\mathcal{D}}^{(0)})\) such that,
    \begin{align*}
        \mathbb{E}\Big[\max_{1\leq n \leq N} \|\Pi_{\mathcal{D}} u_{\mathcal{D}}^{(n)}\|_{L^2(\mathcal{O})}^2\Big]
        \leq \, & C(T, \bm{\sigma},  u_{\mathcal{D}}^{(0)}). \qedhere
    \end{align*}
\end{proof}

\noindent\textbf{Step 2} Extraction of convergent subsequences.
\begin{lemma}\label{lem:grad noise cv1}
    Suppose that Assumption \ref{def:assumption opertator a} holds.
    Let \((\mathcal{D}_m)_{m\in \mathbb{N}}\) be a consistent, limit-conforming sequence, and coercive of gradient discretisation.
    Let \(X_{\mathcal{D}}\)-valued process \(u_{\mathcal{D}_m}\) be a solution to the gradient scheme \ref{model:GS transport noise}  with \(\mathcal{D} \equiv \mathcal{D}_m\).
    Up to a subsequence, as \(m \rightarrow \infty\), there exists \(dt \otimes \mathbb{P}\)-progressively measurable \(\overline{u} \in L^2(\Omega_T; L^2(\mathcal{O})) \cap L^{p}(\Omega_T; W_0^{1,p}(\mathcal{O}))\) with \(p \in [2, \infty)\) such that
    \begin{subequations}
        \begin{align}
             & \Pi_{\mathcal{D}_m} u_{\mathcal{D}_m}   \rightarrow \overline{u}  \text{ weakly in \(L^2(\Omega_T; L^2(\mathcal{O}))\)};\label{eq-lem:grad noise cv1 1}                                                          \\
             & \nabla_{\mathcal{D}_m} u_{\mathcal{D}_m}   \rightarrow \nabla \overline{u} \text{ weakly  in \(L^p( \Omega_T; L^p(\mathcal{O}))^d\)};\label{eq-lem:grad noise cv1 2}                                             \\
             & (\bm{\sigma} \cdot \nabla_{\mathcal{D}_m}) \, u_{\mathcal{D}_m}   \rightarrow (\bm{\sigma} \cdot \nabla) \, \overline{u} \text{ weakly  in \(L^2( \Omega_T; L^2(\mathcal{O}))\)};\label{eq-lem:grad noise cv1 4}
        \end{align}
    \end{subequations}
    There exists a \(dt \otimes \mathbb{P}\)-progressively measurable \(\bm{A} \in L^{p'}( \Omega_T; L^{p'}(\mathcal{O}))^d\) where \(p' = p/(p-1)\) such that
    \begin{equation}\label{eq-lem:grad noise cv1 3}
        \bm{a}(x, \nabla_{\mathcal{D}_m} u_{\mathcal{D}_m}) \rightarrow \bm{A} \text{ weakly  in \(L^{p'}( \Omega_T; L^{p'}(\mathcal{O}))^d\)}.
    \end{equation}
\end{lemma}
\begin{proof}
    These convergences follow from the same argument in the proof of Lemma \ref{lem:continuous noise cv1} except for \eqref{eq-lem:grad noise cv1 4}.
    Take \(\Phi \in L^2(\Omega_T; L^2(\mathcal{O}))\), by the convergence \eqref{eq-lem:grad noise cv1 2}, as \(m \rightarrow \infty\),
    \begin{equation*}
        \mathbb{E}\int_0^T \langle (\bm{\sigma} \cdot \nabla_{\mathcal{D}_m}) \, u_{\mathcal{D}_m}, \Phi \rangle
        =
        \mathbb{E}\int_0^T \langle \nabla_{\mathcal{D}_m} u_{\mathcal{D}_m}, \bm{\sigma} \Phi \rangle
        \rightarrow
        \mathbb{E}\int_0^T \langle \nabla \overline{u}, \bm{\sigma} \Phi \rangle
        =
        \mathbb{E}\int_0^T \langle (\bm{\sigma} \cdot \nabla) \, \overline{u}, \Phi \rangle,
    \end{equation*}
    which completes the proof.
\end{proof}
\begin{lemma}\label{lem:grad noise cv2}
    Let \((\mathcal{D}_m)_{m\in \mathbb{N}}\) be a limit-conforming and consistent sequence of gradient discretisation.
    Let \(\overline{u}\) and \(\bm{A}\) be the limits obtained in Lemma \ref{lem:continuous noise cv1}.
    Then,
    \(\overline{u}\) is a continuous \(L^2(\mathcal{O})\)-valued process and
    for any \(t \in [0,T]\),
    \begin{equation}\label{eq-lem:grad noise cv2 2}
        \begin{aligned}
            \mathbb{E} \big[\| \overline{u}(t)\|_{L^2(\mathcal{O})}^2 \big]
            +
            \mathbb{E} \bigg[\int_0^t  2\langle \bm{A}, \nabla \overline{u} \rangle \, ds \bigg]
            = \, & \mathbb{E} \big[ \| \overline{u}_0\|_{L^2(\mathcal{O})}^2  \big].
        \end{aligned}
    \end{equation}
    And along a subsequence as \(m \rightarrow \infty\),
    \begin{equation}\label{lem:grad noise cv2 2}
        \Pi_{\mathcal{D}_m} u_{\mathcal{D}_m} (T)
        \rightarrow
        \overline{u}(T)
        \text{ weakly in \(L^2(\Omega; L^2(\mathcal{O}))\)}.
    \end{equation}
\end{lemma}
\begin{proof}
    Using weak convergences in Lemma \eqref{lem:grad noise cv1}, the following can be obtianed similarly to \eqref{eq-proof-lem:continuous noise cv2 weak formulation}, for any \(\phi \in W_0^{1,p}(\mathcal{O})\),
    \begin{equation}
        \label{eq-proof-lem:grad noise cv2 dual equality}
        \langle\overline{u}(t), \phi \rangle
        + \int_0^t   \langle \bm{A}, \nabla \phi  \rangle \, + \frac{1}{2} \big\langle (\bm{\sigma} \cdot \nabla) \, \overline{u}, (\bm{\sigma} \cdot \nabla) \, \phi \big\rangle ds
        =   \langle  \overline{u}_0, \phi \rangle
        +   \int_0^t \langle \phi , (\bm{\sigma} \cdot \nabla) \, \overline{u} \, d\beta_s \big\rangle
    \end{equation}
    Since \(W_0^{1,p}(\mathcal{O})\) (\(p \in [2, \infty)\)) is continuously and densely embedded in \(L^2(\mathcal{O})\) (see, e.g. \cite[Section 9.4]{brezis2011functional}),
    invoking \cite[Theorem 4.2.5]{liu2015stochastic},
    then there is a \(L^2(\mathcal{O})\)-valued continuous version of \(\overline{u}\) such that,
    for any \(t \in [0,T]\), \eqref{eq-proof-lem:grad noise cv2 dual equality} holds \(\mathbb{P}\)-almost surely,
    and
    \begin{align*}
          & \| \overline{u}(t)\|_{L^2(\mathcal{O})}^2
        +
        \int_0^t 2 \langle \bm{A}, \nabla \hat{u}\rangle
        +
        \langle (\bm{\sigma} \cdot \nabla) \overline{u}, (\bm{\sigma} \cdot \nabla) \hat{u} \rangle \, ds \\
        = & \,
        \| \overline{u}(0)\|_{L^2(\mathcal{O})}^2
        +
        \int_0^t \| (\sigma \cdot \nabla) \overline{u}\|_{L^2(\mathcal{O})}^2 ds
        +
        \int_0^t \langle \overline{u}, (\bm{\sigma} \cdot \nabla) \overline{u} \, d\beta_S\rangle,
    \end{align*}
    where \(\hat{u}\) is any \(W_0^{1,p}(\mathcal{O})\)-valued pregressively measurable \(dt \otimes \mathbb{P}\)-version of \(\overline{u}\);
    then taking expectation and cancelling out the Itô correction term with the second term on the right-hand side yield \eqref{eq-lem:grad noise cv2 2}. And \eqref{lem:grad noise cv2 2} follows similar argument to the proof of \eqref{eq-lem:continuous noise cv2 2 2}.
\end{proof}
\noindent\textbf{Step 3.} Identification of limit operator.
\begin{proof}[Proof of Theorem \ref{thm:cv 2}]
    It remains to identify \(\bm{A}\).
    Summing \eqref{eq-proof-lem:gradient noise apriori estimate 1} over \(n\) from \(1\) to \(N_m-1\),
    taking the expectation, the following can be obtained:
    \begin{align*}
             & \mathbb{E} \big[\| \Pi_{\mathcal{D}_m}u_{\mathcal{D}_m}(T)\|_{L^2(\mathcal{O})}^2 \big]
        +
        \mathbb{E} \Big[2 \int_0^T \langle \bm{a}(x, \nabla_{\mathcal{D}_m} u_{\mathcal{D}_m}), \nabla_{\mathcal{D}_m} u_{\mathcal{D}_m} \rangle \, ds\Big] \\
        \leq & \,
        k_{\mathcal{D}_m} \mathbb{E}\big[\| (\bm{\sigma} \cdot \nabla_{\mathcal{D}_m})u_{\mathcal{D}_m}^{(0)}\|_{L^2(\mathcal{O})}^2\big]
        +
        \mathbb{E}\big[\| \Pi_{\mathcal{D}_m}u_{\mathcal{D}_m}^{(0)}\|_{L^2(\mathcal{O})}^2\big].
    \end{align*}
    Rearranging terms yields
    \begin{align*}\label{eq-proof-thm:cv 2 2}
        -  2 \mathbb{E} \Big[ \int_0^T \langle \bm{a}(x, \nabla_{\mathcal{D}_m} u_{\mathcal{D}_m}), \nabla_{\mathcal{D}_m}u_{\mathcal{D}_m} \rangle \, ds \Big]
        \geq & \,
        \mathbb{E} \big[
        \| \Pi_{\mathcal{D}_m}u_{\mathcal{D}_m}(T)\|_{L^2(\mathcal{O})}^2
        -
        \| \Pi_{\mathcal{D}_m}u_{\mathcal{D}_m}^{(0)}\|_{L^2(\mathcal{O})}^2 \\
             & \, -
        k_{\mathcal{D}_m} \| (\bm{\sigma} \cdot \nabla_{\mathcal{D}_m})u_{\mathcal{D}_m}^{(0)}
        \big].
    \end{align*}
    Combining this and
    \begin{equation*}\label{equation:ID limits functional 2}
        \forall \Phi \in L^p(\Omega_T; L^p(\mathcal{O}))^d, \quad
        0 \geq -\mathbb{E} \int_0^T 2  \big\langle \nabla_{\mathcal{D}_m} u_{\mathcal{D}_m} - \Phi, \bm{a} (x, \nabla_{\mathcal{D}_m} u_{\mathcal{D}_m}) - \bm{a} (x, \Phi) \big\rangle \, ds,
    \end{equation*}
    letting \(m \rightarrow \infty\),
    recalling the weak convergence \eqref{lem:grad noise cv2 2} and lower semicontinuity of norm (see, e.g. \cite[Proposition 3.5 (iii)]{brezis2011functional}),
    combining the resulting inequality with \eqref{eq-lem:grad noise cv2 2},
    leads to
    \begin{equation*}
        0  \leq \mathbb{E} \int_0^T \big\langle \nabla \overline{u} - \Phi, \bm{A} - \bm{a}(x, \Phi) \big\rangle \, ds  , \quad \forall \Phi \in  L^p(\Omega_T;L^p(\mathcal{O}))^d.
    \end{equation*}
    The rest follows from the same argument in Theorem \ref{thm:cv 1}.
\end{proof} 

\printbibliography 

\begin{appendix}
    \section{Auxiliary results}
\begin{lemma}\label{lem:ineq LL op1}
        Suppose that Assumption \ref{def:assumption opertator a 2} (i, ii) hold.
        Then for \(p \leq 2\), for almost every \(x \in \mathcal{O}\) and all \((\mu, \nu) \in \mathbb{R}^d \times \mathbb{R}^d \),
        \begin{equation}\label{eq-lem:ineq LL op1}
            | \bm{a}(x, \bm{\mu}) - \bm{a}(x, \bm{\nu}) | \leq (2\beta_{\bm{a}} + 2^{p-1} \alpha_{\bm{a}} + \alpha_{\bm{a}} )|\bm{\mu} -  \bm{\nu}|^{p-1}.
        \end{equation}
    \end{lemma}
    \begin{lemma}\label{lem:ineq LL op2}
        Suppose that Assumption \ref{def:assumption opertator a 2} (iii) holds.
        For almost every \(x\in \mathcal{O}\) and all \(( \bm{\mu}, \bm{\nu} ) \in \mathbb{R}^d \times \mathbb{R}^d\),
        \begin{itemize}[leftmargin=*]
            \item If \(p < 2\),
                  \begin{equation}\label{eq-lem:ineq LL op2 1}
                      |\bm{\mu} - \bm{\nu}|^p \leq \gamma_{\bm{a}}^{-p/2} 2^{(p-1)(2-p)/2} \big( [\bm{a}(x, \bm{\mu}) - \bm{a}(x, \bm{\nu})] \cdot [\bm{\mu} - \bm{\nu}]\big)^{p/2} \big(|\bm{\mu}|^p + |\bm{\nu}|^p \big)^{(2-p)/2};
                  \end{equation}
            \item If \(p \geq 2\),
                  \begin{equation}\label{eq-lem:ineq LL op2 2}
                      |\bm{\mu} - \bm{\nu}|^p \leq \gamma_{\bm{a}}^{-1} [\bm{a}(x, \bm{\mu}) - \bm{a}(x, \bm{\nu})] \cdot [\bm{\mu} - \bm{\nu}].
                  \end{equation}
        \end{itemize}
    \end{lemma}

    \begin{lemma}\label{lem:strong cv piecewise approximation}
        Let \(q \in [1, \infty)\), let \(E\) be a Banach space, and let \(f \in L^{q}(0,T; E)\).
        For each \(m \in \mathbb{N}\), let \(0 = t^{(0)} < \cdots < t^{(N_m)} = T\) be a uniform partition
        with time step \(k_{m} = T/N_m\), and set, for \(n = 0, 1, \ldots, N_m - 1\),
        \begin{equation*}
          F_m^{(n)} := \frac{1}{k_{m}}\int_{t^{(n)}}^{t^{(n+1)}} f(s) \, ds
          \quad\text{and}\quad
          F_m(t) :=
          \begin{cases}
            F_m^{(0)} & \text{if \(t = 0\)}, \\
            F_m^{(n)} & \text{if \(t\in (t^{(n-1)}, t^{(n)}]\), \(1 \leq n\leq N_m - 1\),} \\
            0         & \text{if \(t \in (t^{(N_m-1)}, T]\)}.
          \end{cases}
        \end{equation*}
        As \(m \rightarrow \infty\),
        if \(k_{m} \rightarrow 0\),
        then \(F_m \rightarrow f\) strongly in \(L^{q}(0,T; E)\).

        In particular, for \(1 \leq p < \infty\), 
        if \(f \in L^{q}(\Omega_T; L^p(\mathcal{O}))\), 
        then \(F_m\) converges to \(f\) strongly in \(L^{q}(\Omega_T; L^p(\mathcal{O}))\).
      \end{lemma}
      
      \begin{proof}
        Let \(\tilde f \in L^{q}(\mathbb{R}; E)\) denote the zero extension of \(f\) outside \((0,T)\). 
        Let \(t \in (t^{(\ell-1)}, t^{(\ell)}]\) for \(\ell \in \{1, \ldots, N_m - 1\}\),
        then 
        \[
          F_m(t) - f(t) = \frac{1}{k_m}\int_{t^{(\ell)}}^{t^{(\ell+1)}}  f(s) - f(t) \, ds.
        \]
        Since the mapping \(v \mapsto \|v\|_{E}^{q}\) is convex and continuous for \(q \geq 1\),
        by 
        the Jensen inequality 
        and 
        the change of variable (\(s = t + \tau\) and \(s - t \in [0, 2k_m]\)), 
        \begin{equation*}
          \| F_m(t) - f(t)\|_{E}^{q}
          \leq \frac{1}{k_m}\int_{t^{(\ell)}}^{t^{(\ell+1)}} \| f(s) - f(t) \|_{E}^{q} \, ds
          \leq \frac{1}{k_m}\int_{0}^{2k_m} \| \tilde f(t+\tau) - \tilde f(t) \|_{E}^{q} \, d\tau,
        \end{equation*} 
        Integrating over \((0,T)\) and using Tonelli's theorem yields
        \begin{align*}
          \| F_m - f\|_{L^{q}(0,T; E)}^{q}
          &\leq \frac{1}{k}\int_{0}^{2k_m} \| \tilde f(\cdot + \tau) - \tilde f \|_{L^{q}(\mathbb{R}; E)}^{q} \, d\tau
            + \| f \|_{L^{q}(T - k, T; E)}^{q} \\
          &\leq 2 \sup_{0 \leq \tau \leq 2k_m} \| \tilde f(\cdot + \tau) - \tilde f \|_{L^{q}(\mathbb{R}; E)}^{q}
            + \| f \|_{L^{q}(T - k, T; E)}^{q}.
        \end{align*}
        Letting \(m \rightarrow \infty\), \(k \rightarrow 0\),
        the first term vanishes by the continuity of translations in \(L^{q}(\mathbb{R}; E)\) for \(1 \leq q < \infty\),
        and the second term vanishes by the absolute continuity of the integral.
      \end{proof}
  
    \begin{lemma}\label{lem:stoc auxiliary 1}
        Let \(a\) and \(b\) be \(\mathbb{R}\)-valued integrable functions such that for all \(t \in [0,T]\)
        \begin{subequations}
            \begin{equation}\label{eq-lem:auxiliary 1 1}
                a(t) \leq a(0) + \int_{0}^{t}b(s)\,ds.
            \end{equation}
            Then for all \(\kappa \geq 0\) and for all \(t \in [0,T]\)
            \begin{equation}\label{eq-lem:auxiliary 1 2}
                e^{-\kappa t}a(t) + \kappa \int_0^t e^{-\kappa s} a(s) \, ds \leq a(0) + \int_0^t e^{-\kappa s}b(s)\, ds.
            \end{equation}
            The equality in \eqref{eq-lem:auxiliary 1 2} holds when equality in \eqref{eq-lem:auxiliary 1 1} holds.
        \end{subequations}
    \end{lemma}
    \begin{proof}
        From the left-hand side of \eqref{eq-lem:auxiliary 1 2},
        by \eqref{eq-lem:auxiliary 1 1},
        we have
        \begin{align*}
            e^{-\kappa t} a(t) + \kappa \int_0^t e^{- \kappa s} a(s) \, ds
            \leq
            e^{-\kappa t} \bigg[ a(0) + \int_{0}^{t}b(s)\,ds \bigg]
            +
            \kappa \int_0^t  e^{-\kappa s} \bigg[ a(0) + \int_{0}^{s}b(z)\,dz\bigg] \, ds.
        \end{align*}
        Consider the second term on the right-hand side
        \begin{align*}
            \kappa \int_0^t  e^{-\kappa s} \bigg[ a(0) + \int_{0}^{s}b(z)\,dz\bigg] \, ds =  \kappa \int_0^t  e^{-\kappa s}  a(0)\, ds +  \kappa \int_0^t  e^{-\kappa s} \bigg[\int_{0}^{s}b(z)\,dz\bigg] \, ds =: T_1 + T_2.
        \end{align*}
        First, we compute the term \(T_1\),
        \begin{equation*}
            T_1 = a(0) \big[ -e^{-\kappa s}\big]_{s = 0}^t = a(0) \big( 1 - e^{- \kappa t}\big).
        \end{equation*}
        Next, we compute the term \(T_2\), by integration by parts,
        \begin{equation*}
            T_2 = \bigg[ - e^{-\kappa s} \int_{0}^{s}b(z)\,dz \bigg]_{s = 0}^t + \int_0^t e^{-\kappa s} b(s) \, ds
            = -e^{- \kappa t}\int_0^t b(z) \, dz + \int_0^t e^{-\kappa s} b(s) \, ds.
        \end{equation*}
        Collecting these terms, we arrive at \eqref{eq-lem:auxiliary 1 2}.
    \end{proof}
    \begin{lemma}\label{lem:uniform weak cv}
        Let \(q \in (1,\infty)\), \(\kappa \geq 0\),
        and let \(E\) denote a Banach space.
        If \(g_m \rightarrow g\) weakly in \(L^q(\Omega_T; E)\), then for every \(h \in L^{q'}(\Omega_T; E')\),
        \begin{equation*}
            \sup_{t \in [0,T]} \bigg| \mathbb{E}\Big[\int_0^t e^{-\kappa s} \langle g_m - g, h \rangle \, ds\Big] \bigg| \rightarrow 0
            \quad \text{as } m \rightarrow \infty.
        \end{equation*}
    \end{lemma}
    \begin{proof}
        Set \(\Phi(t) := \mathbf{1}_{[0,t]} e^{-\kappa \cdot} h \in L^{q'}(\Omega_T; E')\).
        For \(0 \leq s \leq t \leq T\),
        by continuity of the integral,
        \begin{equation*}
            \|\Phi(t) - \Phi(s)\|_{L^{q'}(\Omega_T; E')}^{q'}
            \leq \mathbb{E}\Big[\int_s^t \|h(r)\|_{E'}^{q'} \, dr\Big]
            \rightarrow 0
            \text{ as \(|t - s| \rightarrow 0\).}
        \end{equation*}
        Hence \(\Phi\) is uniformly continuous on \([0,T]\).
        Since weakly convergent sequence \((g_m)\) is bounded,  \(M := \sup_m \|g_m - g\|_{L^q(\Omega_T; E)} < \infty\).
        Given \(\epsilon > 0\),
        the uniform continuity of \(\Phi\) allows us to choose a partition \(0 = s_0 < s_1 < \cdots < s_J = T\) such that
        \(\|\Phi(t) - \Phi(s_j)\|_{L^{q'}(\Omega_T; E')} \leq \epsilon\) for all \(t \in [s_{j-1}, s_j]\) and all \(j\).
    
        For any \(t \in [0, T]\), the Hölder's inequality then gives
        \begin{equation*}
            \big| \mathbb{E}\big[\langle g_m - g, \Phi(t) \rangle\big] \big|
            \leq M \epsilon + \max_{1 \leq j \leq J} \big| \mathbb{E}\big[\langle g_m - g, \Phi(s_j) \rangle\big] \big|.
        \end{equation*}
        Taking the supremum over \(t\) on the left-hand side, we have
        \begin{equation*}
            \sup_{t\in [0,T]}
            \big| \mathbb{E}\big[\langle g_m - g, \Phi(t) \rangle\big] \big|
            \leq M \epsilon + \max_{1 \leq j \leq J} \big| \mathbb{E}\big[\langle g_m - g, \Phi(s_j) \rangle\big] \big|.
        \end{equation*}
        Now taking the limit superior as \(m \rightarrow \infty\) leads to
        \begin{equation*}
            \limsup_{m \rightarrow \infty} \sup_{t\in [0,T]}
            \big| \mathbb{E}\big[\langle g_m - g, \Phi(t) \rangle\big] \big|
            \leq M \epsilon.
        \end{equation*}
        Since left-hand side does not depend on \(\epsilon\),
        letting \(\epsilon \rightarrow 0\) yields
        \begin{equation*}
            \limsup_{m \rightarrow \infty} \sup_{t\in [0,T]}
            \big| \mathbb{E}\big[\langle g_m - g, \Phi(t) \rangle\big] \big|
            \leq 0,
        \end{equation*}
        which completes the proof.
    \end{proof}
    \begin{lemma}\label{lem:weak cv moving times}
        For every \(t \in [0,T]\) and every sequence \((t_m)_{m} \subset [0,T]\) with \(t_m \rightarrow t\),
        \begin{equation*}
            \Pi_{\mathcal{D}_m} u_{\mathcal{D}_m}(t_m) \rightarrow \overline{u}(t)
            \quad \text{weakly in } L^2(\Omega; L^2(\mathcal{O})).
        \end{equation*}
    \end{lemma}
    \begin{proof}
        Recall \eqref{eq-lem:continuous noise cv2 2 2}, \(\Pi_{\mathcal{D}_m} u_{\mathcal{D}_m}(t) \rightarrow \overline{u}(t)\) weakly in \(L^2(\Omega; L^2(\mathcal{O}))\).
        Moreover, \((\Pi_{\mathcal{D}_m} u_{\mathcal{D}_m}(t_m))_m\) is bounded in \(L^2(\Omega; L^2(\mathcal{O}))\) by \eqref{eq-lem:continuous noise apriori estimate},
        and the span of \(\{\xi \varphi : \xi \in L^\infty(\Omega), \varphi \in C_c^\infty(\mathcal{O})\}\) is dense in \(L^2(\Omega; L^2(\mathcal{O}))\).
        It therefore suffices to show that, for all any \(\xi \in L^{\infty}(\Omega)\) and \(\varphi \in C_c^{\infty}(\mathcal{O})\),
        \begin{equation}\label{eq-proof-lem:weak cv moving times 1}
            \mathbb{E}\big[\xi \langle \Pi_{\mathcal{D}_m} u_{\mathcal{D}_m}(t_m) - \Pi_{\mathcal{D}_m} u_{\mathcal{D}_m}(t), \varphi \rangle\big]
            \rightarrow 0
            \quad \text{as } m \rightarrow \infty.
        \end{equation}
        Let \(n_m, n_m'\) be such that \(\Pi_{\mathcal{D}_m} u_{\mathcal{D}_m}(t_m) = \Pi_{\mathcal{D}_m} u_{\mathcal{D}_m}^{(n_m)}\)
        and \(\Pi_{\mathcal{D}_m} u_{\mathcal{D}_m}(t) = \Pi_{\mathcal{D}_m} u_{\mathcal{D}_m}^{(n_m')}\).
        Without loss of generality, we consider \(n_m' \leq n_m\).
        Then we can write  \[(n_m - n_m') k_{\mathcal{D}_m} \leq |t_m - t| + k_{\mathcal{D}_m}.\]
        Let \(v_m := P_{\mathcal{D}_m} \varphi \in X_{\mathcal{D}_m}\) be the interpolant of \(\varphi\) that realises the minimum of \(S_{\mathcal{D}}(\varphi)\),
        then by consistency of GD, \(\|\Pi_{\mathcal{D}_m} v_m - \varphi\|_{L^2(\mathcal{O})} \rightarrow 0\) and \(\|\nabla_{\mathcal{D}_m} v_m\|_{L^p(\mathcal{O})} \leq C_\varphi\).
        Summing the gradient scheme   \ref{def:GS SPDE with continuous ito noise} with test function \(v_m\) over \(n = n_m', \ldots, n_m - 1\) gives
        \begin{align*}
            \langle \Pi_{\mathcal{D}_m} u_{\mathcal{D}_m}^{(n_m)} - \Pi_{\mathcal{D}_m} u_{\mathcal{D}_m}^{(n_m')}, \Pi_{\mathcal{D}_m} v_m \rangle
            = \, & - \sum_{n = n_m'}^{n_m - 1} k_{\mathcal{D}_m} \big\langle \bm{a}(x, \nabla_{\mathcal{D}_m} u_{\mathcal{D}_m}^{(n+1)}), \nabla_{\mathcal{D}_m} v_m \big\rangle \\
                 & + \sum_{n = n_m'}^{n_m - 1} \big\langle B(\Pi_{\mathcal{D}_m} u_{\mathcal{D}_m}^{(n)}) \Delta^{(n+1)} W, \Pi_{\mathcal{D}_m} v_m \big\rangle.
        \end{align*}
        By the Hölder's inequality and the a priori estimates,
        \begin{align*}
            \mathbb{E}\bigg[\sum_{n = n_m'}^{n_m - 1} k_{\mathcal{D}_m} \big\langle \bm{a}(x, \nabla_{\mathcal{D}_m} u_{\mathcal{D}_m}^{(n+1)}), \nabla_{\mathcal{D}_m} v_m \big\rangle \bigg]
            \leq \,     & C_\varphi (|t_m - t| + k_{\mathcal{D}_m})^{1/p}\, \|\bm{a}(x, \nabla_{\mathcal{D}_m} u_{\mathcal{D}_m})\|_{L^{p'}(\Omega_T; L^{p'}(\mathcal{O}))} \\
            \rightarrow & \,  0 \quad \text{ as \(t_m \rightarrow t\) and \(k_{\mathcal{D}_m} \rightarrow 0\)}
        \end{align*}
        By the Itô isometry and the growth condition of \(B\),
        \begin{align*}
                        & \mathbb{E} \bigg[\sum_{n = n_m'}^{n_m - 1} \big\langle B(\Pi_{\mathcal{D}_m} u_{\mathcal{D}_m}^{(n)}) \Delta^{(n+1)} W, \Pi_{\mathcal{D}_m} v_m \big\rangle\bigg] \\
            \leq \,     &
            \mathbb{E}\bigg[\sum_{n = n_m'}^{n_m - 1} k_{\mathcal{D}_m} \|B(\Pi_{\mathcal{D}_m} u_{\mathcal{D}_m}^{(n)})\|_{L_2^0}^2 \bigg]^{1/2}                                           \\
            \leq \,     & (|t_m - t| + k_{\mathcal{D}_m})^{1/2}
            \bigg(1 + \mathbb{E}\Big[\max_{0 \leq n \leq N_m} \|\Pi_{\mathcal{D}_m} u_{\mathcal{D}_m}^{(n)}\|_{L^2(\mathcal{O})}^2\Big]\bigg)^{1/2}                                         \\
            \rightarrow & \,  0 \quad \text{ as \(t_m \rightarrow t\) and \(k_{\mathcal{D}_m} \rightarrow 0\)}.
        \end{align*}
    
        Next, we write
        \begin{align*}
            \mathbb{E}\big[\langle \Pi_{\mathcal{D}_m} u_{\mathcal{D}_m}(t_m) - \Pi_{\mathcal{D}_m} u_{\mathcal{D}_m}(t), \varphi \rangle\big]
            = \, &
            \mathbb{E}\big[\langle \Pi_{\mathcal{D}_m} u_{\mathcal{D}_m}(t_m) - \Pi_{\mathcal{D}_m} u_{\mathcal{D}_m}(t), \varphi - \Pi_{\mathcal{D}_m}v_m  \rangle\big] \\
                 & +
            \mathbb{E}\big[\langle \Pi_{\mathcal{D}_m} u_{\mathcal{D}_m}(t_m) - \Pi_{\mathcal{D}_m} u_{\mathcal{D}_m}(t), \Pi_{\mathcal{D}_m}v_m  \rangle\big]           \\
            =:   & T_1 + T_2.
        \end{align*}
        By the Cauchy--Schwarz inequality and \eqref{eq-lem:continuous noise apriori estimate},
        \begin{equation*}
            T_1
            \leq 2 \|\varphi - \Pi_{\mathcal{D}_m} v_m\|_{L^2(\mathcal{O})}
            \, \mathbb{E}\Big[\max_{0 \leq n \leq N_m} \|\Pi_{\mathcal{D}_m} u_{\mathcal{D}_m}^{(n)}\|_{L^2(\mathcal{O})}^2\Big]^{1/2}
            \rightarrow 0.
        \end{equation*}
        And \(T_2 \rightarrow 0\) follows from the analysis from above.
        Therefore, 
        \begin{equation*}
            \Big|\mathbb{E}\big[\xi \langle \Pi_{\mathcal{D}_m} u_{\mathcal{D}_m}(t_m) - \Pi_{\mathcal{D}_m} u_{\mathcal{D}_m}(t), \varphi \rangle\big] \Big| 
            \leq 
            \|\xi\|_{L^\infty(\Omega)} \, \mathbb{E}\big[\langle \Pi_{\mathcal{D}_m} u_{\mathcal{D}_m}(t_m) - \Pi_{\mathcal{D}_m} u_{\mathcal{D}_m}(t), \varphi \rangle\big] 
            \rightarrow 0.
        \end{equation*}
        Recalling weak convergence \eqref{eq-lem:continuous noise cv2 2 2} finishes the proof.
    \end{proof}
      
    \subsection{Proof of Property \ref{prop:grids satisfy G}}\label{sec:proof of prop:grids satisfy G}

    \begin{proof}
        We first prove \emph{(a).}
        Since \(\gamma p > 1\) (cf. \cite[Theorem 2.1]{flandoli1995martingale}),
        \[W^{\gamma,p}(0,T;E) \hookrightarrow C^{0,\gamma - 1/p}([0,T];E).\]
        For \eqref{eq:G1}, consider \[\sum_n k_{\mathcal{D}} \|g(t^{(n+1)})\|_{E}^p
            \le T \|g\|_{C([0,T];E)}^p \lesssim T \|g\|_{\mathbb{W}^{\gamma,p}(E)}^p,\] 
        then take \(\mathbb{E}\).
 
        Now consider \eqref{eq:G2}. 
        Let \( I_n := (t^{(n)}, t^{(n+1)})\), 
        for \(s, t \in \overline{I_n}\),
        the following can be obtained from \cite[Lemma 1.1]{garsia1970real},
        \[
            \| g(t) - g(s) \|_{E}
            \leq C(\gamma, p) |t-s|^{\gamma - 1/p} \, [ g]_{W^{\gamma,p}(I_n;E)}.
        \]
        Hence
        \[
            \int_{I_n} \| g(t^{(n+1)}) - g(s) \|_{E}^p \, ds
            \leq C\, k_{\mathcal{D}} \cdot k_{\mathcal{D}}^{\gamma p - 1}\, [g]_{W^{\gamma,p}(I_n;E)}^p
            = C\, k_{\mathcal{D}}^{\gamma p}\, [g]_{W^{\gamma,p}(I_n;E)}^p
            \leq C\, k_{\mathcal{D}}^{\gamma p}\, \| g\|_{W^{\gamma,p}(I_n;E)}^p.
        \]
        Summing over \(n\) and using that the \(I_n\) are pairwise disjoint,
        considering
        \[\sum_n \|g\|_{W^{\gamma,p}(I_n; E)}^p \leq \|g\|_{W^{\gamma,p}(0,T; E)}^p,\]
        then taking \(\mathbb{E}\) finishes the proof.

        Now we prove \emph{(b).}
        Set \(I_n^- := [k_{\mathcal{D}}(n - 1/4), k_{\mathcal{D}}(n+1/4)]\) for \(1 \le n \le N-1\),
        \(I_0^- := \{0\}\), and \(I_N^- := [T -  k_{\mathcal{D}}/4, T]\).

        By \eqref{eq:kappa bounds} and independence of \(\mathfrak{t}^{(n)}\) and \( \mathfrak{t}^{(n+1)}\),
        \begin{equation*}
            \overline{\mathbb{E}}\big[ \kappa^{(n+1)} \|g(\mathfrak{t}^{(n+1)})\|_{E}^p \big]
            \leq \frac{3}{2} k_{\mathcal{D}} \cdot \frac{2}{k_{\mathcal{D}}} \int_{I_{n+1}^-} \|g(\xi)\|_{E}^p \, d\xi
            = 3 \int_{I_{n+1}^-} \|g(\xi)\|_{E}^p \, d\xi .
        \end{equation*}
        Since the intervals \(I_{n+1}^- \subset [0,T]\), \(n = 0, \ldots, N-1\), are pairwise disjoint subsets,
        then summing over \(n\) and taking \(\mathbb{E}\) yields \eqref{eq:G1}.

        Now consider \eqref{eq:G2}. Fix \(1 \le n \le N-2\) and set \(J_n := [k_{\mathcal{D}}(n-1/4), k_{\mathcal{D}}(n+5/4)]\) with length \(3k_{\mathcal{D}}/2\), so that \(I_n^- \cup I_{n+1}^- \subset J_n\).
        Since \(\mathfrak{t}^{(n)}\) and \( \mathfrak{t}^{(n+1)}\) are independent,
        \[
            \overline{\mathbb{E}}\Big[ \int_{\mathfrak{t}^{(n)}}^{\mathfrak{t}^{(n+1)}}
            \|g(\mathfrak{t}^{(n+1)}) - g(s)\|_{E}^p \, ds \Big]
            = \frac{4}{k_{\mathcal{D}}^2}
            \int_{I_n^-} \int_{I_{n+1}^-} \int_{\zeta}^{\xi}
            \|g(\xi) - g(s)\|_{E}^p \, ds \, d\xi \, d\zeta .
        \]
        For \(\zeta \leq s \le \xi\) with \(\zeta \in I_n^-\), \(\xi \in I_{n+1}^-\),
        we have
        \(|\xi - s| \leq  3k_{\mathcal{D}}/2\)
        and
        \begin{equation*}
            \|g(\xi) - g(s)\|_{E}^p
            \leq
            \Big(\frac{3  k_{\mathcal{D}}}{ 2}\Big)^{1+\gamma p}
            \frac{\|g(\xi)-g(s)\|_{E}^p}{ |\xi-s|^{1+\gamma p}}.
        \end{equation*} 
        Then
        \begin{equation*}
            \overline{\mathbb{E}}\Big[ \int_{\mathfrak{t}^{(n)}}^{\mathfrak{t}^{(n+1)}}
            \|g(\mathfrak{t}^{(n+1)}) - g(s)\|_{E}^p \, ds \Big]
            \leq \frac{4}{k_{\mathcal{D}}^2}\cdot \frac{k_{\mathcal{D}}}{2}
            \cdot \Big(\frac{3  k_{\mathcal{D}}}{ 2}\Big)^{1+\gamma p}
            [g]_{W^{\gamma,p}(J_n; E)}^p
            \leq C \, k_{\mathcal{D}}^{\gamma p}\, \|g\|_{W^{\gamma,p}(J_n;E)}^p .
        \end{equation*}
        Note that
        \[
            \sum_{n} \mathbbm{1}_{J_n \times J_n}(\xi,s)
            = \sum_{n} \mathbbm{1}_{J_n}(\xi)\, \mathbbm{1}_{J_n}(s)
            \leq \sum_{n} \mathbbm{1}_{J_n}(\xi)
            \leq 2
            \qquad \text{for all } (\xi,s) \in [0,T]^2 .
        \] 
        By the Fubini theorem, 
        we interchange the sum and the integral,
        \begin{align*}
            \sum_n \|g\|_{W^{\gamma,p}(J_n; E)}^p
            = \,    &
            \int_0^T \sum_{n} \mathbbm{1}_{J_n}(\xi) \, \|g(\xi)\|_{E}^p \, d\xi
            +
            \iint_{[0,T]^2} \Big( \sum_n \mathbbm{1}_{J_n \times J_n}(\xi,s) \Big) \frac{\|g(\xi) - g(s)\|_{E}^p}{|\xi - s|^{1+\gamma p}}\, d\xi \, ds \\
            \leq \, &
            2 \|g\|_{L^p(0,T;E)}^p
            +
            2 \iint_{[0,T]^2} \frac{\|g(\xi) - g(s)\|_{E}^p}{|\xi - s|^{1+\gamma p}} \, d\xi \, ds
            = 2\, \|g\|_{W^{\gamma,p}(0,T; E)}^p .
        \end{align*}
        Taking \(\mathbb{E}\) finishes the proof.
    \end{proof}
\end{appendix}

\end{document}